\documentclass{amsart}

\usepackage{amsfonts}
\usepackage{amssymb}
\usepackage{amsmath}
\usepackage{amsthm}
\usepackage{graphicx}
\usepackage{xy}
\xyoption{v2}
\usepackage{fullpage}

\newtheorem{theorem}{Theorem}
\newtheorem{prop}{Proposition}[subsection]
\newtheorem{lemma}{Lemma}[subsection]

\newtheorem{corr}{Corollary}[subsection]

\theoremstyle{definition}
\newtheorem{defn}{Definition}[subsection]

\def\O{\mathfrak{o}}
\def\k{\bar{k}}
\def\e{\epsilon}
\def\W{\widetilde{W}}
\def\w{\widetilde{w}}
\def\val{\mathop{\mathrm{val}}}
\def\coker{\mathop{\mathrm{coker}}}
\def\s{\sigma}
\def\A{A(\O_F)}
\def\en{\e^{\nu}}
\def\enm{\e^{-\nu}}
\def\emu{\e^{\mu}}
\def\enmu{\e^{-\mu}}
\def\X{X_x(\en)}
\def\Xw{X_x(\e^{w^{-1}\nu})}
\def\phi{\varphi}
\def\bw{\bigwedge^m\!}
\def\f{f_\nu}
\def\ff{\bar{f}_\nu}
\def\Z{\mathbb{Z}}
\def\Umn{U_{m,N}}
\def\U{\bar{U}}
\def\heximages#1#2#3#4#5#6{
  \heximagessqueezedcarefully{1.7em}{-1em}{#1}{#2}{#3}{#4}{#5}{#6}
}
\def\heximagessqueezed#1#2#3#4#5#6#7{
  \heximagessqueezedcarefully{#1}{2em}{#2}{#3}{#4}{#5}{#6}{#7}
}
\def\heximagessqueezedcarefully#1#2#3#4#5#6#7#8{
  \begin{array}{ccc}
    & #3 & \\[#2]
    #5 \hspace{-#1} & & \hspace{-#1} #4 \\[2em]
    #6 \hspace{-#1} & & \hspace{-#1} #7 \\[#2]
    & #8 &
  \end{array}
}

\newenvironment{caselist}
	       {\begin{list}{\underline{Case \arabic{enumi}}:}
		   {\usecounter{enumi}
		     \setlength{\itemindent}{0.5in}
		     \setlength{\leftmargin}{0in}
		     \setlength{\rightmargin}{0in}
	       }}
	       {\end{list}}

\newcommand\subsubsubsection[1]{\vspace{0.5em}\begin{paragraph}{}\noindent \normalfont\large\itshape #1. \end{paragraph}\vspace{0.5em}}

\begin{document}
  \title{On Some Stratifications of Affine Deligne-Lusztig Varieties for $SL_3$}
  \author{Boris Zbarsky}
  \email{bzbarsky@mit.edu}
  \address{20 E Milton Rd, Unit 3
           Brookline, MA 02445}

  \begin{abstract}
  Let $L := \k((\e))$, where $k$ is a finite field with $q$ elements and $\e$
  is an indeterminate, and let $\s$ be the Frobenius automorphism.  Let $G$ be
  a split connected reductive group over the fixed field of $\s$ in $L$, and
  let $I$ be the Iwahori subgroup of $G(L)$ associated to a given Borel
  subgroup of $G$.  Let $\W$ be the extended affine Weyl group of $G$.  Given
  $x \in \W$ and $b \in G(L)$, we have some subgroup of $G(L)$ that acts on the
  affine Deligne-Lusztig variety $X_x(b) = \{ gI \in G(L)/I : g^{-1}b\s(g) \in
  IxI\}$ and hence a representation of this subgroup on the Borel-Moore
  homology of the variety.  This dissertation investigates this representation
  for certain $b$ in the cases when $G = SL_2$ and $G = SL_3$.
  \end{abstract}
  
  \maketitle

  \section{Introduction}
  Let $k$ be a finite field with $q$ elements and let $\k$ be an algebraic
  closure of $k$.  Let $\s:\k\to\k$ be the Frobenius morphism $\s(a) = a^q$.
  Let $L := \k((\e))$, where $\e$ is an indeterminate, and extend $\s$ to $L$
  by setting $\s(\e) = \e$.  Denote the valuation ring $\k[[\e]]$ of $L$ by
  $\O_L$.  Let $F := k((\e))$ and denote $k[[\e]] \subset \O_L$ by $\O_F$.

  Let $G$ be a split connected reductive group over $F$.  Let $A$ be a split
  maximal torus of $G$.  Let $W$ denote the Weyl group of $A$ in $G$ and let
  $\W = W \ltimes X_*(A)$ denote the extended affine Weyl group.  Fix a Borel
  subgroup $B$ containing $A$, so that $B=AU$, with $U$ unipotent, and let $I$
  denote the corresponding Iwahori subgroup of $G(L)$.  Then we have the
  Bruhat decomposition of $G(L)$ into double cosets $IxI$, where $x \in \W$.
  Let $X = G(L)/I$.  Let $U_w = w^{-1}U(L)w$, so that $U_1 = U(L)$.

  If $b \in G(L)$, then the $\s$-conjugacy class of $b$ is $\{g^{-1}b\s(g): g
  \in G(L)\}$.  For every $x \in \W$ we define (following~\cite{Kottwitz-f_nu})
  the affine Deligne-Lusztig variety $X_x(b) = \{gI \in X \,:\,
  g^{-1}b\s(g) \in IxI\}$.  Note that if $b_1$ and $b_2$ are in the same
  $\s$-conjugacy class, with $b_1 = h^{-1}b_2\s(h)$, then the varieties
  $X_x(b_1)$ and $X_x(b_2)$ are isomorphic, with the isomorphism $X_x(b_1) \to
  X_x(b_2)$ given by the translation $gI \to hgI$.

  Consider the subgroup $H < G(L)$ consisting of elements $h$ such
  that $h^{-1}b\s(h) = b$.  Elements of $H$ then act on the variety $X_x(b)$ by
  left-multiplication.  This action induces a representation of $H$ on the
  Borel-Moore homology of $X_x(b)$.

  Our goal is to study this representation when $G = SL_n$, $n = 2,3$ and $b$
  is a diagonal matrix whose nonzero entries have the form $\e^{\nu_i}$, where
  $\nu_i \neq \nu_j$ if $i \neq j$.  We will refer to such a $b$ as $\en$,
  where $\nu = (\nu_1, \nu_2, \ldots, \nu_n)$.  In general, we will use
  $\e^\mu$ to refer to an element of $G(L)$ which has the form
  \begin{equation*}
    \begin{pmatrix}
      \e^{\mu_1} & 0 & \cdots & 0 \\
      0 & \e^{\mu_2} & \cdots & 0 \\
      \vdots& \vdots & \ddots & \vdots \\
      0 & 0 & \cdots & \e^{\mu_n}
    \end{pmatrix}.
  \end{equation*}
  For our choice of $b$, the subgroup $H$ acting on $\X$ is $A(F) = \Z^{n-1}
  \times \A$.  We will show that the subgroup $\A$ acts trivially on the
  Borel-Moore homology of $\X$, and therefore the representation of $A(F)$
  factors through the representation of $\Z^{n-1}$.  This last representation
  is induced by the action which is given by $(i_1, \ldots, i_{n-1}) \cdot gI =
  \e^{(i_1, \ldots, i_{n-1}, -i_1-\ldots-i_{n-1})}gI$, and corresponds to
  permutation of the homology spaces of disjoint closed subsets of $\X$.

  In order to study these representations, we will develop, in
  Section~\ref{sec:DeterminingIwahoriDoubleCosetForElement}, a method that,
  given $g \in G(L)$ and $x \in \W$, gives necessary and sufficient conditions
  for $g$ to be in $IxI$ in terms of the valuations of the determinants of the
  minors of $g$ (including the $1 \times 1$ minors).  We will also develop, in
  Section~\ref{sec:DeterminingUOrbitForElement}, a method that, given $g \in
  G(L)$ and $w \in \W$, produces the element $x \in \W$ such that $g \in
  w^{-1}U_1 w x I$.  Then in
  Sections~\ref{sec:U-orbits}~and~\ref{sec:triviality-results} we will prove
  some general theorems applicable to $SL_n$ and $GL_n$ which we will later use
  for $SL_3$.

  For the case $G = SL_2$, we will show that the representation of $\A$ on
  the Borel-Moore homology of $\X$ is trivial by showing that not only do we
  have a left-multiplication action of $\A$ on $\X$ but that we also have
  a left-multiplication action of the bigger group $A(\O_L)$ on $\X$.  Since
  $A(\O_L)$ is connected, the action of $\A$ on the homology of $\X$ must
  be trivial.  This will be done in Chapter~\ref{chap:SL_2}.

  For the case $G = SL_3$, this approach would work in most cases, but there
  are some situations in which $A(\O_L)$ does not act on $\X$ by
  left-multiplication.  The approach we will take for $SL_3$ will be to
  decompose $\X$ into a union of disjoint closed subsets, each of which is
  preserved by $\A$, to produce a stratification of each of these closed
  subsets into strata preserved by $\A$, and finally to extend the action of
  $\A$ on each stratum to an action of $A(\O_L)$.  We will then argue that this
  means that the representation of $\A$ on the Borel-Moore homology of each of
  the disjoint closed subsets is trivial.  This will be done in
  Chapter~\ref{chap:SL_3}.

  The author would like to thank Robert Kottwitz for his suggestion of a
  research direction and many useful conversations.

  \section{Preliminaries}
  In this chapter we develop some techniques that will be used later.
  Throughout the chapter, $G$ is $SL_n$ or $GL_n$ for any $n \ge 2$.  Let $K =
  G(\O_L)$.

  Let $V = L^n$.  If we fix a basis for $V$, we may view elements of $G(L)$ as
  $n\times n$ matrices with elements in $L$.  Fix the basis for $V$ such that
  the elements of $A(L)$ are diagonal matrices and the elements of $B(L)$ are
  upper triangular; call this basis $\{ v_1, \ldots, v_n \}$.  Then with
  respect to this basis, an element $M$ of $I$ has the form
  \begin{equation}
    \label{eq:Iwahori_form}
    M_{ij} \in 
    \begin{cases} 
      \O_L^\times & \text{if $i=j$} \\ 
      \O_L & \text{if $i < j$} \\
      \e\O_L & \text{if $i > j$}
    \end{cases}.
  \end{equation}
  From this point on, we always work in this basis, and identify elements of
  $G(L)$ with matrices, and elements of $\W$ with their unique representatives
  which are monomial matrices whose nonzero entries have the form $\e^k$, $k$
  an integer.
  
  \subsection{Determining the Iwahori Double-Coset for a Given Element of $G(L)$}
  \label{sec:DeterminingIwahoriDoubleCosetForElement}
  Let $M \in G(L)$ and $x \in \X$.  We want to give necessary and sufficient
  conditions for $M$ to be in $IxI$.  We will show that these conditions can be
  expressed in terms of conditions on the valuations of the determinants of
  minors of $M$.
  
  To prove this, we first reduce the problem to that of looking at valuations
  of elements by considering $\bw V$ for $m = 1, \ldots, n$.  Any ordered
  $m$-tuple of distinct integers from $1$ to $n$ determines a vector in $\bw
  V$---given the $m$-tuple $(i_1, \ldots, i_m)$ we get the vector $v_{i_1}
  \wedge \cdots \wedge v_{i_m}$.  The vectors corresponding to the set of
  increasing $m$-tuples give a basis for $\bw V$.  We pick an order for this
  basis by lexicographically ordering the increasing $m$-tuples, and from this
  point on this is the basis we use for $\bw V$.

  Given any matrix $M \in G(L)$ we define the matrix $\bw M$ by requiring that
  \begin{equation*}
    \left(\bw M\right) (v_{i_1} \wedge \cdots \wedge v_{i_m}) =
    (Mv_{i_1}) \wedge \cdots \wedge (Mv_{i_m}).
  \end{equation*}

  \begin{lemma}
    \label{lemma:Form_lemma}
    If $M$ has the form given in (\ref{eq:Iwahori_form}), then so does $\bw M$ for
    any $m$.
  \end{lemma}
  \begin{proof}
    An entry of $\bw M$ is the determinant of an $m\times m$ minor of M.  To be
    more precise, if we number the increasing $m$-tuples of integers between 1
    and $n$, ordered lexicographically, from 1 to $\binom{n}{m}$ then $(\bw
    M)_{ij}$ is the determinant of the minor of $M$ which consists of entries
    in rows given by the $i$-th $m$-tuple and columns given by the $j$-th
    $m$-tuple.  Since all entries of $M$ are in $\O_L$, it is clear that any
    such determinant will be in $\O_L$.
    
    If $i = j$, then the two $m$-tuples are identical and if we reduce all
    entries in the minor modulo $\e\O_L$ we will get an upper triangular matrix
    with elements of $\k^\times$ on the diagonal.  Its determinant will be an
    element of $\k^\times$.  But that means that the determinant of our
    original minor is in $\O_L^\times$, as desired.

    Finally, if $i > j$ we need to show that the determinant of the minor is in
    $\e\O_L$.  If $m = 1$ this is true because $M$ has the form given in
    (\ref{eq:Iwahori_form}).  Now we proceed by induction on $m$.

    If the first integer in the $i$-th $m$-tuple is greater than the first
    integer in the $j$-th $m$-tuple, then all integers in the $i$-th $m$-tuple
    are greater than the first integer of the $j$-th m-tuple (since we are
    considering increasing $m$-tuples).  In this case the minor we are
    considering has only elements of $\e\O_L$ in its first column, and hence
    its determinant is in $\e\O_L$.

    The other possibility is that the first integer in the $i$-th $m$-tuple is
    equal to the first integer in the $j$-th $m$-tuple.  Now we find the
    determinant of our original minor by expanding around its first column.
    All entries in this column except for the first entry are in $\e\O_L$
    because we are considering increasing $m$-tuples.  The first entry in the
    column is in $\O_L^\times$ but the determinant of the corresponding
    $(m-1)\times(m-1)$ minor is in $\e\O_L$.  Indeed, since $i > j$, the
    $(m-1)$-tuple we get by dropping the first integer in the $i$-th $m$-tuple
    is greater than the $(m-1)$-tuple we get by dropping the first integer in
    the $j$-th $m$-tuple.  By the inductive hypothesis, the determinant of the
    $(m-1)\times(m-1)$ minor we are interested in is in $\e\O_L$.  So the
    determinant of our original $m\times m$ minor is in $\e\O_L$.
  \end{proof}

  Now we attach to any square $n \times n$ matrix $M$ with entries in $L$ a
  triple of integers $(v, d, c)$.  Here $v$ is the minimum of the valuations of
  entries of $M$, $d$ is the minimum of $\{ n + j-i : \val(M_{ij}) = v\}$ (so
  the number of the first diagonal in which an entry of valuation $v$ occurs,
  numbering from bottom left), and $c$ is the minimum of $\{ j :
  \val(M_{(j-d+n),j}) = v\}$ (so the minimum of the column numbers of entries
  on diagonal $d$ that have valuation $v$).

  \begin{lemma}
    \label{lemma:Triple_lemma}
    The triple $(v, d, c)$ attached to a matrix $M$ is invariant under both
    left and right multiplication by matrices that have the form given in
    (\ref{eq:Iwahori_form}).
  \end{lemma}
  \begin{proof}
    First note that the set of matrices that have the given form is a subset of
    $K$, and $v$ is invariant under left and right multiplication by elements
    of $K$.  So we only need to deal with $d$ and $c$.

    Fix an arbitrary matrix $M$ and let $N$ have the form given in
    (\ref{eq:Iwahori_form}). Let $(v, d, c)$ be the triple of integers attached
    to $M$.  Let $r = c - d + n$ be the row number of the entry which is in the
    $c$-th column and on the $d$-th diagonal.  Let the triple of integers
    attached to $MN$ be $(v, x, y)$ and let the triple of integers attached to
    $NM$ be $(v, w, z)$.  We want to prove that $x = w = d$ and $y = z = c$.

    Clearly,
    \begin{align}
      (NM)_{rc} &= \sum_{k = 1}^n N_{rk}M_{kc} \nonumber \\
      \label{eq:NM-decomposition}
      &= \sum_{k = 1}^{r-1}N_{rk}M_{kc} + N_{rr}M_{rc} +
         \sum_{k = r+1}^{n}N_{rk}M_{kc}.
    \end{align}

    But for $k < r$ we have $N_{rk} \in \e\O_L$.  So the valuation of the first
    term in (\ref{eq:NM-decomposition}) is at least $v + 1$.  The valuation of
    the second term is exactly $v$, since $\val(N_{rr}) = 0$ and $\val(M_{rc})
    = v$ by definition of $r$, $c$, and $d$.  Finally, $\val(M_{kc}) > v$ for
    $k > r$, since then $n + c - k < d$.  Since $N \in K$, the valuation of the
    third term is at least $v + 1$.  Thus the valuation of $(NM)_{rc}$ is $v$.
    By a similar calculation, the valuation of $(MN)_{rc}$ is $v$.  Therefore
    $x \le d$ and $w \le d$ (since $d$ is in the sets that $x$ and $w$ are
    minima of).

    Now consider any $i,j$ such that $\val((NM)_{ij}) = v$.  Since $(NM)_{ij} =
    \sum_{k=1}^n N_{ik}M_{kj}$ and since $\val(N_{ik}) > 0$ for $k < i$ while
    $\val(M_{kj}) \ge v$ for all $k,j$ and $N$ is in $K$, we must have
    $\val(M_{kj}) = v$ for some $k \ge i$.  Then we know that $n + j - k \ge d$
    and hence $n + j - i \ge n + j - k \ge d$.  Since $x$ is the minimum of
    such $n + j -i$, this means that $x \ge d$.  We already knew that
    $x \le d$, so we conclude that $x = d$.  By a similar argument applied to
    $(MN)_{ij}$, $w = d$.

    Now that we know that $x = d$, the fact that $\val((NM)_{rc}) = v$ means
    that $y \le c$.  Consider any $i,j$ such that $\val((NM)_{ij}) = v$ and $n
    + j - i = d$.  As before, $(NM)_{ij} = \sum_{k=1}^n N_{ik}M_{kj}$, so
    $\val(M_{kj}) = v$ for some $k \ge i$.  If $k > i$, then $n + j - k < n + j
    - i = d$, which cannot happen by definition of $d$.  So $k = i$ and
    $\val(M_{ij}) = v$.  But $n + j - i = d$, so by definition of $c$ we have
    $j \ge c$.  Since $y$ is the minimum of all such $j$, we must have $y \ge
    c$, hence $y = c$.  By a similar argument applied to $(MN)_{ij}$, $z = c$.
  \end{proof}

  \begin{theorem}
    \label{thm:DeterminingIwahoriDoubleCosetForElement}
    Given an element $M \in G(L)$, the $x \in \W$ such that $M \in IxI$ is
    uniquely determined by the valuations of determinants of all minors of $M$.
  \end{theorem}
  \begin{proof}
    We will explicitly compute the monomial matrix $x$.  Indeed, there are two
    matrices $N_1, N_2 \in I$ such that $N_1 M N_2$ is this monomial matrix.
    Then for any $m$, $\bw N_1 \cdot \bw M \cdot \bw N_2 = \bw x$.  By
    Lemma~\ref{lemma:Form_lemma} and Lemma~\ref{lemma:Triple_lemma} the triples
    of integers associated to $\bw M$ and $\bw x$ are the same.  We can
    explicitly compute these triples for $\bw M$.  So the problem comes down to
    reconstructing $x$ given the triples of integers for $\bw x$, $m = 1,
    \ldots, n$.

    Let the triple of integers for $\bw x$ be $(v_m, d_m, c_m)$, and let $x_m$
    be the minor of $x$ that corresponds to the element in row $c_m - d_m +
    \binom{n}{m}$ and column $c_m$ in $\bw x$.  Since the determinant of any
    $m\times m$ minor of $x$ is either 0 or the product of $m$ of the nonzero
    entries of $x$, we know that $v_m$ is the sum of the $m$ smallest
    valuations of the entries of $x$.
  
    Now the triple for $m=1$ tells us that $x$ has an entry of valuation $v_1$
    in row $c_1 - d_1 + n$ and column $c_1$.  For $m > 1$, the $m$-tuple
    corresponding to $c_m$ is gotten from the $(m-1)$-tuple corresponding to
    $c_{(m-1)}$ by inserting a single integer $j_m$ somewhere.  Similarly, the
    $m$-tuple corresponding to $c_m - d_m + \binom{n}{m}$ is gotten from the
    $(m-1)$-tuple corresponding to $c_{(m-1)} - d_{(m-1)} + \binom{n}{m-1}$ by
    inserting a single integer $i_m$.  Indeed, to go from $x_{m-1}$ to $x_m$ we
    simply find the unique entry of $x$ which satisfies the following
    conditions:
    \begin{enumerate}
      \item Is not already in $x_{m-1}$.
      \item Has minimal valuation amongst entries satisfying condition 1.
      \item Is on the bottom-left-most diagonal amongst entries satisfying
	conditions 1 and 2.
      \item Is in the left-most column amongst entries satisfying
	conditions 1, 2, and 3.
    \end{enumerate}
    Then we let $x_m$ be the unique minor which contains this entry and
    $x_{m-1}$.  Since by assumption $x_{m-1}$ corresponds to the entry in row
    $c_{m-1} - d_{m-1} + \binom{n}{m-1}$ and column $c_{m-1}$, the above
    conditions enforce that $x_m$ corresponds to the entry in row $c_m - d_m +
    \binom{n}{m}$ and column $c_m$.

    Then we know that $x$ has an entry with valuation $v_m - v_{(m-1)}$ in row
    $i_m$ and column $j_m$.  As $m$ runs from 1 to $n$, we fill in all $n$
    nonzero entries of $x$.
  \end{proof}

  As a consequence we know the necessary and sufficient conditions on the
  valuations of determinants of minors of $g$ such that $g \in IxI$ for a given
  $x \in \W$.  In particular, if we find the triple $(v, d, c)$ for $x$, then
  entries of $g$ that are on diagonals further toward the lower-left corner
  than diagonal $d$ must have valuations strictly greater than $v$.  All other
  entries must have valuation at least $v$, and the entry in row $c - d + n$
  and column $c$ must have valuation equal to $v$.  Similar conditions apply
  to the determinants of minors of $g$, in terms of the triples $(v_m, d_m,
  c_m)$  for the matrices $\bw x$.

  \subsection{Determining the $w^{-1}U_1 w$-Orbit for a Given Element of $G(L)$}
  \label{sec:DeterminingUOrbitForElement}
  We fix $w \in W$.  Let $U' = w^{-1}U_1 w$.  Then $G(L)$ is partitioned into
  double cosets $U'xI$, where $x \in \W$.  Given an element $M \in G(L)$, we
  can find the unique $x$ such that $M \in U'xI$ by applying the following
  algorithm:
  \begin{enumerate}
    \item Let $i$ range from $n$ to 1, inclusive, starting at $n$.
    \item \label{item:u-orbit-repeat_from_here} Let $r_i \in \{1, \ldots, n\}$
      be the image of $i \in \{1, \ldots, n\}$ under the permutation in
      $\Sigma_n$ represented by $w^{-1}$.
    \item Find the entry in the $r_i$ row of $M$ which has minimal valuation in
      that row and which is the leftmost entry with this valuation.  Let $c_i$
      be the number of the column in which this entry is found.
    \item Use column operations by elements of $I$ to eliminate all other
      entries in this row.  This is possible because the entry we chose was the
      leftmost entry of minimal valuation in this row.
    \item Use row operations by elements of $U'$ to eliminate all entries in
      $c_i$ which are not in row $r_i$.  This is possible because all the
      entries that we wouldn't be able to eliminate with an element of $U'$
      have already been eliminated at earlier steps (for greater values of
      $i$).
    \item Decrement $i$ by 1 and repeat from
    step~\ref{item:u-orbit-repeat_from_here}.
  \end{enumerate}
  When this procedure is finished, the resulting matrix is a monomial matrix.
  It can be multiplied on the right by an element of $I$ to make all the
  entries be powers of $\e$, at which point we have a representative for an
  element of $\W$.  This is the $x$ we sought.

  The reason this procedure works is that finding $x$ such that $M \in U'xI$ is
  equivalent to finding $x$ such that $wM \in U_1 wxI$.  Looking at rows of $M$
  in the order $r_n, r_{n-1}, \ldots, 1$ corresponds to looking at rows of $wM$
  in the oder $n, n-1, \ldots, 1$.  Since elements of $U_1$ are
  upper-triangular, looking at the rows in this order would clearly let us
  compute $wx$ starting with the $n$-th row and working upward.  Since we want
  to compute $x$, we want to multiply the result by $w^{-1}$, which just
  permutes the rows.  So we are computing the rows of $x$ in the order $r_n,
  r_{n-1}, \ldots, 1$, which is exactly what the above algorithm does.

  \subsection{Sets of the form $\X \cap U_1 wI$}
  \label{sec:U-orbits}
  Fix a coweight $\nu = (\nu_1, \nu_2, \ldots \nu_n)$ which is strictly
  dominant.  That is, $\nu_i > \nu_j$ if $i < j$.  Fix $w \in W$, and let $I' =
  w I w^{-1}$ and $y = w x w^{-1}$.  We will consider the intersection $\X \cap
  U_1 wI$ and the action of $\A$ on this intersection.  We are interested in
  the left-multiplication action, but since $\A \subset I$, the
  left-multiplication action and the action $t \cdot gI = tgt^{-1}I$ are in
  fact the same action, which we will call the ``conjugation action''
  throughout this section.

  First, note that, as discussed in \cite{Kottwitz-f_nu}, left multiplication
  by $w^{-1}$ gives an isomorphism between $\X \cap U_1 wI$ and $\Xw \cap U_w
  I$.  Under this isomorphism, the conjugation action of $\A$ becomes a
  composition of an automorphism of $\A$ (conjugation by $w$) and the
  conjugation action of $\A$.  Therefore, if we show that all elements of $\A$
  act trivially on the homology of $\Xw \cap U_w I$ when they act by
  conjugation on the set, then they all act trivially on the homology of $\X
  \cap U_1 wI$ when they act by conjugation on the set.
  
  Now for any coweight $\mu$ and any $w \in W$ we can define $f_\mu : U_w \to
  U_w$ by $f_\mu(h) = h^{-1}\emu\s(h)\enmu$.  Then
  \begin{equation*}
    X_x(\emu) \cap U_w I = f_\mu^{-1}(IxI\enmu \cap U_w)/(U_w \cap I)
  \end{equation*}
  and setting $\mu = w^{-1}\nu$ we get
  \begin{equation}
    \label{f_u}
    \Xw \cap U_w I = f_{w^{-1}\nu}^{-1}(IxI\e^{-w^{-1}\nu} \cap U_w)/(U_w \cap I).
  \end{equation}
  But if $h \in U_w$ then
  \begin{equation*}
    f_{w^{-1}\nu}(h) = h^{-1} w^{-1} \en w \s(h) w^{-1}
    \enm w = w^{-1} \f(whw^{-1}) w,
  \end{equation*}
  so for any $S \subset U_w$ 
  \begin{equation*}
    f^{-1}_{w^{-1}\nu}(S) = w^{-1}\f^{-1}(wSw^{-1})w
  \end{equation*}
  and in particular, setting $S = IxI\e^{-w^{-1}\nu} \cap U_w$,
  \begin{equation}
    \label{move_w}
    f_{w^{-1}\nu}^{-1}(IxI\e^{-w^{-1}\nu} \cap U_w) =
    w^{-1}\f^{-1}(wIxIw^{-1}\enm
    \cap U_1) w.
  \end{equation}
  Combining (\ref{move_w}) with (\ref{f_u}) we see that
  \begin{equation}
    \Xw \cap U_w I = w^{-1} \bigl( \f^{-1}(I'yI'\enm \cap U_1) / (U_1 \cap
    I') \bigr) w
  \end{equation}
  and therefore
  \begin{equation}
    \label{eq:f_nu_isomorphism}
    \Xw \cap U_w I \cong \f^{-1}(I'yI'\enm \cap U_1) / (U_1 \cap I') 
  \end{equation}
  so that it is sufficient to study the behavior of $\f : U_1 \to U_1$ for
  strictly dominant $\nu$.

  \subsubsection{Behavior of $\f$}
  Let $g \in U_1$, and write $g_{ij}$ for the entry in the $i$-th row and
  $j$-th column of $g$.  Then we have
  \begin{equation}
    \label{eq:inverse}
    (g^{-1})_{ij} = \begin{cases}
      -\sum_{k = i + 1}^j g_{ik}(g^{-1})_{kj} & i < j \\
      1 & i = j \\
      0 & i > j
    \end{cases},
  \end{equation}
  with the cases when $i \ge j$ following from the fact that $g^{-1} \in U_1$.
  This allows computation of any entry of $g^{-1}$, since the lower bound of
  the summation is strictly greater than $i$, so for any given $j$ we simply
  compute the entries of the $j$-th column of $g^{-1}$ starting at the diagonal
  and going up until the entry we want.

  Now, since $g_{lj} = 0$ for $l > j$ and $(g^{-1})_{il} = 0$ for $i > l$,
  \begin{equation}
    \label{eq:def_f(g)_ij}
    \left(\f(g)\right)_{ij} = \sum_{l=i}^j (g^{-1})_{il} \cdot \e^{\nu_l-\nu_j}\s(g_{lj})
  \end{equation}
  and since $(g^{-1})_{ii} = 1$
  \begin{equation*}
    \left(\f(g)\right)_{ij} = \e^{\nu_i-\nu_j}\s(g_{ij}) + \sum_{l=i+1}^j (g^{-1})_{il} \cdot \e^{\nu_l-\nu_j}\s(g_{lj}).
  \end{equation*}
  Substituting the expression from (\ref{eq:inverse}), we get
  \begin{align}
    \label{eq:explicit_f_nu}
    \left(\f(g)\right)_{ij} &= \e^{\nu_i-\nu_j}\s(g_{ij}) +
    \sum_{l=i+1}^j  \left(- \sum_{k=i+1}^l g_{ik}(g^{-1})_{kl} \right) \cdot \e^{\nu_l-\nu_j}\s(g_{lj}) \nonumber \\
    &= \e^{\nu_i-\nu_j}\s(g_{ij}) - 
    \sum_{k=i+1}^j g_{ik} \left(\sum_{l=k}^j (g^{-1})_{kl} \cdot
    \e^{\nu_l-\nu_j}\s(g_{lj})\right). \nonumber \\
    \noalign{\noindent Now by (\ref{eq:def_f(g)_ij}) the summation in
    parentheses is just $\left(\f(g)\right)_{kj}$, and $\left(\f(g)\right)_{jj}
    = 1$, so we see that}
    \left(\f(g)\right)_{ij} &= \e^{\nu_i-\nu_j}\s(g_{ij}) - g_{ij} -
    \sum_{k=i+1}^{j-1} g_{ik} \left(\f(g)\right)_{kj}.
  \end{align}

  \begin{prop}
    \label{prop:f_nu_dependence}
    $\left(\f(g)\right)_{ij}$ only depends on $g_{ij}$ and on $\{ g_{pq} \,|\,
    q-p < j - i \}$---the entries of $g$ that are closer to the main diagonal
    than $(i,j)$.
  \end{prop}
  \begin{proof}
    We take (\ref{eq:explicit_f_nu}) and proceed by induction on $j-i$.  The
    base case $j-i = 1$ follows because the summation drops out of
    (\ref{eq:explicit_f_nu}).  The inductive step follows because for
    all the terms in the summation $j - k < j - i$, since $k > i$, and $k - i <
    j - i$, since $k < j$.
  \end{proof}

  \begin{prop}
    \label{prop:f_nu_bijection}
    $\f$ is a bijection and $\f(U(\O_L)) = U(\O_L)$.
  \end{prop}
  \begin{proof}
    First we prove surjectivity.  Given an $h \in U_1$ we construct $g \in U_1$
    such that $\f(g) = h$.  For a given $i$, we construct the entries $g_{ij}$,
    for $j \ge i$, by induction on $j$.  For the base case $j=i$, we have
    $g_{ij} = 1$.  Now suppose that we already know $g_{ii}, \ldots,
    g_{i,j-1}$.  Then by (\ref{eq:explicit_f_nu}) we must find $g_{ij}$ such
    that
    \begin{equation*}
      \e^{\nu_i-\nu_j}\s(g_{ij}) - g_{ij} = h_{ij} +
      \sum_{k=i+1}^{j-1} g_{ik} h_{kj}.
    \end{equation*}
    But $\nu$ is strictly dominant, so $\nu_i - \nu_j > 0$, and given any $r >
    0$ and any $a \in L$ the equation $\e^{r}\s(y) - y = a$ has a solution in
    $L$.  The coefficients of the solution can be computed explicitly---the
    leading coefficient is negative the leading coefficient of $a$, and the
    others can be computed inductively.  So we can find a $g_{ij}$ that
    satisfies our constraints.  Since we can do this for all pairs $(i,j)$, we
    have constructed a $g$ such that $\f(g) = h$.  Note that if $h$ is in
    $U(\O_L)$, then by the same induction on $j$ we see that $g$ must be in
    $U(\O_L)$.  Therefore $\f^{-1}(U(\O_L)) \subset U(\O_L)$.

    To prove injectivity, assume that
    \begin{equation*}
      \f(g_1) = \f(g_2).
    \end{equation*}
    Then we have
    \begin{align*}
      g_1^{-1}\en\s(g_1)\enm &= g_2^{-1}\en\s(g_2)\enm \\
      \en\s(g_1 g_2^{-1})\enm &= g_1 g_2^{-1}
    \end{align*}
    But since $\nu$ is strictly dominant, the only way this can happen is if
    the valuations of all the off-diagonal entries of $g_1 g_2^{-1}$ are
    infinite, which means $g_1 = g_2$.  Thus $\f$ is a bijection.

    Finally, by (\ref{eq:explicit_f_nu}) and because $\nu$ is strictly
    dominant, if $g$ is in $U(\O_L)$, then so is $\f(g)$.  Since we already
    knew that $\f^{-1}(U(\O_L)) \subset U(\O_L)$, we see that $\f(U(\O_L)) =
    U(\O_L)$.
  \end{proof}

  \subsubsection{The Conjugation Action of $\A$}
  We start with the ``conjugation action'' of $\A$ on $\Xw \cap U_w I$, as
  defined at the beginning of Section~\ref{sec:U-orbits}.  If we define an
  action of $\A$ on $I'yI'\enm \cap U_1$ by $t \cdot g = (wtw^{-1}) g
  (wt^{-1}w^{-1})$, then the two actions are compatible with the isomorphism in
  (\ref{eq:f_nu_isomorphism}).  We will investigate the structure of subsets of
  $\Xw \cap U_w I$ on which the action of $\A$ can be extended to an action of
  $A(\O_L)$.

  \begin{defn}
    \label{def:U_N}
    For any $N \in \Z$, let $\mu = (N, 2N, \ldots, nN)$ and define $U_N :=
    \e^{-\mu}U(\O_L)\e^{\mu}$.
  \end{defn}
  Note that $U_N$ is a subgroup of $U_1$.
  \begin{prop}
    \label{prop:f_bijection_U_N}
    $\f$ can be viewed as a function $U_N \to U_N$, and this function is
    bijective.
  \end{prop}
  \begin{proof}
    Since
    \begin{equation*}
      \f(\enmu g \emu) = \enmu \f(g) \emu,
    \end{equation*}
    and since $\f(U(\O_L)) = U(\O_L)$ by Proposition~\ref{prop:f_nu_bijection},
    we see that $\f(U_N) = U_N$.  Also by
    Proposition~\ref{prop:f_nu_bijection}, $\f$ is a bijection.
  \end{proof}

  \begin{defn}
    For any $m \in \Z$, $m > 0$, let $\phi_m: U(\O_L) \to U(\O_L/\e^m\O_L)$ be the
    map induced by the quotient map $\O_L \to \O_L/\e^m\O_L$.  Let $\mu$ be as in
    definition~\ref{def:U_N} and define $\Umn := \e^{-\mu} (\ker \phi_m)
    \e^{\mu}$.
  \end{defn}

  Since $\ker \phi_m$ is a normal subgroup of $U(\O_L)$, $\Umn$ is a normal
  subgroup of $U_N$.  Furthermore, since $\nu$ is dominant, $\en (\ker\phi_m)
  \enm \subset \ker\phi_m$, and hence
  \begin{equation*}
    \en (\Umn) \enm \subset \Umn.
  \end{equation*}

  \begin{prop}
    \label{prop:additive}
    If $g \in U_N$ and $u \in \Umn$, then $(gu)_{ij} - g_{ij} \in
    \e^{((j-i)N+m)}\O_L$ for all $(i, j)$.
  \end{prop}
  \begin{proof}
    \begin{equation*}
      \e^\mu gu \e^{-\mu} = (\e^\mu g \e^{-\mu}) (\e^\mu u \e^{-\mu})
    \end{equation*}
    and $(\e^\mu g \e^{-\mu}) \in U(\O_L)$ and $(\e^\mu u \e^{-\mu}) \in \ker
    \phi_m$.  Therefore 
    \begin{equation*}
      (\e^\mu gu \e^{-\mu})_{ij} - (\e^\mu g \e^{-\mu})_{ij} \in \e^m\O_L.
    \end{equation*}
    But $(gu)_{ij} = \e^{(j-i)N} (\e^\mu gu \e^{-\mu})_{ij}$ and $g_{ij} =
    \e^{(j-i)N} (\e^\mu g \e^{-\mu})_{ij}$, and the proposition follows.
  \end{proof}

  \begin{prop}
    \label{prop:well-def}
    If $g,g' \in U_N$ and $g^{-1}g' \in \Umn$, then $(\f(g))^{-1}\f(g') \in
    \Umn$.
  \end{prop}
  \begin{proof}
    Let $u = g^{-1}g' \in \Umn$.  Then $g' = gu$ and we have:
    \begin{align*}
      (\f(g))^{-1}\f(g') &= (\f(g))^{-1}(g')^{-1}\en\s(g')\enm \\
      &= (\f(g))^{-1} \cdot u^{-1} \cdot g^{-1}\en\s(g)\s(u)\enm \\
      &= (\f(g))^{-1} \cdot u^{-1} \cdot \f(g)\cdot\en\s(u)\enm. \\
      \noalign{\noindent And since $\Umn$ is normal in $U_N$ }
      &= u'\cdot\en\s(u)\enm
    \end{align*}
    for some $u' \in \Umn$.  But $\s(\Umn) \subset \Umn$, and since $\nu$ is
    dominant we have $\en\s(u)\enm \in \Umn$.  Hence $u' \cdot \en\s(u)\enm \in
    \Umn$.
  \end{proof}

  \begin{defn}
    Let $\U = U_N/\Umn$. Given an equivalence class $g\Umn$, we define
    $\ff(g\Umn)$ to be $\f(g)\Umn$.  By Proposition~\ref{prop:well-def} this
    gives us a well-defined map $\ff : \U \to \U$.
  \end{defn}

  Note that, by Proposition~\ref{prop:additive}, $\U$ is an affine variety of
  dimension $n(n-1)m/2$ over $\k$.

  \begin{prop}
    \label{prop:isomorphism}
    $\ff : \U \to \U$ is an isomorphism of varieties.
  \end{prop}
  \begin{proof}
    Since $\f$ is surjective, so is $\ff$. 

    To show injectivity, consider $g\Umn$ and $g'\Umn$ such that
    \begin{equation*}
      \ff(g\Umn) = \ff(g'\Umn).
    \end{equation*}
    This means that $\f(g)$ and $\f(g')$ differ by
    right-multiplication by some element of $\Umn$.  Call this element $u$.
    Then we have:
    \begin{align}
    \f(g) &= \f(g')u \nonumber \\
    g^{-1}\en\s(g)\enm &= (g')^{-1}\en\s(g')\enm u \nonumber \\
    g' g^{-1} &= \en\s(g')\enm u \en\s(g^{-1})\enm \nonumber \\
    \noalign{\noindent and, since $\Umn$ is normal in $U_N$ and
      $\en\s(g^{-1})\enm \in U_N$,} \nonumber \\
    \label{eq:differ_by_u}
    g' g^{-1} &= \en\s(g' g^{-1})\enm u'
    \end{align}
    for some $u' \in \Umn$.  But now we must have:
    \begin{equation*}
      (g'g^{-1})_{ij} = (\en\s(g' g^{-1})\enm u')_{ij}
    \end{equation*}
    and by Proposition~\ref{prop:additive}
    \begin{equation*}
      (g'g^{-1})_{ij} - \e^{\nu_i - \nu_j}\s((g'g^{-1})_{ij}) \in \e^{((j-i)N+m)}\O_L. \\
    \end{equation*}
    Since $\nu$ is strictly dominant this is only possible if $(g'g^{-1})_{ij}
    \in \e^{((j-i)N+m)}\O_L$ for all $i < j$, which means $g' g^{-1} \in
    \Umn$, which means that $g\Umn = g'\Umn$.  Thus $\ff$ is bijective.
    
    Now by Proposition~\ref{prop:f_nu_dependence}, an entry of $\f(g)$ only
    depends on the corresponding entry of $g$ and on the entries of $g$ that
    are closer to the main diagonal than itself.  Therefore, we can pick a
    basis for $\U$ (just listing bases for the entries starting with the ones
    near the main diagonal and working out) in which $d\ff$, which is the
    matrix of the differential of $\ff$, is block-lower-triangular.  And since
    $\nu$ is dominant and $d\s = 0$, the blocks on the diagonal are themselves
    lower-triangular, with all diagonal entries equal to -1, and hence $\det
    d\ff$ is everywhere nonzero (and in fact is $\pm 1$).  Therefore $d\ff$ is
    bijective at all points, and $\ff$ is an isomorphism.
  \end{proof}

  \begin{prop}
    \label{prop:T_subset_U_N}
    Let $T \subset U_1$ such that the set of valuations of entries of elements
    of $T$ is bounded below. Then we can pick $N$ such that $T \subset U_N$ and
    $\f^{-1}(T) \subset U_N$.
  \end{prop}
  \begin{proof}
    Let $N$ be any negative integer smaller than the lower bound of the
    valuations of entries of elements of $T$.  Then $T \subset U_N$ and hence,
    by Proposition~\ref{prop:f_bijection_U_N}, $\f^{-1}(T) \subset U_N$.
  \end{proof}
  \begin{corr}
    \label{corr:T_subset_U_N}
    The $N$ can be chosen such that $U_1 \cap I' \subset U_N$ as well.
  \end{corr}
  \begin{proof}
    The valuations of all entries of elements of $U_1 \cap I'$ are bounded
    below by the difference between the smallest and the largest valuations of
    entries of $w$.  So $N$ just needs to be selected to also be smaller than
    this difference.
  \end{proof}

  \begin{prop}
    \label{prop:f_equals_fbar}
    Let $T \subset U_1$ is as in Proposition~\ref{prop:T_subset_U_N} and pick
    an $N$ per that proposition.  If there is an $m$ such that $T\Umn \subset
    T$ then $\ff^{-1}(T/\Umn) = \f^{-1}(T)/\Umn$, where both sides are viewed
    as subsets of $\U$.
  \end{prop}
  \begin{proof}
    By Proposition~\ref{prop:T_subset_U_N}, the sets $T/\Umn$ and
    $\f^{-1}(T)/\Umn$ are well-defined subsets of $\U$.

    If an element of $\U$ is in $\f^{-1}(T)/\Umn$, we can pick a representative
    $g$ for it such that $\f(g) \in T$.  But then $\ff(g\Umn) = \f(g)\Umn$ has
    nonempty intersection with $T$, so $g\Umn \in \ff^{-1}(T/\Umn)$ as desired.

    Conversely, say an element of $\U$ is in $\ff^{-1}(T/\Umn)$.  This means
    that for any representative $g$ we have $\f(g)\Umn \cap T \ne \emptyset$.
    Hence, $\f(g) \in T \Umn \subset T$.  So $g \in \f^{-1}(T)$ and our
    original element of $\U$ is in $\f^{-1}(T)/\Umn$.
  \end{proof}

  \begin{prop}
    \label{prop:trivial-action}
    Assume we have a set $T$ as in Proposition~\ref{prop:f_equals_fbar} and can
    pick $m$ such that it satisfies the conditions of that proposition and such
    that $\Umn \subset I'$.  Assume further that $\f^{-1}(T)$ is preserved
    under right-multiplication by $U_1 \cap I'$, that $\f^{-1}(T)/(U_1 \cap
    I')$ is a variety, that $\A$ acts on $T$ via the action $t \cdot g =
    (wtw^{-1}) g (wt^{-1}w^{-1})$, and that the action of $\A$ on $T$ can be
    extended to an action of $A(\O_L)$ on $T$ given by the same formula.

    Then $\A$ acts on $\f^{-1}(T)/(U_1 \cap I')$, with the action given by the
    same formula as the action on $T$, and the resulting representation of $\A$
    on the Borel-Moore homology of $\f^{-1}(T)/(U_1 \cap I')$ is trivial.
  \end{prop}
  \begin{proof}
    First, we note that the action of $\A$ described above is compatible with
    $\f$ and clearly preserves both $U_1 \cap I'$ and $\Umn$.  Since it is
    compatible with $\f$ and preserves $\Umn$, it is compatible with $\ff$ and
    descends to an action on $T/\Umn$.
  
    By Corollary~\ref{corr:T_subset_U_N} we can pick $N$ small enough that $U_1
    \cap I' \subset U_N$.  Then
    \begin{align*}
      \f^{-1}(T)/(U_1 \cap I') &= \f^{-1}(T)/(U_N \cap I') \\
      &=  \biggl(\bigl(\f^{-1}(T)/\Umn \biggr)/\biggl((U_N
      \cap I') / \Umn\biggr) \\
      \noalign{\noindent and by Proposition~\ref{prop:f_equals_fbar}}
      & = \biggl(\ff^{-1}\bigl(T/\Umn\bigr)\biggr)/\biggl((U_N
      \cap I') / \Umn\biggr).
    \end{align*}
    Now $(U_N \cap I') / \Umn$ is a finite-dimensional affine space by
    Proposition~\ref{prop:additive}, and the action of $\A$ preserves all the
    quotients involved, so the Borel-Moore homology of $\f^{-1}(T)/(U_1 \cap
    I')$ is the same as that of $\ff^{-1}(T/\Umn)$ but shifted in degree.
    Since $\ff$ is an isomorphism which is compatible with the action of $\A$,
    it is enough to consider the representation of $\A$ on the homology of
    $T/\Umn$ induced by the action of $\A$ on $T/\Umn$.  But on $T$ we can
    extend the action of $\A$ to an action of $A(\O_L)$, and this action also
    descends to $T/\Umn$.  Since $A(\O_L)$ is connected, the representation on
    the homology must be trivial.
  \end{proof}

  \subsection{Some Triviality Results}
  \label{sec:triviality-results}
  
  \begin{theorem}
    \label{thm:trivial-on-U-orbit}
    If $\nu$ is strictly dominant and $w \in W$, then the representation of
    $\A$ on the Borel-Moore homology of $\X \cap U_1 wI$ induced by the
    left-multiplication action of $\A$ on the set $\X \cap U_1 wI$ is trivial.
  \end{theorem}
  \begin{proof}
    Since $\A \subset I$, the left-multiplication action and the conjugation
    action coincide.  By (\ref{eq:f_nu_isomorphism}),
    \begin{equation*}
      \X \cap U_1 wI \cong \f^{-1}(I'yI'\enm \cap U_1)/(U_1 \cap I')
    \end{equation*}
    and the isomorphism sends the conjugation action on $\X \cap U_1wI$ to
    exactly the action described in the conditions of
    Proposition~\ref{prop:trivial-action} on $\f^{-1}(I'yI'\enm \cap U_N)/(U_N
    \cap I')$.  In this case, $T = I'yI'\enm \cap U_1$.  The valuations of the
    entries of elements of $T$ are bounded below, so
    Proposition~\ref{prop:T_subset_U_N} applies.  If we take $m$ large enough
    that $\enm\Umn\en \subset I'$, then
    \begin{align*}
      T\Umn &\subset (I'yI'\enm \en I' \enm) \cap U_N \\
      &= (I'yI'\enm) \cap U_N \\
      &= T
    \end{align*}
    so Proposition~\ref{prop:f_equals_fbar} applies.  $f^{-1}(T)$ is preserved
    under right-multiplication by $U_1 \cap I'$, $\A$ acts on $T$ by the
    requisite twisted conjugation action, and this action on $T$ can clearly be
    extended to an action of $A(\O_L)$.  So
    Proposition~\ref{prop:trivial-action} applies.
  \end{proof}
  \begin{theorem}
    \label{thm:trivial-on-a-piece}
    Assume that $\nu$ is strictly dominant.  Let $w = 1$ or the longest element
    of $W$.  Pick an integer $\delta$ which is nonpositive, and negative if $w
    = 1$.  Let $Y_\delta$ be the subset of $\X \cap U_1 w I$ which consists of
    elements which can be represented by an element of $U_1w$ which has a
    $(1,w(2))$ entry whose valuation is $\delta$.  Then $\A$ acts on $Y_\delta$
    by left-multiplication, and the representation of $\A$ induced on the
    Borel-Moore homology of $Y_\delta$ by this action is trivial.
  \end{theorem}
  \begin{proof}
    Left-multiplication by elements of $\A$ does not change the valuations of
    entries, so $\A$ acts by left-multiplication on $Y_\delta$.  Since $\A
    \subset I$, this action coincides with the conjugation action of $\A$ on
    $Y_\delta$.  By (\ref{eq:explicit_f_nu}), since $\nu$ is strictly dominant,
    we see that $\val(\f(g)_{1,2}) = \val(g_{1,2})$.  In particular, $Y_\delta
    \cong \f^{-1}(Z_\delta)/(U_1 \cap I')$, where $Z_\delta$ is the subset of
    $I'yI'\enm \cap U_1$ which consists of elements that can be represented by
    a matrix which has an entry of valuation $\delta$ in the $(1,2)$ position,
    and the isomorphism sends the conjugation action on $Y_\delta$ to the
    action described in the conditions of Proposition~\ref{prop:trivial-action}
    on $\f^{-1}(Z_\delta)/(U_1 \cap I')$.  In this case, $T = Z_\delta$.
    $Z_\delta$ satisfies the conditions of Corollary~\ref{prop:T_subset_U_N},
    since it is a subset of a set that satisfies those conditions.  If we take
    $m$ such that $\enm\Umn\en \subset I'$ and $N + m > \delta$, then by the
    argument for Theorem~\ref{thm:trivial-on-U-orbit} right-action by $\Umn$
    preserves $I'yI'\enm \cap U_1$, and by Proposition~\ref{prop:additive} this
    action preserves $Z_\delta$.  So $Z_\delta$ satisfies the conditions of
    Proposition~\ref{prop:f_equals_fbar}.  Since $\delta \le 0$ and $\delta <
    0$ if $w = 1$, $\f^{-1}(Z_\delta)$ is preserved by right-multiplication by
    $U_1 \cap I'$.  The twisted conjugation actions of both $\A$ and $A(\O_L)$
    preserve $Z_\delta$, so Proposition~\ref{prop:trivial-action} applies.
  \end{proof}

  Now we prove a result that we will be able to apply to most cases when $G =
  SL_3$.
  \begin{prop}
    \label{prop:Trivial-For-Closed-Transitive-Set}
    Assume that $\X$ is a disjoint union of subsets which have the following
    properties:
    \begin{itemize}
      \item Each subset is preserved by the left-multiplication action of
	$\A$.
      \item One of the subsets is closed in $\X$.
      \item The induced representation of $\A$ on the Borel-Moore homology of
        one of the subsets is trivial.
      \item $A(F)/\A$ acts simply transitively on the collection of subsets.
    \end{itemize}
    Then the representation of $\A$ on the Borel-Moore homology of $\X$ which
    is induced by the left-multiplication action of $\A$ on $\X$ is trivial.
    Furthermore, the representation of $A(F)$ on the Borel-Moore homology of
    $\X$ simply permutes the homology spaces of the subsets in our
    decomposition.
  \end{prop}
  \begin{proof}
    Since left-multiplication by $\e^\mu$ commutes with left-multiplication by
    elements of $\A$, and since $A(F)/\A$ acts simply transitively on our
    subsets, the representation of $\A$ on the Borel-Moore homology of each
    subset is trivial.  Since one of the subsets is closed in $\X$, and all the
    subsets are translates of each other by elements of $A(F)$, all the
    subsets are closed in $\X$.  Now $\X$ is a disjoint union of closed
    subsets, so the Borel-Moore homology of $\X$ is just the coproduct of the
    Borel-Moore homologies of the pieces.  Therefore, $\A$ also acts trivially
    on the Borel-Moore homology of $\X$.

    Finally, $A(F)/\A$ acts simply transitively on the closed pieces we have
    decomposed $\X$ into.  Therefore the representation of $A(F)$ on the
    Borel-Moore homology of $\X$ just permutes the homology spaces of the
    pieces.
  \end{proof}
  
  \begin{theorem}
    \label{thm:Triviality-For-One-Closed-Intersection}
    Assume that $\nu$ is strictly dominant, that $\X \cap U_1 wI$ is empty for
    all but one $w \in W$ and that the one nonempty intersection is closed in
    $\X$.  Then the sets $\X \cap U_1 \w I$ for $\w \in \W$ satisfy the
    conditions of Proposition~\ref{prop:Trivial-For-Closed-Transitive-Set}, and
    hence the conclusion of that proposition follows.
  \end{theorem}
  \begin{proof}
    Since the sets $U_1 \w I$ for $\w \in \W$ partition $X$, we see that
    \begin{equation*}
      \X = \coprod_{\w\in\W} \X \cap U_1 \w I
    \end{equation*}
    as a set. Now if $\w = \e^\mu w$ with $w \in W$, then
    \begin{equation}
      \label{eq:translates-of-w-do-not-matter}
      \X \cap U_1 \w I = \e^\mu (\X \cap U_1 w I),
    \end{equation}
    since $\e^\mu \X = \X$ and $\e^\mu U_1 = U_1 \e^\mu$.  Let $w_0 \in W$ be
    the unique Weyl group element such that $\X \cap U_1 w_0 I$ is nonempty.
    By (\ref{eq:translates-of-w-do-not-matter}), if $\X \cap U_1 \w I$ is
    nonempty, we must have $\w = \e^\mu w_0$ for some $\mu$.  Now the
    representation of $\A$ on the Borel-Moore homology $\X \cap U_1 w_0 I$ is
    trivial by Theorem~\ref{thm:trivial-on-U-orbit}.  Since each $\X \cap U_1
    \w I$ is preserved by the action of $\A$, the conditions of
    Proposition~\ref{prop:Trivial-For-Closed-Transitive-Set} are satisfied.
  \end{proof}

  Finally, we prove a result that will be used for the remaining $G = SL_3$
  cases.
  \begin{prop}
    \label{prop:Triviality-For-Stratified-Set}
    Assume that we have a variety $S \subset X$ on which $\A$ acts, such the
    valuations of entries of representatives of points of $S$ are bounded
    below.  Further, assume that we have a stratification $S_0 \subset S_1
    \subset \cdots \subset S_m = S$, where $S_i$ is closed in $S_{i + 1}$ for
    all $i < m$.  Assume that on $S_0$ and on $T_i = S_i \setminus S_{i - 1}$
    for $i \ge 1$ the action of $\A$ can be extended to an action of $A(\O_L)$.
    Then the representation of $\A$ on the Borel-Moore homology of $S$ induced
    by the action of $\A$ on $S$ is trivial.
  \end{prop}
  \begin{proof}
    Denote the Borel-Moore homology by $H^{BM}$.  Because of our assumption
    that the action of $\A$ can be extended to an action of $A(\O_L)$ on $T_i$,
    $\A$ acts trivially on $H_j^{BM}(T_i)$.  We will prove by induction on $i$
    that it acts trivially on $H_j^{BM}(S_i)$ for all $i$, which will give us
    our conclusion when $i = m$.

    The base case $i = 0$ follows from our assumption about the action of $\A$
    on $S_0$ being extensible to an action of $A(\O_L)$.  For the induction
    step, note that since $S_i$ is closed in $S_{i + 1}$, we have a long exact
    sequence in compactly supported cohomology:
    \begin{equation*}
      \xymatrix{
	\cdots \ar[r] &
	H^j_c(T_i) \ar[r] &
	H^j_c(S_{i+1}) \ar[r] &
	H^j_c(S_i) \ar[r] &
	\cdots}.
    \end{equation*}
    Taking duals, we have a long exact sequence in Borel-Moore homology:
    \begin{equation*}
     \xymatrix{
       \cdots \ar[r] &
       H_j^{BM}(S_i) \ar[r]^-f &
       H_j^{BM}(S_{i+1}) \ar[r]^-g &
       H_j^{BM}(T_i) \ar[r] &
       \cdots}
    \end{equation*}
    and hence the short exact sequence
    \begin{equation*}
     \xymatrix{
       0 \ar[r] &
       \ker(f) \ar[r]^-f &
       H_j^{BM}(S_{i+1}) \ar[r]^-g &
       \coker(g) \ar[r] &
       0}
    \end{equation*}
    Now $\A$ acts trivially on $H_j^{BM}(S_i)$ and hence on $\ker(f)$.  It acts
    trivially on $H_j^{BM}(T_i)$ and hence on $\coker(g)$.  Furthermore, since
    the valuations of entries of representatives of elements of $S$ are bounded
    below, there is some $N > 0$ such that the action of $\A$ on $S$ factors
    through $A(\O_F/\e^N\O_F)$, which is a finite group with $q^{nN}$ elements.
    Therefore, the representations of $\A$ on the Borel-Moore homology of
    subvarieties of $S$ are a semisimple category, and in particular we must
    have
    \begin{equation*}
      H_j^{BM}(S_{i+1}) = \ker(f) \oplus \coker(g)
    \end{equation*}
    as representations of $\A$.  Therefore, the representation of $\A$ on
    $H_j^{BM}(S_{i+1})$ is trivial.
  \end{proof}

  \begin{theorem}
    \label{thm:Triviality-For-Nice-Stratification}
    Assume that $\nu$ is strictly dominant and that that $\X \cap U_1 wI$ is
    empty for all but two values of $w$: $w_0$ and $w_1$, where $w_1$ is either
    $1$ or the longest element in $W$.  Assume further that if $\X \cap U_1
    w_1I$ is divided up into subsets $Y_\delta$ as in
    Theorem~\ref{thm:trivial-on-a-piece}, then for each $\delta$ there is a
    $\mu_\delta \in A(F)/\A$ such that
    \begin{equation*}
      (\X \cap U_1 w_0 I) \bigcup \left(\bigcup_{\delta > m} \e^{\mu_\delta} Y_\delta\right)
    \end{equation*}
    is closed in
    \begin{equation*}
      Z = (\X \cap U_1 w_0 I) \bigcup \left(\bigcup_{\delta} \e^{\mu_\delta}
      Y_\delta\right)
    \end{equation*}
    for all $m \in \Z$ and that $Z$ is closed in $\X$.  Then the representation
    of $\A$ on the Borel-Moore homology of $\X$ which is induced by the
    left-multiplication action of $\A$ on $\X$ is trivial.  Furthermore, the
    representation of $A(F)$ on the Borel-Moore homology of $\X$ simply
    permutes the homology spaces of the translates of $Z$.
  \end{theorem}
  \begin{proof}
    First, note that, by Theorem~\ref{thm:trivial-on-a-piece}, $\A$ acts on the
    sets $Y_\delta$, and hence the sets $\e^{\mu_\delta}Y_\delta$, by
    left-multiplication, and the resulting representation on the Borel-Moore
    homology of $\e^{\mu_\delta}Y_\delta$ is trivial.  $\A$ also acts on $(\X
    \cap U_1 w_0 I)$ by left-multiplication, so it acts on $Z$ by
    left-multiplication.  By Theorem~\ref{thm:trivial-on-U-orbit}, the
    representation of $\A$ on the Borel-Moore homology of $(\X \cap U_1 w_0 I)$
    is trivial.  So by Proposition~\ref{prop:Triviality-For-Stratified-Set},
    $\A$ acts trivially on the Borel-Moore homology of $Z$.
    
    Since $A(F)/\A$ acts simply transitively on the translates of $Z$ and $Z$
    is closed in $\X$ by assumption,
    Proposition~\ref{prop:Trivial-For-Closed-Transitive-Set} applies to give us
    the desired result.
  \end{proof}

  \section{$G = SL_2$}
  \label{chap:SL_2}
  When $G = SL_2$, the left-multiplication action of $\A$ on $\X$ can be
  directly extended to the left-multiplication action of $A(\O_L)$ on $\X$.
  Indeed, let $g \in SL_2(L)$ and let
  \begin{equation*}
    \en = \begin{pmatrix}
      \e^{m} & 0 \\
      0 & \e^{-m} \\
    \end{pmatrix} 
  \end{equation*}
  with $m \neq 0$.  Pick an element of $A(\O_L)$, call it
  \begin{equation*}
    \tau = \begin{pmatrix}
      t & 0 \\
      0 & t^{-1}
    \end{pmatrix},
  \end{equation*}
  and let $g' = \tau g$.

  Now $g \in U_1 \w I$ for some $\w \in \W$.  There are two cases:
  \begin{caselist}
    \item
      \begin{equation*}
	\w = \begin{pmatrix} \e^k & 0 \\ 0 & \e^{-k} \end{pmatrix}
      \end{equation*}
      with $k \in \Z$.  Since we can change $g$ by right-multiplication
      by elements of $I$ without affecting anything, we can take
      \begin{equation*}
	g = \begin{pmatrix} \e^k & \e^{-k}a \\ 0 & \e^{-k} \end{pmatrix}
      \end{equation*}
      with $a \in L$.

      Let
      \begin{equation*}
	h = g^{-1}\en \s(g) =
	\begin{pmatrix}
	  \e^m & \s(a)\e^{m-2k} - a \e^{-m-2k} \\
	  0 & \e^{-m}
      	\end{pmatrix}
      \end{equation*}
      and 
      \begin{equation*}
	h' = (g')^{-1}\en \s(g') = 
	\begin{pmatrix}
	  \e^m t^{-1}\s(t) & \s(a)\e^{m-2k}t^{-1}\s(t) - a \e^{-m-2k}
	  t\s(t^{-1}) \\
	  0 & \e^{-m}t\s(t^{-1})
	\end{pmatrix}.
      \end{equation*}
      Now the valuations of $t^{-1}\s(t)$ and $t\s(t^{-1})$ are 0, so the
      valuations of the top-left, bottom-left, and bottom-right entries of $h$
      and $h'$ are clearly the same.  Since $m \neq 0$, the two terms in the
      top-right entry of each matrix have different valuations.  We see that
      the valuation of the top-right entry of $h$ is $\val(a) - |m| - 2k$, and
      the same is true for the top-right entry of $h'$.  Both $h$ and $h'$ have
      only one $2 \times 2$ minor, and its valuation is 1.  So by
      Theorem~\ref{thm:DeterminingIwahoriDoubleCosetForElement}, $h$ and $h'$
      are both in $IxI$ for the same $x$.  This means that $g$ and $g'$ are
      both in $\X$ for the same $x$.
    \item
      \begin{equation*}
	\w = \begin{pmatrix} 0 & \e^k \\ \e^{-k} & 0 \end{pmatrix}
      \end{equation*}
      with $k \in \Z$.  Since we can change $g$ by right-multiplication
      by elements of $I$ without affecting anything, we can take
      \begin{equation*}
	g = \begin{pmatrix} a\e^{-k} & \e^k \\ \e^{-k} & 0 \end{pmatrix}
      \end{equation*}
      with $a \in L$.

      Let
      \begin{equation*}
	h = g^{-1}\en \s(g) =
	\begin{pmatrix}
	  \e^{-m} & 0 \\
	  \s(a)\e^{m-2k} - a \e^{-m-2k} & \e^{m} \\
      	\end{pmatrix}
      \end{equation*}
      and 
      \begin{equation*}
	h' = (g')^{-1}\en \s(g') = 
	\begin{pmatrix}
	  \e^{-m} t\s(t)^{-1} & 0 \\
	  \s(a)\e^{m-2k}t^{-1}\s(t) - a \e^{-m-2k} t\s(t^{-1}) & \e^{m}t^{-1}\s(t)\\
	\end{pmatrix}.
      \end{equation*}
      Now the valuations of $t^{-1}\s(t)$ and $t\s(t^{-1})$ are 0, so the
      valuations of the top-left, top-right, and bottom-right entries of $h$
      and $h'$ are clearly the same.  Since $m \neq 0$, the two terms in the
      bottom-left entry of each matrix have different valuations.  We see that
      the valuation of the bottom-left entry of $h$ is $\val(a) - |m| - 2k$,
      and the same is true for the bottom-left entry of $h'$.  Both $h$ and
      $h'$ have only one $2 \times 2$ minor, and its valuation is 1.  So by
      Theorem~\ref{thm:DeterminingIwahoriDoubleCosetForElement}, $h$ and $h'$
      are both in $IxI$ for the same $x$.  This means that $g$ and $g'$ are
      both in $\X$ for the same $x$.
  \end{caselist}
  Since in both cases we found that $g$ and $\tau g$ are both in $\X$ for the
  same $x$, we conclude that $\X$ is preserved by the left-multiplication
  action of $A(\O_L)$.  Therefore, the representation of $\A$ on the
  Borel-Moore homology of $\X$ is trivial.
  
  \section{$G = SL_3$}
  \label{chap:SL_3}
  When $G = SL_3$, the left-multiplication action of $\A$ on $\X$ cannot be
  directly extended to a left-multiplication action of $A(\O_L)$ in all cases.
  We have to treat various cases directly.  Throughout this chapter, we
  identify $W$ with the permutation group $\Sigma_3$, and label the
  transpositions $(12)$ and $(23)$ by $s_1$ and $s_2$ respectively.  We will
  use these two elements as generators for $W$ as a Coxeter group.  Lengths of
  elements of $W$ will mean the lengths of the shortest expression in terms of
  $s_1$ and $s_2$.  Let  
  \begin{equation*}
    \eta := s_1 s_2 s_1 = s_2 s_1 s_2 = \begin{pmatrix}
      0 & 0 & 1 \\
      0 & 1 & 0 \\
      1 & 0 & 0
    \end{pmatrix}
  \end{equation*}
  be the maximal-length element of $W$.

  Now let us fix a point in the base alcove of the Bruhat-Tits building for
  $SL_3(L)$.  As discussed in~\cite{Kottwitz-Harmonic-Analysis}, for every
  point of $X$ there is a corresponding convex polytope in the standard
  apartment of the building.  In the case of $SL_3$, this has six vertices and
  the standard apartment is a plane, so the convex polytope is a hexagon.
  These hexagons have sides that are perpendicular to the edges of the base
  alcove.  To find the hexagon corresponding to a given point $gI$ of $X$, one
  needs to find, for each $w \in W$, the $x \in \W$ such that $g \in w^{-1}U_1w
  x I$.  Applying those six extended affine Weyl group elements to the base
  alcove gives us six images of our chosen point in the standard apartment.
  These images are the vertices of the hexagon.  Two vertices are connected to
  each other if the corresponding $w \in W$ have lengths that differ by 1.  We
  will label the vertices of a hexagon with the corresponding Weyl group
  elements.

  Given such a hexagon $E$, the set of all $gI$ such that the hexagon
  corresponding to $g$ is a subset of $E$ is a closed set in $X$.  This means
  that if we have a subset $S$ of $X$, the set of points whose hexagon is
  contained in the hexagon of some point of $S$ is a closed set containing $S$,
  and hence contains the closure of $S$.

  \begin{prop}
    \label{prop:closedness-of-hexagon-piece}
    Assume that $\X$ is a disjoint union of subsets which satisfy the following
    properties:
    \begin{itemize}
      \item Each subset is preserved by the left-multiplication action of
	$\A$.
      \item $A(F)/\A$ acts simply transitively on the collection of subsets.
      \item There is some subset $Y$, a $w_1 \in W$ and $y_1, y_2 \in \W$ such
	that if $E$ is the hexagon corresponding to any element of $Y$ the
	$w_1$ corner of $E$ is given by $y_1$ and the $\eta w_1$ corner of $E$
	is given by $y_2$.
    \end{itemize}
    Then $Y$ is a closed subset of $\X$.
  \end{prop}
  \begin{proof}
    By assumption, if $gI \in \X$, then $gI = \e^\mu hI$ for some $\mu$, where
    $hI \in Y$.  This means that the corners of the hexagon corresponding to
    $gI$ are translates by $\e^\mu$ of the corners of the hexagon corresponding
    to $hI$.  Now assume that $gI \notin Y$, so that $\mu \neq (0, 0, 0)$.

    By assumption, all the hexagons corresponding to elements of $Y$ share a
    pair of opposite vertices.  The hexagon corresponding to $gI$ is a
    translate of one of those hexagons.  But if two hexagons share a pair of
    opposite vertices, one of them cannot contain a translate of the other.
    Indeed, let the two opposite vertices that the hexagons share be $z_1$ and
    $z_2$.  These are points in the standard apartment.  Since the sides of the
    hexagons must be perpendicular to the sides of the base alcove, both of the
    hexagons we are considering must lie in the intersection of two closed
    cones, one with vertex at $z_1$, and one with vertex at $z_2$, as shown in
    Figure~\ref{fig:cones}.  The angle of each cone is $120^\circ$.  Since this
    is less than $180^\circ$, if $z_1$ and $z_2$ are both translated by the
    same nonzero vector, one or the other of them will lie outside the
    intersection of the two cones.
    \begin{figure}[htbp]
      \begin{center}
	\includegraphics{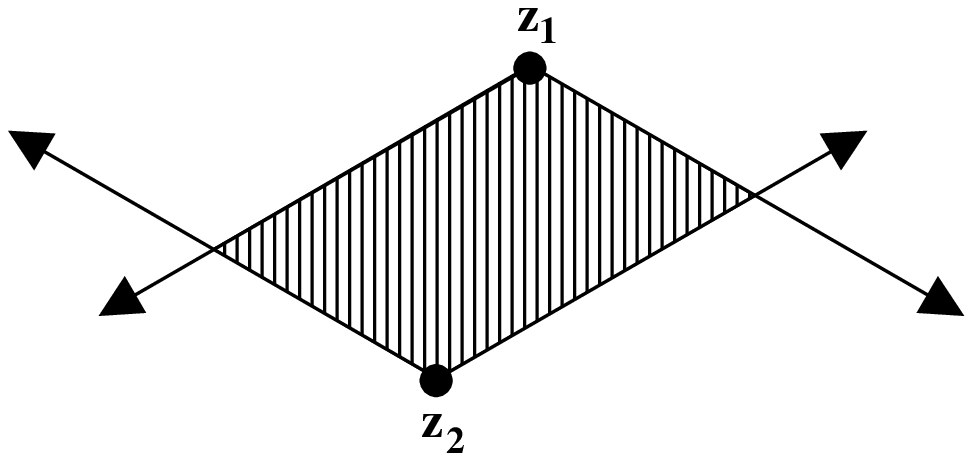}
      \end{center}
      \caption{Illustration of two opposite vertices of a hexagon.  The hexagon
      must be contained in the shaded region.}
      \label{fig:cones}
    \end{figure}

    Therefore, the hexagon corresponding to $gI$ is not contained in any of the
    hexagons corresponding to elements of $Y$.  This means that $gI$ is not in
    the closure of $Y$ in $\X$.  Since this was true for any $gI \notin Y$, it
    follows that $Y$ is closed in $\X$.
  \end{proof}

  \begin{theorem}
    \label{thm:Triviality-For-One-Intersection}
    Assume that $\nu$ is strictly dominant, that $\X \cap U_1 wI$ is empty for
    all but one $w \in W$, and that there is a $w_1 \in W$ and $y_1, y_2 \in
    \W$ such that if $E$ is the hexagon corresponding to any element of $\X
    \cap U_1 wI$ the $w_1$ corner of $E$ is given by $y_1$ and the $\eta w_1$
    corner of $E$ is given by $y_2$.  Then the representation of $\A$ on the
    Borel-Moore homology of $\X$ which is induced by the left-multiplication
    action of $\A$ on $\X$ is trivial.  Furthermore, the representation of
    $A(F)$ on the Borel-Moore homology of $\X$ simply permutes the homology
    spaces of translates of the one nonempty intersection $\X \cap U_1 wI$.
  \end{theorem}
  \begin{proof}
    Let $w_0 \in W$ be the unique Weyl group element such that $\X \cap U_1 w_0
    I$ is nonempty.  Since $\X$ is the disjoint union of sets of the form $\X
    \cap U_1 \w I$ for $\w \in \W$, we see that by
    Proposition~\ref{prop:closedness-of-hexagon-piece} $\X \cap U_1 w_0 I$ is
    closed.  Now by Theorem~\ref{thm:Triviality-For-One-Closed-Intersection},
    our conclusion follows.
  \end{proof}

  Now, we will consider all possible $x \in \W$ and $\nu = (i, j, k)$.  We will
  reduce the set of combinations of $x$ and $\nu$ that we need to consider, and
  show that $\nu$ can always be assumed to be dominant.  Then for each
  remaining combination of $x$ and $\nu$ we will show that either
  Theorem~\ref{thm:Triviality-For-Nice-Stratification} or
  Theorem~\ref{thm:Triviality-For-One-Intersection} or applies.
  
  \subsection{Reduction Steps}
  \label{sec:Reduction-Steps}
  Following Beazley~\cite{Beazley-codimensions}, we define two outer
  automorphisms of $SL_3(L)$ that preserve $I$ and commute with $\s$.

  Let
  \begin{equation*}
    \tau :=
    \begin{pmatrix}
      0 & 0 & \e^{-1} \\
      1 & 0 & 0 \\
      0 & 1 & 0
    \end{pmatrix}
  \end{equation*}
  and define $\phi(g) := \tau g \tau^{-1}$.  Then $\phi$ is an automorphism of
  $SL_3(L)$ of order 3, clearly commutes with $\s$, since $\s(\tau) = \tau$,
  and by explicit computation preserves $I$.  Also by explicit computation, if
  $w \in W$, then
  \begin{equation*}
    \phi(\e^{(\mu_1,\mu_2,\mu_3)}w) = \e^{(-1, 0, 0) + s_1s_2w(0, 0, 1) +
    s_1s_2(\mu_1, \mu_2, \mu_3)}s_1s_2ws_2s_1.
  \end{equation*}
  In particular,
  \begin{equation*}
    \phi(s_1) = s_2 \quad \mbox{and} \quad \phi(s_2) = s_1s_2s_1 = s_2s_1s_2.
  \end{equation*}

  Define $\psi(g) := \eta(g^t)^{-1} \eta^{-1}$, where $\eta$ is the maximal
  length element of $W$.  Then $\psi$ is an automorphism of $SL_3(L)$ of order
  2, commutes with $\s$, since $\s(\eta) = \eta$, and by explicit computation
  preserves $I$.  Also by explicit computation, if $w \in W$, then
  \begin{equation*}
    \psi(\e^{(\mu_1,\mu_2,\mu_3)}w) = \e^{-\eta(\mu_1, \mu_2, \mu_3)}\eta w \eta.
  \end{equation*}
  In particular,
  \begin{equation*}
    \psi(s_1) = s_2 \quad \mbox{and} \quad \psi(s_2) = s_1.
  \end{equation*}

  Since $\phi$
  and $\psi$ preserve $I$ and commute with $\s$, we see that
  \begin{equation*}
    \X \cong \phi(\X) = X_{\phi(x)}(\e^{s_1s_2\nu})
  \end{equation*}
  and
  \begin{equation*}
    \X \cong \psi(\X) = X_{\psi(x)}(\e^{-\eta\nu}).
  \end{equation*}
  Now let $t \in \A$.  Then
  \begin{equation*}
    \phi(tg) = \phi(t)\phi(g) = t'\phi(g)
  \end{equation*}
  where $t' = s_1s_2ts_2s_1 \in \A$.  Similarly,
  \begin{equation*}
    \psi(tg) = \psi(t)\psi(g) = t'\psi(g)
  \end{equation*}
  where $t' = \eta t^{-1} \eta \in \A$.

  As we can see, $\phi$ and $\psi$ do not commute with the left-multiplication
  action of $\A$ on $\X$, but under these isomorphisms this action becomes the
  composition of some endomorphism of $\A$ and the left-multiplication action.
  As a result, if all elements of $\A$ act trivially on the homology of
  $\phi(\X)$ or $\psi(\X)$, then they all act trivially on the homology of
  $\X$.

  There is also an isomorphism between $\X$ and $X_x(\e^{w\nu})$, as discussed
  in \cite{Kottwitz-f_nu}, given by left-multiplication by $w$.  Again, this
  isomorphism does not commute with the left-multiplication action of $\A$, but
  under this isomorphism this action becomes a composition of the conjugation
  action of $w$ on $\A$ and the left-multiplication action.  As a result, if
  all elements of $\A$ act trivially on the cohomology of $X_x(\e^{w\nu})$,
  then they all act trivially on the cohomology of $\X$.  Since by assumption
  the integers in $\nu$ are all distinct, by appropriate choice of $w$, we can
  always. without changing $x$, reduce to the case where $\nu$ is strictly
  dominant: $\nu = (\nu_1, \nu_2, \nu_3)$, with $\nu_1 > \nu_2 > \nu_3$.

  Now we will use $\phi$, $\psi$, and left-multiplication by appropriate $w$ to
  reduce the number of cases we need to consider.  There are three
  possibilities, depending on $x$.
  \begin{caselist}
    \item The permutation part of $x$ is the identity.  In this case, $x =
      \e^{(\mu_1,\mu_2,\mu_3)}$, so
      \begin{equation*}
        \phi(x) = \e^{(\mu_3, \mu_1, \mu_2)}
      \end{equation*}
      and
      \begin{equation*}
        \phi^2(x) = \e^{(\mu_2, \mu_3, \mu_1)}.
      \end{equation*}
      Thus we can use $\phi$ to reduce to the cases where $\mu_3 \le \mu_2 \le
      \mu_1$ or $\mu_1 \le \mu_2 \le \mu_3$, then use $\psi$ to reduce to the
      cases which have $\mu_3 \le \mu_2 \le \mu_1$, and finally use
      left-multiplication by the appropriate $w$ to make sure $\nu = (i,j,k)$
      is strictly dominant.  So all cases where the permutation part of $x$ is
      the identity reduce to the cases where $\nu$ is strictly dominant,
      \begin{equation*}
        x = \begin{pmatrix}
          \e^d & 0 & 0 \\
          0 & \e^e & 0 \\
          0 & 0 & \e^f
        \end{pmatrix},
      \end{equation*}
      and $f \le e \le d$.

    \item The permutation part of $x$ is a transposition.  Since $\phi(s_2) =
      \phi^2(s_1) = \eta$, we can use $\phi$ to reduce to the cases where the
      permutation part of $x$ is $\eta$.  Then we can use left-multiplication
      by the appropriate $w$ to reduce to the cases where $\nu$ is strictly
      dominant.  Now if $\nu = (i, j, k)$ and $x = \e^{(d,e,f)}\eta$, then
      $\psi(\X) = X_{x'}(\e^{\nu'})$ where $\nu' = (-k, -j, -i)$ and $x' =
      \e^{(-f, -e, -d)}\eta$.  In particular, by using $\psi$ we can make sure
      that either $e$ is maximal in $\{d,e,f\}$ or that it's not minimal and $e
      \ge j$.  Note that if $\nu$ was strictly dominant, so is $\nu'$.  So all
      cases in which the permutation part of $x$ is a transposition reduce to
      the cases where $\nu = (i, j, k)$ is strictly dominant,
      \begin{equation*}
        x = \begin{pmatrix}
          0 & 0 & \e^d \\
          0 & \e^e & 0 \\
          \e^f & 0 & 0
        \end{pmatrix},
      \end{equation*}
      and one of the following four conditions holds:
      \begin{itemize}
      \item $f \le d < e$
      \item $f \le e \le d$ and $e \ge j$
      \item $d < f \le e$
      \item $d < e < f$ and $e \ge j$.
      \end{itemize}
    \item The permutation part of $x$ is a 3-cycle.  Since $\psi(s_1s_2) =
      s_2s_1$, we can use $\psi$ to reduce to the cases where the permutation
      part of $x$ is $s_2s_1$.  Further, given $x = \e^{(\mu_1, \mu_2,
        \mu_3)}s_2s_1$, we have:
      \begin{equation*}
        \phi(x) = \e^{(\mu_3-1, \mu_1, \mu_2 + 1)} s_2s_1
      \end{equation*}
      and
      \begin{equation*}
        \phi^2(x) = \e^{(\mu_2, \mu_3-1, \mu_1 + 1)} s_2s_1.
      \end{equation*}
      Thus one of $x, \phi(x), \phi^2(x)$ is of the form $\emu s_2s_1$, where
      $\mu = (\mu_1, \mu_2, \mu_3)$ and $\mu_3 > \max(\mu_1, \mu_2)$.
      Therefore, we can reduce all cases where the permutation part of $x$ is a
      3-cycle to the case
      \begin{equation*}
        x = \begin{pmatrix}
          0 & \e^d & 0 \\
          0 & 0 & \e^e \\
          \e^f & 0 & 0
        \end{pmatrix}
      \end{equation*}
      where $d < f$ and $e < f$.  We will treat this as two cases: $d \le e <
      f$, and $e < d < f$.  Using left-multiplication by the appropriate $w$ we
      can further reduce to the cases where $\nu$ is strictly dominant.
  \end{caselist}

  Note that in all cases, we have reduced to the situation where $\nu$ is
  strictly dominant.


  \subsection{Hexagons Corresponding to Certain Elements of $X$}
  We will now look at various types of elements of $X$ that can arise in the
  cases that we will consider.  For each such element $gI$, we will compute the
  possible hexagons that could correspond to it by finding, for each $U'$ (a
  conjugate of $U_1$ by an element of $W$) the possible $\w \in \W$ such that
  $g \in U'\w I$, as described in
  Section~\ref{sec:DeterminingUOrbitForElement}.  As described in that section,
  we will repeatedly look for the leftmost entry in a given row of some matrix
  which has valuation minimal amongst the entries in that row.  We will refer
  to such an entry as a ``minimal entry.''

  \subsubsection{Hexagons of Elements of $U_1 I$}
  \label{sec:hexagons-u}
  Let
  \begin{equation*}
    g =
    \begin{pmatrix}
      1 & a & b \\
      0 & 1 & c \\
      0 & 0 & 1
    \end{pmatrix}
  \end{equation*}
  be an element of $U_1$, and assume that it satisfies the following
  conditions:
  \begin{itemize}
    \item If $\val(b) \ge 0$, then $\val(a) < 0$.
    \item If $\val(c) \ge 0$, then $\val(b-ac) < 0$.
    \item If $\val(b) \ge \val(a)$, then $\val(b-ac) < 0$.
  \end{itemize}

  If $U' = U_1$, we look at the rows of $g$ in the order 3,2,1.  There is only
  one nonzero entry in the third row.  Once we have eliminated the other
  entries in the third column, the matrix we are left with is
  \begin{equation} \label{eq:w1-row3}
    \begin{pmatrix}
      1 & a & 0 \\ 
      0 & 1 & 0 \\
      0 & 0 & 1
    \end{pmatrix}.
  \end{equation}
  There is only one nonzero entry in the second row, so in this case
  \begin{equation*}\w = \begin{pmatrix}
      1 & 0 & 0 \\
      0 & 1 & 0 \\
      0 & 0 & 1
    \end{pmatrix}.
  \end{equation*}

  If $U' = s_1 U_1 s_1$, we look at the rows of $g$ in the order 3,1,2.  After
  the first step our matrix is the one in (\ref{eq:w1-row3}).  Now we look for
  the minimal entry in the first row of this simplified matrix.  Which entry in
  the first row is minimal depends on $\val(a)$, and we see that
  \begin{equation*}
    \w = \begin{pmatrix}
      0 & \e^{\val(a)} & 0 \\
      \e^{-\val(a)} & 0 & 0 \\
      0 & 0 & 1
    \end{pmatrix}
  \end{equation*}
  if $\val(a) < 0$ and 
  \begin{equation*}\w = \begin{pmatrix}
      1 & 0 & 0 \\
      0 & 1 & 0 \\
      0 & 0 & 1
    \end{pmatrix}.
  \end{equation*}
  otherwise.

  If $U' = s_2 U_1 s_2$, we look at the rows of $g$ in the order 2,3,1.  In the
  second row, the minimal entry depends on $\val(c)$, but in either case there
  is only one nonzero entry in the third row.  We see that
  \begin{equation*}
    \w = \begin{pmatrix}
      1 & 0 & 0 \\
      0 & 0 & \e^{\val(c)} \\
      0 & \e^{-\val(c)}  & 0
    \end{pmatrix}
  \end{equation*}
  if $\val(c) < 0$ and 
  \begin{equation*}\w = \begin{pmatrix}
      1 & 0 & 0 \\
      0 & 1 & 0 \\
      0 & 0 & 1
    \end{pmatrix}.
  \end{equation*}
  otherwise.

  If $U' = s_2s_1 U_1 s_1s_2$, we look at the rows of $g$ in the order 2,1,3.
  If $\val(c) < 0$, then the minimal entry in the second row is $c$, and once
  we have eliminated the other entries in the second row and third column the
  matrix we are left with is
  \begin{equation*}\begin{pmatrix}
      1 & a - \frac{b}{c} & 0 \\
      0 & 0 & c \\
      0 & -\frac{1}{c} & 0
    \end{pmatrix}.
  \end{equation*}
  The minimal entry in the first row depends on how $\val(b-ac)$ and $\val(c)$
  compare.  If $\val(b-ac) \ge \val(c)$, then the minimal entry is 1, and
  \begin{equation*}
    \w = \begin{pmatrix}
      1 & 0 & 0 \\
      0 & 0 & \e^{\val(c)} \\
      0 & \e^{-\val(c)}  & 0
    \end{pmatrix}.
  \end{equation*}
  Otherwise, the minimal entry is $a - b/c$ and
  \begin{equation*}
    \w = \begin{pmatrix}
      0 & \e^{\val(b-ac) - \val(c)} & 0 \\
      0 & 0 & \e^{\val(c)} \\
      \e^{-\val(b-ac)} & 0 & 0
    \end{pmatrix}.
  \end{equation*}
  If, on the other hand, $\val(c) \ge 0$, then the minimal entry in the second
  row is $1$, and once we have eliminated the other entries in the second row
  and second column the matrix we are left with is
  \begin{equation*}\begin{pmatrix}
      1 & 0 & b-ac \\
      0 & 1 & 0 \\
      0 & 0 & 1
    \end{pmatrix}.
  \end{equation*}
  Since in this case, by assumption, $\val(b-ac) < 0$, $b-ac$ is the minimal
  entry in the first row, and we see that
  \begin{equation*}\w = \begin{pmatrix}
      0 & 0 & \e^{\val(b-ac)} \\
      0 & 1 & 0 \\
      \e^{-\val(b-ac)} & 0 & 0
    \end{pmatrix}.
  \end{equation*}

  If $U' = s_1s_2 U_1 s_2s_1$, we look at the rows of $g$ in the order 1,3,2.
  Since $\val(a) < 0$ if $\val(b) \ge 0$, the minimal entry in the first row is
  either $a$ or $b$.  If $\val(b) < \val(a)$, this entry is $b$, and after
  eliminating the other entries in the first row and third column we are left
  with
  \begin{equation}\label{eq:w1-row1-b-less-a}
    \begin{pmatrix}
      0 & 0 & b \\
      -\frac{c}{b} & 1 - \frac{ac}{b} & 0 \\
      -\frac{1}{b} & -\frac{a}{b} & 0
    \end{pmatrix}.
  \end{equation}
  In this case, $\w$ depends on how $\val(a)$ compares to 0.  If $\val(a) < 0$,
  we get
  \begin{equation*}\w = \begin{pmatrix}
      0 & 0 & \e^{\val(b)} \\
      \e^{-\val(a)} & 0 & 0 \\
      0 & \e^{\val(a) - \val(b)} & 0  
    \end{pmatrix}.
  \end{equation*}
  Otherwise, we get 
  \begin{equation*}\w = \begin{pmatrix}
      0 & 0 & \e^{\val(b)} \\
      0 & 1 & 0 \\
      \e^{-\val(b)} & 0 & 0
    \end{pmatrix}.
  \end{equation*}
  If, instead, $\val(b) \ge \val(a)$, then the minimal entry in the first row
  is $a$, and after eliminating the other entries in the first row and second
  column we are left with
  \begin{equation}\label{eq:w1-row1-b-ge-a}
    \begin{pmatrix}
      0 & a & 0 \\
      -\frac{1}{a} & 0 & c - \frac{b}{a} \\
      0 & 0 & 1
    \end{pmatrix}.
  \end{equation}
  There is only one nonzero entry in the third row, so we get
  \begin{equation*}
    \w = \begin{pmatrix}
      0 & \e^{\val(a)} & 0 \\
      \e^{-\val(a)} & 0 & 0 \\
      0 & 0 & 1
    \end{pmatrix}.
  \end{equation*}

  If $U' = s_1s_2s_1 U_1 s_1s_2s_1$, we look at the rows of $g$ in the order
  1,2,3. After the first step, if $\val(b) < \val(a)$, we are left with the
  matrix in (\ref{eq:w1-row1-b-less-a}).  The minimal entry in the second row
  depends on how $\val(b-ac)$ compares to $\val(c)$.  If $\val(b-ac) <
  \val(c)$, then we get
  \begin{equation*}\w = \begin{pmatrix}
      0 & 0 & \e^{\val(b)} \\
      0 & \e^{\val(b-ac) - \val(b)} & 0 \\
      \e^{-\val(b-ac)} & 0 & 0
    \end{pmatrix}.
  \end{equation*}
  Otherwise, we get
  \begin{equation*}\w = \begin{pmatrix}
      0 & 0 & \e^{\val(b)} \\
      \e^{\val(c)-\val(b)} & 0 & 0 \\
      0 & \e^{-\val(c)} & 0  
    \end{pmatrix}.
  \end{equation*}
  If, on the other hand, $\val(b) \ge \val(a)$, then we are left with the
  matrix in (\ref{eq:w1-row1-b-ge-a}).  Since in this case $\val(b-ac) <
  0$ by assumption, we get
  \begin{equation*}
    \w = \begin{pmatrix}
      0 & \e^{\val(a)} & 0 \\
      0 & 0 & \e^{\val(b-ac) - \val(a)} \\
      \e^{-\val(b-ac)} & 0 & 0
    \end{pmatrix}.
  \end{equation*}

  \subsubsection{Hexagons of Elements of $U_1 s_1 I$}
  \label{sec:hexagons-u-s1}
  Let
  \begin{equation*}
    g =
    \begin{pmatrix}
      a & 1 & b \\
      1 & 0 & c \\
      0 & 0 & 1
    \end{pmatrix}
  \end{equation*}
  be an element of $U_1 s_1$, and assume that it satisfies the following
  conditions:
  \begin{itemize}
    \item $\val(c) < 0$.
    \item If $\val(a) \le 0$, then $\val(b) < \val(a)$.
  \end{itemize}

  If $U' = U_1$, we look at the rows of $g$ in the order 3,2,1.  There is only
  one nonzero entry in the third row.  Once we have eliminated the other
  entries in the third column, the matrix we are left with is
  \begin{equation}\label{eq:w2-row3}
    \begin{pmatrix}
      a & 1 & 0 \\
      1 & 0 & 0 \\
      0 & 0 & 1
    \end{pmatrix}.
  \end{equation}
  There is only one nonzero entry in the second row, so we get 
  \begin{equation*}
    \w = \begin{pmatrix}
      0 & 1 & 0 \\
      1 & 0 & 0 \\
      0 & 0 & 1
    \end{pmatrix}.
  \end{equation*}

  If $U' = s_1 U_1 s_1$, we look at the rows of $g$ in the order 3,1,2.  After
  the first step our matrix is the one in (\ref{eq:w2-row3}).  Which entry in
  the first row is minimal depends on $\val(a)$, and we see that
  \begin{equation*}
    \w = \begin{pmatrix}
      \e^{\val(a)} & 0 & 0 \\
      0 & \e^{-\val(a)} & 0 \\
      0 & 0 & 1
    \end{pmatrix}
  \end{equation*}
  if $\val(a) \le 0$ and 
  \begin{equation*}\w = \begin{pmatrix}
      0 & 1 & 0 \\
      1 & 0 & 0 \\
      0 & 0 & 1
    \end{pmatrix}.
  \end{equation*}
  otherwise.

  If $U' = s_2 U_1 s_2$, we look at the rows of $g$ in the order 2,3,1.  Since
  by assumption $\val(c) < 0$, $c$ is the minimal entry in the second row.
  After eliminating the other entries in the second row and third column, we
  are left with
  \begin{equation}\label{eq:w2-row2}
    \begin{pmatrix}
      a-\frac{b}{c} & 1 & 0 \\
      0  & 0 & c \\
      -\frac{1}{c} & 0 & 0
    \end{pmatrix}.
  \end{equation}
  There is only one nonzero entry in the third row, so we get
  \begin{equation*}
    \w = \begin{pmatrix}
      0 & 1 & 0 \\
      0 & 0 & \e^{\val(c)} \\
      \e^{-\val(c)} & 0 & 0
    \end{pmatrix}.
  \end{equation*}

  If $U' = s_2s_1 U_1 s_1s_2$, we look at the rows of $g$ in the order 2,1,3.
  After the first step our matrix is the one in (\ref{eq:w2-row2}).  The
  minimal entry in the first row depends on how $\val(b-ac)$ compares to
  $\val(c)$.  If $\val(b-ac) \le \val(c)$, then the minimal entry is the first
  one, and we get
  \begin{equation*}
    \w = \begin{pmatrix}
      \e^{\val(b-ac)-\val(c)} & 0 & 0 \\
      0 & 0 & \e^{\val(c)} \\
      0 & \e^{-\val(b-ac)} & 0
    \end{pmatrix}.
  \end{equation*}
  Otherwise, the minimal entry is the second one, and we get
  \begin{equation*}
    \w = \begin{pmatrix}
      0 & 1 & 0 \\
      0 & 0 & \e^{\val(c)} \\
      \e^{-\val(c)} & 0 & 0
    \end{pmatrix}.
  \end{equation*}

  If $U' = s_1s_2 U_1 s_2s_1$, we look at the rows of $g$ in the order 1,3,2.
  Since by assumption either $\val(a) > 0$ or $\val(b) < \val(a)$, the minimal
  entry is either 1 or $b$.  If $\val(b) < 0$, we eliminate the other entries
  in the first row and third column and are left with
  \begin{equation}\label{eq:w2-row1-b-less-zero}
    \begin{pmatrix}
      0 & 0 & b \\
      1-\frac{ac}{b} & -\frac{c}{b} & 0 \\
      -\frac{a}{b} & -\frac{1}{b} & 0
    \end{pmatrix}.
  \end{equation}
  In this case, we see that
  \begin{equation*}
    \w = \begin{pmatrix}
      0 & 0 & \e^{\val(b)} \\
      0 & \e^{-\val(a)} & 0  \\
      \e^{\val(a)-\val(b)} & 0 & 0
    \end{pmatrix}
  \end{equation*}
  if $\val(a) \le 0$ and 
  \begin{equation*}
    \w = \begin{pmatrix}
      0 & 0 & \e^{\val(b)} \\
      1 & 0 & 0  \\
      0 & \e^{-\val(b)} & 0
    \end{pmatrix}
  \end{equation*}
  if $\val(a) > 0$.  If, on the other hand, $\val(b) \ge 0$, then 1 is the
  minimal entry in the first row, and after eliminating the other entries in
  that row and the second column we are left with
  \begin{equation}\label{eq:w2-row1-b-ge-zero}
    \begin{pmatrix}
      0 & 1 & 0 \\
      1 & 0 & c \\
      0 & 0 & 1
    \end{pmatrix}.
  \end{equation}
  Since there is only one nonzero entry in the third row, we get 
  \begin{equation*}\w = \begin{pmatrix}
      0 & 1 & 0 \\
      1 & 0 & 0 \\
      0 & 0 & 1
    \end{pmatrix}.
  \end{equation*}

  If $U' = s_1s_2s_1 U_1 s_1s_2s_1$, we look at the rows of $g$ in the order
  1,2,3.  After the first step, if $\val(b) < 0$ we are left with the matrix in
  (\ref{eq:w2-row1-b-less-zero}).  The minimal entry in the second row depends
  on how $\val(b-ac)$ compares to $\val(c)$.  We get
  \begin{equation*}
    \w = \begin{pmatrix}
      0 & 0 & \e^{\val(b)} \\
      \e^{\val(b-ac)-\val(b)} & 0 & 0  \\
      0 & \e^{-\val(b-ac)} & 0
    \end{pmatrix}
  \end{equation*}
  if $\val(b-ac) \le \val(c)$ and 
  \begin{equation*}
    \w = \begin{pmatrix}
      0 & 0 & \e^{\val(b)} \\
      0 & \e^{\val(c) - \val(b)} & 0  \\
      \e^{-\val(c)} & 0 & 0
    \end{pmatrix}
  \end{equation*}
  if $\val(b-ac) > \val(c)$.  If, on the other hand, $\val(b) \ge 0$, we are
  left with the matrix in (\ref{eq:w2-row1-b-ge-zero}).  Since $\val(c) < 0$ by
  assumption, in this case we get
  \begin{equation*}
    \w = \begin{pmatrix}
      0 & 1 & 0 \\
      0 & 0 & \e^{\val(c)} \\
      \e^{-\val(c)} & 0 & 0
    \end{pmatrix}.
  \end{equation*}

  \subsubsection{Hexagons of Elements of $U_1 s_2 I$}
  \label{sec:hexagons-u-s2}
  Let
  \begin{equation*}
    g =
    \begin{pmatrix}
      1 & b & a \\
      0 & c & 1 \\
      0 & 1 & 0
    \end{pmatrix}
  \end{equation*}
  be an element of $U_1 s_2$, and assume that it satisfies the following
  conditions:
  \begin{itemize}
    \item $\val(a) < 0$.
    \item $\val(b-ac) < \val(c)$.
    \item $\val(b-ac) < 0$.
  \end{itemize}

  If $U' = U_1$, we look at the rows of $g$ in the order 3,2,1.  There is only
  one nonzero entry in the third row.  Once we have eliminated the other
  entries in the second column, the matrix we are left with is
  \begin{equation}\label{eq:w3-row3}
    \begin{pmatrix}
      1 & 0 & a \\
      0 & 0 & 1 \\
      0 & 1 & 0
    \end{pmatrix}.
  \end{equation}
  There is only one nonzero entry in the second row, so we get 
  \begin{equation*}
    \w = \begin{pmatrix}
      1 & 0 & 0 \\
      0 & 0 & 1 \\
      0 & 1 & 0
    \end{pmatrix}.
  \end{equation*}

  If $U' = s_1 U_1 s_1$, we look at the rows of $g$ in the order 3,1,2.  After
  the first step our matrix is the one in (\ref{eq:w3-row3}).  Since
  $\val(a) < 0$, the lowest-valuation entry in the first row is $a$, and we get
  \begin{equation*}
    \w = \begin{pmatrix}
      0 & 0 & \e^{\val(a)} \\
      \e^{-\val(a)} & 0 & 0 \\
      0 & 1 & 0
    \end{pmatrix}.
  \end{equation*}

  If $U' = s_2 U s_2$, we look at the rows of $g$ in the order 2,3,1.  In the
  second row the minimal entry depends on the valuation of $c$.  If $\val(c)
  \le 0$, then the minimal entry is $c$.  Once we have eliminated the other
  entries in the second row and second column, the matrix we are left with is
  \begin{equation}
    \label{eq:w3-row2-c-min}
    \begin{pmatrix}
      1 & 0 & a - \frac{b}{c} \\ 
      0 & c & 0 \\
      0 & 0 & -\frac{1}{c}
      \end{pmatrix}.
  \end{equation}
  If $\val(c) > 0$, then the minimal entry in the second row is 1.  Once we
  have eliminated the entries in the second row and third column, the matrix we
  are left with is
  \begin{equation}
    \label{eq:w3-row2-1-min}
    \begin{pmatrix}
      1 & b-ac & 0\\ 
      0 & 0 & 1 \\
      0 & 1 & 0
      \end{pmatrix}.
  \end{equation}
  In either case, there is only one nonzero entry in the third row so we get
  \begin{equation*}
    \w = \begin{pmatrix}
      1 & 0 & 0 \\
      0 & \e^{\val(c)} & 0\\
      0 & 0 & \e^{-\val(c)}
    \end{pmatrix}
  \end{equation*}
  if $\val(c) \le 0$ and
  \begin{equation*}
    \w = \begin{pmatrix}
      1 & 0 & 0 \\
      0 & 0 & 1\\
      0 & 1 & 0
    \end{pmatrix}
  \end{equation*}
  if $\val(c) > 0$.

  If $U' = s_2s_1 U s_1s_2$, we look at the rows of $g$ in the order 2,1,3.  As
  for $U' = s_2Us_2$, after the first step we are left with either the matrix
  in (\ref{eq:w3-row2-c-min}) or the matrix in (\ref{eq:w3-row2-1-min}),
  depending on $\val(c)$.  If $\val(c) \le 0$, then because $\val(b-ac) <
  \val(c)$ the minimal entry in the first row is $a-b/c$.  So in this case
  \begin{equation*}
    \w = \begin{pmatrix}
      0 & 0 & \e^{\val(b-ac) - \val(c)}\\
      0 & \e^{\val(c)} & 0 \\
      \e^{-\val(b-ac)} & 0 & 0
    \end{pmatrix}.
  \end{equation*}
  If $\val(c) > 0$, then, because $\val(b-ac) < 0$, the minimal entry in the
   first row is $b-ac$.  So in this case
  \begin{equation*}
    \w =  \begin{pmatrix}
      0 & e^{\val(b-ac)} & 0\\ 
      0 & 0 & 1 \\
      e^{-\val(b-ac)} & 0 & 0
      \end{pmatrix}.
  \end{equation*}

  If $U' = s_1s_2 U s_2s_1$, we look at the rows of $g$ in the order 1,3,2.
  Since $\val(a) < 0$, the minimal entry must be $a$ or $b$.  If it is $a$,
  then once we have eliminated the other entries in the first row and third
  column the matrix we are left with is
  \begin{equation}
    \label{eq:w3-row1-a-min}
    \begin{pmatrix}
      0 & 0 & a \\ 
      -\frac{1}{a} & c - \frac{b}{a} & 0 \\
      0 & 1 & 0
    \end{pmatrix}.
  \end{equation}
  There is only one nonzero entry in the third row, so in this case 
  \begin{equation*}
    \w =  \begin{pmatrix}
      0 & 0 & \e^{\val(a)} \\ 
      \e^{-\val(a)} & 0 & 0 \\
      0 & 1 & 0 
      \end{pmatrix}.
  \end{equation*}
  If the minimal entry in the first row is $b$, then once we have eliminated
  the other entries in the first row and second column, the matrix we are left
  with is
  \begin{equation}
    \label{eq:w3-row1-b-min}
    \begin{pmatrix}
      0 & b & 0 \\ 
      -\frac{c}{b} & 0 & 1 - \frac{ac}{b} \\
      -\frac{1}{b} & 0 & -\frac{a}{b}
    \end{pmatrix}.
  \end{equation}
  Since $\val(a) < 0$, the minimal entry in the third row is $-a/b$, so in this
  case
  \begin{equation*}
    \w =  \begin{pmatrix}
      0 & \e^{\val(b)} & 0\\ 
      \e^{-\val(a)} & 0 & 0 \\
      0 & 0 & \e^{\val(a) - \val(b)}
      \end{pmatrix}.
  \end{equation*}

  If $U' = s_1s_2s_1 U s_1s_2s_1$, we look at the rows of $g$ in the order
  1,2,3. After the first step we are left with the matrix in
  (\ref{eq:w3-row1-a-min}) if $\val(a) < \val(b)$ or the matrix in
  (\ref{eq:w3-row1-b-min}) if $\val(a) \ge \val(b)$.  In the former case, since
  $\val(b-ac) < 0$, we have
  \begin{equation*}
    \w =  \begin{pmatrix}
      0 & 0 & \e^{\val(a)} \\ 
      0 & \e^{\val(b-ac) - \val(a)} & 0 \\
      \e^{-\val(b-ac)} & 0 & 0 
      \end{pmatrix}.
  \end{equation*}
  In the latter case, since $\val(b-ac) < c$, we have
  \begin{equation*}
    \w =  \begin{pmatrix}
      0 & \e^{\val(b)} & 0\\ 
      0 & 0 & \e^{\val(b-ac) - \val(b)} \\
      \e^{-\val(b-ac)} & 0 & 0
      \end{pmatrix}.
  \end{equation*}  
  
  \subsubsection{Hexagons of Elements of $U_1 s_2 s_1 I$}
  \label{sec:hexagons-u-s2-s1}
  Let
  \begin{equation*}
    g =
    \begin{pmatrix}
      b & 1 & a \\
      c & 0 & 1 \\
      1 & 0 & 0
    \end{pmatrix}
  \end{equation*}
  be an element of $U_1 s_2 s_1$, and assume that it satisfies the following
  conditions:
  \begin{itemize}
    \item $\val(c) \le 0$.
    \item $\val(a) \ge 0$.
    \item $\val(b) \ge \val(c)$.
  \end{itemize}
  Since $\val(a) \ge 0$, it can be eliminated by the right-action of $I$
  without changing anything else in $g$, so we can assume that $a = 0$.

  If $U' = U_1$, we look at the rows of $g$ in the order 3,2,1.  If $U' = s_1
  U_1 s_1$, we look at the rows of $g$ in the order 3,1,2.  In either case,
  there is only one nonzero entry in the third row, and once we have eliminated
  the other entries in the first column we are left with
  \begin{equation*}
    \w = \begin{pmatrix}
      0 & 1 & 0 \\
      0 & 0 & 1 \\
      1 & 0 & 0
    \end{pmatrix}.
  \end{equation*}

  If $U' = s_2 U_1 s_2$, we look at the rows of $g$ in the order 2,3,1.  If $U'
  = s_2s_1 U_1 s_1s_2$, we look at the rows of $g$ in the order 2,1,3.  In
  either case, by assumption $\val(c) \le 0$, so $c$ is the minimal entry in
  the second row.  After eliminating the other entries in the second row and
  first column, we are left with
  \begin{equation}\label{eq:w5-row2}
    \begin{pmatrix}
      0 & 1 & -\frac{b}{c} \\
      c & 0 & 0 \\
      0 & 0 & -\frac{1}{c}
    \end{pmatrix}.
  \end{equation}
  Since by assumption $\val(b) \ge \val(c)$, the $-b/c$ entry can be eliminated
  by the right-action of $I$ without changing anything else in $g$.  So we see
  that in both of these cases 
  \begin{equation*}
    \w = \begin{pmatrix}
      0 & 1 & 0 \\
      \e^{\val(c)} & 0 & 0 \\
      0 & 0 & \e^{-\val(c)}
    \end{pmatrix}.
  \end{equation*}

  If $U' = s_1s_2 U_1 s_2s_1$, we look at the rows of $g$ in the order 1,3,2.
  The minimal entry in the first row depends on the valuation of $b$.  If
  $\val(b) \le 0$, then $b$ is the minimal entry, and after eliminating the
  other entries in the first row and first column we are left with
  \begin{equation}\label{eq:w5-row1-b-le-zero}
    \begin{pmatrix}
      b & 0 & 0 \\
      0 & -\frac{c}{b} & 1 \\
      0& -\frac{1}{b} & 0
    \end{pmatrix}.
  \end{equation}
  There is only one nonzero entry in the third row, so in this case
  \begin{equation*}
    \w = \begin{pmatrix}
      \e^{\val(b)} & 0 & 0 \\
      0 & 0 & 1 \\
      0 & \e^{-\val(b)} & 0
    \end{pmatrix}.
  \end{equation*}
  If, on the other hand, $\val(b) > 0$, then the minimal entry in the first row
  is $1$, and after eliminating the $b$ we are left with
  \begin{equation}\label{eq:w5-row1-b-greater-zero}
    \begin{pmatrix}
      0 & 1 & 0 \\
      c & 0 & 1 \\
      1 & 0 & 0
    \end{pmatrix}.
  \end{equation}
  There is only one nonzero entry in the third row, so 
  \begin{equation*}
    \w = \begin{pmatrix}
      0 & 1 & 0 \\
      0 & 0 & 1 \\
      1 & 0 & 0
    \end{pmatrix}.
  \end{equation*}

  If $U' = s_1s_2s_1 U_1 s_1s_2s_1$, we look at the rows of $g$ in the order
  1,2,3.  After the first step, if $\val(b) \le 0$ we are left with the matrix
  in (\ref{eq:w5-row1-b-le-zero}).  Since $\val(c) \le \val(b)$, the minimal
  entry in the second row is $-c/b$, and we get
  \begin{equation*}
    \w = \begin{pmatrix}
      \e^{\val(b)} & 0 & 0 \\
      0 & \e^{\val(c) - \val(b)} & 0 \\
      0 & 0 & \e^{-\val(c)}
    \end{pmatrix}.
  \end{equation*}
  If, on the other hand, $\val(b) > 0$, we are left with the matrix in
  (\ref{eq:w5-row1-b-greater-zero}).  Since $\val(c) \le 0$ by assumption, the
  minimal entry in the second row is $c$ and we get
  \begin{equation*}
    \begin{pmatrix}
      0 & 1 & 0 \\
      \e^{\val(c)} & 0 & 0 \\
      0 & 0 & \e^{-\val(c)} 
    \end{pmatrix}.
  \end{equation*}

  \subsubsection{Hexagons of Elements of $U_1 s_1 s_2 s_1 I$}
  \label{sec:hexagons-u-s1-s2-s1}
  Let
  \begin{equation*}
    g =
    \begin{pmatrix}
      b & a & 1 \\
      c & 1 & 0 \\
      1 & 0 & 0
    \end{pmatrix}
  \end{equation*}
  be an element of $U_1 s_1s_2s_1$, and assume that it satisfies the following
  conditions:
  \begin{itemize}
    \item If $\val(b) > 0$, then $\val(a) \le 0$.
    \item $\val(b-ac) \le \val(c)$.
    \item $\val(b-ac) \le 0$.
  \end{itemize}

  If $U' = U_1$, we look at the rows of $g$ in the order 3,2,1.  There is only
  one nonzero entry in the third row.  Once we have eliminated the other
  entries in the first column, the matrix we are left with is
  \begin{equation} \label{eq:w4-row3}
    \begin{pmatrix}
      0 & a & 1 \\ 
      0 & 1 & 0 \\
      1 & 0 & 0
    \end{pmatrix}.
  \end{equation}
  There is only one nonzero entry in the second row, so in this case
  \begin{equation*}\w = \begin{pmatrix}
      0 & 0 & 1 \\
      0 & 1 & 0 \\
      1 & 0 & 0
    \end{pmatrix}.
  \end{equation*}

  If $U' = s_1 U_1 s_1$, we look at the rows of $g$ in the order 3,1,2.  After
  the first step our matrix is the one in (\ref{eq:w4-row3}).  Which entry in
  the first row is minimal depends on $\val(a)$, and we see that
  \begin{equation*}
    \w = \begin{pmatrix}
      0 & \e^{\val(a)} & 0 \\
      0 & 0 & \e^{-\val(a)}  \\
      1 & 0 & 0
    \end{pmatrix}
  \end{equation*}
  if $\val(a) \le 0$ and 
  \begin{equation*}\w = \begin{pmatrix}
      0 & 0 & 1 \\
      0 & 1 & 0 \\
      1 & 0 & 0
    \end{pmatrix}.
  \end{equation*}
  otherwise.
  
  If $U' = s_2 U_1 s_2$, we look at the rows of $g$ in the order 2,3,1.  In the
  second row, the minimal entry depends on $\val(c)$, but in either case there
  is only one nonzero entry in the third row.  We see that
  \begin{equation*}
    \w = \begin{pmatrix}
      0 & 0 & 1 \\
      \e^{\val(c)} & 0 & 0\\
      0 & \e^{-\val(c)}  & 0
    \end{pmatrix}
  \end{equation*}
  if $\val(c) \le 0$ and 
  \begin{equation*}\w = \begin{pmatrix}
      0 & 0 & 1 \\
      0 & 1 & 0 \\
      1 & 0 & 0
    \end{pmatrix}.
  \end{equation*}
  otherwise.

  If $U' = s_2s_1 U_1 s_1s_2$, we look at the rows of $g$ in the order 2,1,3.
  If $\val(c) \le 0$, then the minimal entry in the second row is $c$, and once
  we have eliminated the other entries in the second row and third column the
  matrix we are left with is
  \begin{equation*}\begin{pmatrix}
      0 & a - \frac{b}{c} & 1 \\
      c & 0 & 0 \\
      0 & -\frac{1}{c} & 0
    \end{pmatrix}.
  \end{equation*}
  Since $\val(b-ac) \le \val(c)$, the minimal entry in the first row is $a -
  b/c$, and
  \begin{equation*}
    \w = \begin{pmatrix}
      0 & \e^{\val(b-ac) - \val(c)} & 0 \\
      \e^{\val(c)} & 0 & 0 \\
      0 & 0 & \e^{-\val(b-ac)}
    \end{pmatrix}.
  \end{equation*}
  If, on the other hand, $\val(c) > 0$, then the minimal entry in the second
  row is $1$, and once we have eliminated the other entries in the second row
  and second column the matrix we are left with is
  \begin{equation*}\begin{pmatrix}
      b-ac & 0 & 1 \\
      0 & 1 & 0 \\
      1 & 0 & 0
    \end{pmatrix}.
  \end{equation*}
  Since in this case, by assumption, $\val(b-ac) \le 0$, $b-ac$ is the minimal
  entry in the first row, and we see that
  \begin{equation*}\w = \begin{pmatrix}
      \e^{\val(b-ac)}  & 0 & 0\\
      0 & 1 & 0 \\
      0 & 0 & \e^{-\val(b-ac)} 
    \end{pmatrix}.
  \end{equation*}

  If $U' = s_1s_2 U_1 s_2s_1$, we look at the rows of $g$ in the order 1,3,2.
  Since $\val(a) \le 0$ if $\val(b) > 0$, the minimal entry in the first row is
  either $a$ or $b$.  If $\val(b) \le \val(a)$, this entry is $b$, and after
  eliminating the other entries in the first row and third column we are left
  with
  \begin{equation}\label{eq:w4-row1-b-le-a}
    \begin{pmatrix}
      b & 0 & 0 \\
      0 & 1 - \frac{ac}{b} & -\frac{c}{b}  \\
      0 & -\frac{a}{b} & -\frac{1}{b}
    \end{pmatrix}.
  \end{equation}
  In this case, $\w$ depends on how $\val(a)$ compares to 0.  If $\val(a) \le
  0$, we get
    \begin{equation*}\w = \begin{pmatrix}
      \e^{\val(b)} & 0 &  0\\
      0 & 0 & \e^{-\val(a)}  \\
      0 & \e^{\val(a) - \val(b)} & 0  
    \end{pmatrix}.
  \end{equation*}
  Otherwise, we get 
  \begin{equation*}\w = \begin{pmatrix}
      \e^{\val(b)} & 0 &  0\\
      0 & 1 & 0 \\
      0 & 0 & \e^{-\val(b)} 
    \end{pmatrix}.
  \end{equation*}
  If, instead, $\val(b) > \val(a)$, then the minimal entry in the first row is
  $a$, and after eliminating the other entries in the first row and second
  column we are left with
  \begin{equation}\label{eq:w4-row1-b-g-a}
    \begin{pmatrix}
      0 & a & 0 \\
      c - \frac{b}{a} & 0 & -\frac{1}{a} \\
      1 & 0 & 0
    \end{pmatrix}.
  \end{equation}
  There is only one nonzero entry in the third row, so we get
  \begin{equation*}
    \w = \begin{pmatrix}
      0 & \e^{\val(a)} & 0 \\
      0 & 0 & \e^{-\val(a)} \\
      1 & 0 & 0
    \end{pmatrix}.
  \end{equation*}

  If $U' = s_1s_2s_1 U_1 s_1s_2s_1$, we look at the rows of $g$ in the order
  1,2,3. After the first step, if $\val(b) \le \val(a)$, we are left with the
  matrix in (\ref{eq:w4-row1-b-le-a}).  Since $\val(b-ac) \le \val(c)$, we get
  \begin{equation*}\w = \begin{pmatrix}
      \e^{\val(b)} & 0 & \\
      0 & \e^{\val(b-ac) - \val(b)} & 0 \\
      0 & 0 & \e^{-\val(b-ac)} 
    \end{pmatrix}.
  \end{equation*}
  If, on the other hand, $\val(b) > \val(a)$, then we are left with the
  matrix in (\ref{eq:w4-row1-b-g-a}).  Since in this case $\val(b-ac) \le
  0$ by assumption, we get
  \begin{equation*}
    \w = \begin{pmatrix}
      0 & \e^{\val(a)} & 0 \\
      \e^{\val(b-ac) - \val(a)}  & 0 & 0\\
      0 & 0 & \e^{-\val(b-ac)} 
    \end{pmatrix}.
  \end{equation*}

  \subsection{Stratifications of $\X$}
  For each of our possible values of $x$ and $\nu$, we will examine the
  intersections $\X \cap U_1 w I$ for various $w \in W$.  We will determine the
  possible hexagons that correspond to points of the intersection and use those
  to show that either Theorem~\ref{thm:Triviality-For-Nice-Stratification} or
  Theorem~\ref{thm:Triviality-For-One-Intersection} applies to $\X$.

  For every point in $\X \cap U_1 wI$, we can pick a representative $g \in
  U_1 w$.  Let
  \begin{equation*}
    g = \begin{pmatrix}
      1 & a & b \\
      0 & 1 & c \\
      0 & 0 & 1
    \end{pmatrix}w
  \end{equation*}
  for some $a,b,c \in L$.  Then
  \begin{equation}
    \label{eq:image-matrix}
    h = g^{-1}\en\s(g) = w^{-1}\begin{pmatrix}
      \e^i & \alpha & \beta \\
      0 & \e^j & \gamma \\
      0 & 0 & \e^k
    \end{pmatrix}w
  \end{equation} where
  \begin{align}
    \label{eq:alpha}
    \alpha & = \e^i\s(a) - \e^j a \\
    \label{eq:beta}
    \beta & = \e^i\s(b) - \e^k b - a(\e^j \s(c) - \e^k c) \\
    \label{eq:altbeta}
          &= \e^i\s(b) - \e^k b - a\gamma \\
    \label{eq:gamma}
    \gamma & = \e^j \s(c) - \e^k c. \\
    \noalign{\noindent Note that}
    \label{eq:minor}
    \beta\e^j - \alpha\gamma &= \e^{i+j}\s(b) - \e^{k+j} b - \e^i\s(a)\gamma
  \end{align}

  We will now look at each of the cases that we did not eliminate in
  Section~\ref{sec:Reduction-Steps}.

  \subsubsection{$x = \e^{(d, e, f)}s_2 s_1$, with $d \le e < f$}
  Let
  \begin{equation}
    x = \e^{(d, e, f)}s_1 s_2 s_1 = \begin{pmatrix}
        0 & \e^d & 0 \\
        0 & 0 & \e^e  \\
        \e^f & 0 & 0
      \end{pmatrix}
  \end{equation}
  where $d \le e < f$, with $d + e + f = 0$.  Let $\nu = (i, j, k)$
  with $i > j > k$ and $i + j + k = 0$.  We will show that the intersection $\X
  \cap U_1 wI$ is nonempty for only one value of $w \in W$, and that
  Theorem~\ref{thm:Triviality-For-One-Intersection} applies.

  Since $d \le e < f$, by
  Theorem~\ref{thm:DeterminingIwahoriDoubleCosetForElement} the entry in the
  first row and second column of $h$ must have valuation $d$ and the
  determinant of the $2 \times 2$ minor in the top right-hand corner must have
  valuation $d + e$.  That means that $w$ can only be one of $1$ and $s_2$,
  since for all other values of $w$ either that entry or the determinant of
  that minor is 0.

  If $w = 1$, then
  \begin{equation*}
    h = \begin{pmatrix}
      \e^i & \alpha & \beta \\
      0 & \e^j & \gamma \\
      0 & 0 & \e^k
    \end{pmatrix}.
  \end{equation*}
  Thus $\val(\alpha) = d$.  The valuation of the determinant of the
  \begin{equation*}
    \begin{pmatrix}\alpha & \beta \\ 0 & \e^k\end{pmatrix}
  \end{equation*}
  minor must be greater than $d + e$, so
  \begin{align*}
    \val(\alpha\e^k) & > d + e \\
    d + k &> d + e \\
    k & > e.
  \end{align*}

  If $w = s_2$, then
  \begin{equation*}
    h = \begin{pmatrix}
      \e^i & \beta & \alpha \\
      0 & \e^k & 0 \\
      0 & \gamma & \e^j
    \end{pmatrix}
  \end{equation*}
  The valuation of the determinant of the minor in the top right-hand corner
  must be $d + e$.  This gives us
  \begin{align*}
    \val(\alpha\e^k) &=  d + e  \\
    \noalign{\noindent but $d \le\val(\alpha)$, so}
    d+ k &\le d +e \\
    k &\le e
  \end{align*}
  
  Therefore, once $k$ and $e$ are fixed there is only one possible value of $w$
  for which $\X \cap U_1 wI$ might be nonempty.  If $k > e$ the intersection is
  only nonempty if $w = 1$, and if $k \le e$ the intersection is only nonempty
  if $w = s_2$.

  We now look at these two possibilities.

  \subsubsubsection{The case $k > e$}
  
  In this case, $\X \cap U_1wI$ is nonempty only if $w = 1$, so
  \begin{equation*}
    h = \begin{pmatrix}
      \e^i & \alpha & \beta \\
      0 & \e^j & \gamma \\
      0 & 0 & \e^k
    \end{pmatrix}.
  \end{equation*}
  By Theorem~\ref{thm:DeterminingIwahoriDoubleCosetForElement}, the necessary
  conditions for $h$ to be in $IxI$ include
  \begin{align}
    \label{eq:3-d-min-k-g-e-val-alpha}
    \val(\alpha) &= d \\
    \label{eq:3-d-min-k-g-e-val-beta}
    \val(\beta) &\ge d \\
    \label{eq:3-d-min-k-g-e-val-minor}
    \val(\alpha\gamma - \beta\e^j) &= d + e.
  \end{align}
  We will use these to determine the valuations of $a$, $b$, $c$, and $b -
  ac$.

  From (\ref{eq:alpha})~and~(\ref{eq:3-d-min-k-g-e-val-alpha}) and because $j > k
  > e \ge d$ by assumption we know that
  \begin{equation}
    \label{eq:3-d-min-k-g-e-val-a}
    \val(a) = d - j < 0.
  \end{equation}

  From (\ref{eq:3-d-min-k-g-e-val-beta}) and our assumption that $k > e$ we
  know that $\val(\beta\e^j) \ge d + j > d + k > d + e$.  Then
  (\ref{eq:gamma}),~(\ref{eq:3-d-min-k-g-e-val-minor})~and~(\ref{eq:3-d-min-k-g-e-val-alpha})
  tell us that
  \begin{align}
    \val(\alpha\gamma) &= d + e \nonumber\\
    d + \val(\gamma) &= d + e \nonumber \\
    \val(\gamma) &= e \nonumber \\
    \label{eq:3-d-min-k-g-e-val-c}
    \val(c) &= e - k < 0
  \end{align}

  Using (\ref{eq:3-d-min-k-g-e-val-a})~and~(\ref{eq:3-d-min-k-g-e-val-c}) we
  see that
  \begin{equation}
    \label{eq:3-d-min-k-g-e-trailing-term}
    \val(\e^j a\s(c)) = d + (e - k) < d
  \end{equation}
  and
  \begin{equation*}
    \val(\e^k a c) = d + e - j < d + (e - k) = \val(\e^j a\s(c)).
  \end{equation*}
  By~(\ref{eq:3-d-min-k-g-e-val-beta}), and~(\ref{eq:beta}), this means that
  \begin{align}
    \val(\e^i\s(b) - e^k b) &= d + e - j \nonumber \\
    \val(\e^k b) &= d + e - j \nonumber \\
    \label{eq:3-d-min-k-g-e-val-b}
    \val(b) &= (d-j) + (e - k) = i - f.
  \end{align}
  Comparing (\ref{eq:3-d-min-k-g-e-val-b}) to (\ref{eq:3-d-min-k-g-e-val-a}) we
  see that
  \begin{equation}
    \label{eq:3-d-min-k-g-e-val-b-l-a}
    \val(b) < \val(a),
  \end{equation}
  since $e < k$.

  Now from (\ref{eq:3-d-min-k-g-e-val-b}) we see that
  \begin{equation*}
    \val(\e^i \s(b)) = d + (e-k) + (i - j) > d + (e - k).
  \end{equation*}
  Since $\val(\beta) \ge d$, by (\ref{eq:3-d-min-k-g-e-trailing-term}) we must
  have
  \begin{align}
    \val(\e^k(b-ac)) &= d + (e-k) \nonumber \\
    \label{eq:3-d-min-k-g-e-val-b-m-ac}
    \val(b - ac) &= d + e - 2k = (d - k) + (e - k) < \val(c),
  \end{align}
  where the last inequality follows because $d - k \le e - k < 0$.

  Now $\val(a) < 0$ and $\val(b-ac) < 0$.  So the conditions of
  Section~\ref{sec:hexagons-u} are satisfied.  Since $\val(c) < 0$ and
  $\val(b-ac) < \val(c)$, the hexagon for $g$ has the vertices
  \begin{equation*}
    \heximages
	{\begin{pmatrix}
	    1 & 0 & 0 \\
	    0 & 1 & 0 \\
	    0 & 0 & 1
	\end{pmatrix}}
	{\begin{pmatrix}
	    0 & \e^{d-j} & 0 \\
	    \e^{j-d} & 0 & 0 \\
	    0 & 0 & 1
	\end{pmatrix}}
	{\begin{pmatrix}
	    1 & 0 & 0 \\
	    0 & 0 & \e^{e-k} \\
	    0 & \e^{k-e} & 0
	\end{pmatrix}}
	{\begin{pmatrix}
	    0 & \e^{d-k} & 0 \\
	    0 & 0 & \e^{e-k} \\
	    \e^{2k+f} & 0 & 0
	\end{pmatrix}}
	{\begin{pmatrix}
	    0 & 0 & \e^{i-f} \\
	    \e^{j-d} & 0 & 0 \\
	    0 & \e^{k-e} & 0
	\end{pmatrix}}
	{\begin{pmatrix}
	    0 & 0 & \e^{i-f} \\
	    0 & \e^{j-k} & 0 \\
	    \e^{2k+f} & 0 & 0
	\end{pmatrix}}.
  \end{equation*}
  This hexagon is completely determined by our choice of $i,j,k,d,e,f$.  So all
  elements of $\X \cap U_1 I$ have the same corresponding hexagon, and
  Theorem~\ref{thm:Triviality-For-One-Intersection} applies.

  \subsubsubsection{The case $k \le e$}
  In this case, $\X \cap U_1wI$ is nonempty only if $w = s_2$, so
  \begin{equation*}
    h = \begin{pmatrix}
      \e^i & \beta & \alpha \\
      0 & \e^k & 0 \\
      0 & \gamma & \e^j
    \end{pmatrix}.
  \end{equation*}
  By Theorem~\ref{thm:DeterminingIwahoriDoubleCosetForElement}, necessary
  conditions for $h$ to be in $IxI$ include
  \begin{align}
    \label{eq:3-d-min-k-le-e-val-beta}
    \val(\beta) &= d \\
    \label{eq:3-d-min-k-le-e-k-vs-d}
    \val(\e^k) &> d \implies k > d \\
    \label{eq:3-d-min-k-le-e-i-vs-f}
    \val(\e^{j+k}) &> d + e \implies j + k > d + e \implies i < f \\
    \label{eq:3-d-min-k-le-e-val-alpha}
    \val(\alpha\e^k) &= d + e \\
    \label{eq:3-d-min-k-le-e-val-minor}
    \val(\alpha\gamma - \beta\e^j) &> d + e.
  \end{align}
  We will use these to determine the valuations of $a$, $b$, $c$, and $b -
  ac$.

  From
  (\ref{eq:alpha}),~(\ref{eq:3-d-min-k-le-e-val-alpha}),~and~(\ref{eq:3-d-min-k-le-e-i-vs-f}),
  we know that
  \begin{equation}
    \label{eq:3-d-min-k-le-e-val-a}
    \val(a) = d + e - j - k = i - f < 0.
  \end{equation}

  \begin{lemma}
    \label{lemma:3-d-min-k-le-e-val-c}
    $\val(c) \le 0$ if and only if $e \ge j$, and when this happens
    $\val(c) = j - e$.
  \end{lemma}
  \begin{proof}
    By (\ref{eq:3-d-min-k-le-e-val-minor}), $\val(\alpha\gamma) \le d + e$ if
    and only if $\val(\beta\e^j) \le d + e$.  By
    (\ref{eq:3-d-min-k-le-e-val-beta}),
    \begin{equation*}
      \val(\beta\e^j) = d + j.
    \end{equation*}
    By (\ref{eq:3-d-min-k-le-e-val-alpha})~and~(\ref{eq:gamma}),
    \begin{equation}
      \label{eq:3-d-min-k-le-e-val-alpha-gamma}
      \val(\alpha\gamma) = \val(c) + k + d + e - k = \val(c) + d + e.
    \end{equation}
    Now we see that
    \begin{align*}
      j \le e &\iff d + j \le d + e \\
              &\iff \val(c) + d + e \le d + e \\
              &\iff \val(c) \le 0.
    \end{align*}
    If $j \le e$, then, by (\ref{eq:3-d-min-k-le-e-val-minor}),
    \begin{equation*}
      \val(\alpha\gamma) = \val(\beta\e^j) = d + j,
    \end{equation*}
    and then, by (\ref{eq:3-d-min-k-le-e-val-alpha-gamma}), $\val(c) = j - e$.
  \end{proof}

  \begin{lemma}
    \label{lemma:3-d-min-k-le-e-val-b}
    $\val(b) \le \val(a)$ if and only if $e \ge i$ and when this happens
    $\val(b) = d + 2i$.
  \end{lemma}
  \begin{proof}
    From (\ref{eq:3-d-min-k-le-e-val-minor}) and (\ref{eq:minor}) we see that
    \begin{equation}
      \label{eq:3-d-min-k-le-e-expand-minor}
      \val(\e^{i+j}\s(b) - \e^{j+k}b - \e^i\s(a)\gamma) > d + e.
    \end{equation}
    Now we consider the three possible ways that $e$ can relate to $i$ and $j$.

    If $e \ge i$, then $i > j$ implies $e > j$.  In this case, by
    Lemma~\ref{lemma:3-d-min-k-le-e-val-c},
    (\ref{eq:gamma}),~and~(\ref{eq:3-d-min-k-le-e-val-a}),
    \begin{align}
      \label{eq:3-d-min-k-le-e-trailing-term-e-ge-j}
      \val(\e^i\s(a)\gamma) &= i + (d + e - j - k) + (j - e + k) \nonumber \\
                            &= d + i
    \end{align}
    Since $e \ge i$, $d + i \le d + e$.  But then
    (\ref{eq:3-d-min-k-le-e-expand-minor}) tells us that
    \begin{equation*}
      \val(\e^{i+j}\s(b) - \e^{j+k}b) = d + i.
    \end{equation*}
    Therefore
    \begin{align*}
      \val(b) &= d + i - j - k \\
              &= \val(a) + (i - e) \\
              &\le \val(a)
    \end{align*}
    since $i \le e$.  Since $-j -k = i$, in this case $\val(b) = d + 2i$.

    If $j \le e < i$, then (\ref{eq:3-d-min-k-le-e-trailing-term-e-ge-j}) still
    holds, but now $d + i > d + e$.  By (\ref{eq:3-d-min-k-le-e-expand-minor}),
    \begin{equation*}
      \val(\e^{i+j}\s(b) - \e^{j+k}b) > d + e
    \end{equation*}
    so $\val(b) > d + e - j - k = \val(a)$.

    If $e < j$, then, by Lemma~\ref{lemma:3-d-min-k-le-e-val-c}, $\val(c) > 0$.
    This means that by (\ref{eq:gamma})~and~(\ref{eq:3-d-min-k-le-e-val-a})
    \begin{align*}
      \val(\e^i\s(a)\gamma) &> i + d + e - j - k + k \\
                            &= d + e + (i - j) \\
                            &> d + e.
    \end{align*}
    So by the same argument as in the case when $j \le e < i$, $\val(b) >
    \val(a)$.
  \end{proof}

  \begin{lemma}
    \label{lemma:3-d-min-k-le-e-val-b-m-ac}
    $\val(b-ac) = d - k < 0$.
  \end{lemma}
  \begin{proof}
    From (\ref{eq:beta})~and~(\ref{eq:3-d-min-k-le-e-val-beta}) we see that
    \begin{equation}
      \label{eq:3-d-min-k-le-e-expand-beta}
      \val(\e^i\s(b) - \e^k b + \e^k ac - \e^j a\s(c)) = d
    \end{equation}
    
    If $e \ge j$, then by Lemma~\ref{lemma:3-d-min-k-le-e-val-c} and
    (\ref{eq:3-d-min-k-le-e-val-a})
    \begin{align*}
      \val(\e^k ac) &= k + (d + e - j - k) + (j - e) \nonumber \\
      \val(\e^k ac) &= d \\
      \noalign{\noindent and}
      \val(\e^j a\s(c)) &= j + (d + e - j - k) + (j - e) \nonumber \\
                        &= d + (j - k) \nonumber \\
      \val(\e^j a\s(c)) &> d.
    \end{align*}
    In this case, (\ref{eq:3-d-min-k-le-e-expand-beta}) tells us that
    $\val(\e^i\s(b) - \e^k b) \ge d$ so that $\val(\e^kb) \ge d$.  But then
    \begin{equation*}
      \val(\e^i\s(b)) \ge d + (i - k) > d.
    \end{equation*}
    Looking at (\ref{eq:3-d-min-k-le-e-expand-beta}) again we see that
    \begin{equation*}
      \val(\e^k(b - ac)) = d.
    \end{equation*}

    If $e < j$, then from
    Lemmas~\ref{lemma:3-d-min-k-le-e-val-c}~and~\ref{lemma:3-d-min-k-le-e-val-b},
    and (\ref{eq:3-d-min-k-le-e-val-a}) we see that $\val(c) > 0$ and $\val(b)
    > d + e - j - k$.  So in this case
    \begin{align*}
      \val(\e^i\s(b)) &> d + (e - k) + (i - j) \\
                      &> d
    \end{align*}
    since $e \ge k$ by assumption.  Also,
    \begin{align*}
      \val(\e^j a \s(c)) &> d + e - j - k  + j \\
                         &= d + (e - k) \\
                         &\ge d.
    \end{align*}
    By (\ref{eq:3-d-min-k-le-e-expand-beta}) again we see that
    \begin{equation*}
      \val(\e^k(b - ac)) = d.
    \end{equation*}

    So in either case, $\val(\e^k(b - ac)) = d$, which means $\val(b-ac) = d -
    k$.  By (\ref{eq:3-d-min-k-le-e-k-vs-d}), $d - k < 0$.
  \end{proof}

  Now $\val(a) < 0$, $\val(b-ac) < 0$, and $\val(b-ac) < \val(c)$.  So the
  conditions of Section~\ref{sec:hexagons-u-s2} are satisfied.  We have three
  cases:
  \begin{caselist}
    \item $e < j$.  Then $\val(c) > 0$ and $\val(b) > \val(a)$.  In this case,
      the hexagon for $g$ is
      \begin{equation*}
	\heximages
	    {\begin{pmatrix}
		1 & 0 & 0 \\
		0 & 0 & 1 \\
		0 & 1 & 0
	    \end{pmatrix}}
	    {\begin{pmatrix}
		0 & 0 & \e^{i - f} \\
		\e^{f-i} & 0 & 0 \\
		0 & 1 & 0
	    \end{pmatrix}}
	    {\begin{pmatrix}
		1 & 0 & 0 \\
		0 & 0 & 1 \\
		0 & 1 & 0
	    \end{pmatrix}}
	    {\begin{pmatrix}
		0 & e^{d-k} & 0\\ 
		0 & 0 & 1 \\
		\e^{k-d} & 0 & 0
	    \end{pmatrix}}
	    {\begin{pmatrix}
		0 & 0 & \e^{i - f} \\
		\e^{f-i} & 0 & 0 \\
		0 & 1 & 0
	    \end{pmatrix}}
	    {\begin{pmatrix}
		0 & 0 & \e^{i-f} \\ 
		0 & \e^{j-e} & 0 \\
		\e^{k-d} & 0 & 0
	    \end{pmatrix}}.
      \end{equation*}

      \item $j \le e < i$.  Then $j - e \val(c) \le 0$ and $\val(b) > \val(a)$.
      In this case, the hexagon for $g$ is
      \begin{equation*}
	\heximages
	    {\begin{pmatrix}
		1 & 0 & 0 \\
		0 & 0 & 1 \\
		0 & 1 & 0
	    \end{pmatrix}}
	    {\begin{pmatrix}
		0 & 0 & \e^{i - f} \\
		\e^{f-i} & 0 & 0 \\
		0 & 1 & 0
	    \end{pmatrix}}
	    {\begin{pmatrix}
		1 & 0 & 0 \\
		0 & \e^{j-e} & 0\\
		0 & 0 & \e^{e-j}
	    \end{pmatrix}}
	    {\begin{pmatrix}
		0 & 0 & \e^{i-f}\\
		0 & \e^{j-e} & 0 \\
		\e^{k-d} & 0 & 0
	    \end{pmatrix}}
	    {\begin{pmatrix}
		0 & 0 & \e^{i - f} \\
		\e^{f-i} & 0 & 0 \\
		0 & 1 & 0
	    \end{pmatrix}}
	    {\begin{pmatrix}
		0 & 0 & \e^{i-f} \\ 
		0 & \e^{j-e} & 0 \\
		\e^{k-d} & 0 & 0
	    \end{pmatrix}}.
      \end{equation*}

      \item $e \ge i$.  Then $j - e = \val(c) \le 0$ and $d + 2i = \val(b) \le
      \val(a)$.  In this case, the hexagon for $g$ is
      \begin{equation*}
	\heximages
	    {\begin{pmatrix}
		1 & 0 & 0 \\
		0 & 0 & 1 \\
		0 & 1 & 0
	    \end{pmatrix}}
	    {\begin{pmatrix}
		0 & 0 & \e^{i - f} \\
		\e^{f-i} & 0 & 0 \\
		0 & 1 & 0
	    \end{pmatrix}}
	    {\begin{pmatrix}
		1 & 0 & 0 \\
		0 & \e^{j-e} & 0\\
		0 & 0 & \e^{e-j}
	    \end{pmatrix}}
	    {\begin{pmatrix}
		0 & 0 & \e^{i-f}\\
		0 & \e^{j-e} & 0 \\
		\e^{k-d} & 0 & 0
	    \end{pmatrix}}
	    {\begin{pmatrix}
		0 & \e^{d+2i} & 0\\ 
		\e^{f-i} & 0 & 0 \\
		0 & 0 & \e^{e-i}
	    \end{pmatrix}}
	    {\begin{pmatrix}
		0 & \e^{d+2i} & 0\\ 
		0 & 0 & \e^{j-i} \\
		\e^{k-d} & 0 & 0
	    \end{pmatrix}}.
      \end{equation*}
  \end{caselist}
  In all three cases, the hexagon is completely determined by our choice of
  $i,j,k,d,e,f$.  So all elements of $\X \cap U_1 s_2 I$ have the same
  corresponding hexagon, and Theorem~\ref{thm:Triviality-For-One-Intersection}
  applies.

  \subsubsection{$x = \e^{(d, e, f)}s_2 s_1$, with $e < d < f$}
  Let
  \begin{equation}
    x = \e^{(d, e, f)}s_1 s_2 s_1 = \begin{pmatrix}
        0 & \e^d & 0 \\
        0 & 0 & \e^e  \\
        \e^f & 0 & 0
      \end{pmatrix}
  \end{equation}
  where $e < d < f$, with $d + e + f = 0$.  Let $\nu = (i, j, k)$
  with $i > j > k$ and $i + j + k = 0$.  We will show that the intersection $\X
  \cap U_1 wI$ is nonempty for only one value of $w \in W$, and that
  Theorem~\ref{thm:Triviality-For-One-Intersection} applies.

  Since $e < d < f$, by
  Theorem~\ref{thm:DeterminingIwahoriDoubleCosetForElement} the entry in the
  second row and third column of $h$ must have valuation $e$ and the
  determinant of the $2 \times 2$ minor in the top right-hand corner must have
  valuation $d + e$.  That means that $w$ can only be one of $1$ and $s_1$,
  since for all other values of $w$ either that entry or the determinant of
  that minor is 0.

  If $w = 1$, then
  \begin{equation*}
    h = \begin{pmatrix}
      \e^i & \alpha & \beta \\
      0 & \e^j & \gamma \\
      0 & 0 & \e^k
    \end{pmatrix}.
  \end{equation*}
  Thus $\val(\gamma) = e$.  The valuation of the determinant of the
  \begin{equation*}
    \begin{pmatrix} \e^i & \beta \\ 0 & \gamma \end{pmatrix} \end{equation*}
  minor must be greater than $d + e$.  This gives us
  \begin{align*}
    \val(\gamma\e^i) & > d + e \\
    e + i &> d + e \\
    i & > d.
  \end{align*}

  If $w = s_1$, then
  \begin{equation*}
    h = \begin{pmatrix}
      \e^j & 0 & \gamma \\
      \alpha & \e^i & \beta \\
      0 & 0 & \e^k
    \end{pmatrix}
  \end{equation*}
  The valuation of the determinant of the minor in the top right-hand corner
  must be $d + e$.  This gives us
  \begin{align*}
    \val(\gamma\e^i) &=  d + e  \\
    \noalign{\noindent but $e \le\val(\gamma)$, so}
    e+ i &\le d +e \\
    i &\le d
  \end{align*}
  
  Therefore, once $i$ and $d$ are fixed there is only one possible value of $w$
  for which $\X \cap U_1 w I$ might be nonempty.  If $i > d$ the intersection
  is only nonempty if $w = 1$, and if $i \le d$ the intersection is only
  nonempty if $w = s_1$.

  We now look at these two cases.

  \subsubsubsection{The case $i > d$}
  In this case, $\X \cap U_1 w I$ is only nonempty if $w = 1$, so
  \begin{equation*}
    h = \begin{pmatrix}
      \e^i & \alpha & \beta \\
      0 & \e^j & \gamma \\
      0 & 0 & \e^k
    \end{pmatrix}.
  \end{equation*}
  By Theorem~\ref{thm:DeterminingIwahoriDoubleCosetForElement}, the necessary
  conditions for $h$ to be in $IxI$ include
  \begin{align}
    \label{eq:case4_val_gamma}
    \val(\gamma) &= e \\
    \label{eq:case4_k_vs_e}
    \val(\e^k) &> e \implies k > e \\
    \label{eq:case4_val_beta}
    \val(\beta) &\ge e \\
    \label{eq:case4_i_vs_f}
    \val(\e^{j+k}) &> d + e \implies j + k > d + e \implies i < f \\
    \label{eq:case4_val_minor}
    \val(\alpha\gamma - \beta\e^j) &= d + e.
  \end{align}
  We will use these to determine the valuations of $a$, $b$, $c$, and $b -
  ac$.

  From (\ref{eq:gamma}), (\ref{eq:case4_val_gamma}), and
  (\ref{eq:case4_k_vs_e}), we know that
  \begin{equation}
    \label{eq:case4_val_c}
    \val(c) = e - k < 0.
  \end{equation}

  \begin{lemma}
    \label{lemma:case4_val_a}
    $\val(a) < 0$ if and only if $d < j$ and when this happens $\val(a) = d -
    j$.
  \end{lemma}
  \begin{proof}
    If $d < j$, then, by (\ref{eq:case4_val_beta}), $\val(\beta e^j) > d + e$.
    But then, by (\ref{eq:case4_val_minor}), (\ref{eq:case4_val_gamma}), and
    (\ref{eq:alpha}),
    \begin{align*}
      \val(\alpha\gamma) &= d + e \\
      \val(\alpha) &= d \\
      \val(a) &= d - j < 0.
    \end{align*}

    Conversely, if $\val(a) < 0$, then, by (\ref{eq:alpha}) and
    (\ref{eq:case4_val_gamma}), $\val(\alpha\gamma) < j + e$.  Since, by
    (\ref{eq:case4_val_beta}), $\val(\beta \e^j) \ge j + e$, it follows from
    (\ref{eq:case4_val_minor}) that
    \begin{align*}
      \val(\alpha\gamma - \beta\e^j) &< j + e \\
      d + e &< j + e \\
      d &< j.  \qedhere
    \end{align*}
  \end{proof}

  \begin{lemma}
    \label{lemma:case4_val_b}
    $\val(b) = i - f < 0$ and $\val(b) < \val(a)$.
  \end{lemma}
  \begin{proof}
    First, note that, by (\ref{eq:case4_i_vs_f}), $i - f < 0$.  So we need to
    show that $\val(b) = i - f$ and $\val(b) < \val(a)$.

    If $d < j$, then, by Lemma~\ref{lemma:case4_val_a}, $\val(a) = d
    - j$.  In this case, by (\ref{eq:case4_val_c}),
    \begin{align}
      \val(a\e^k c) &= d - j + k + e - k \nonumber \\
                    &= e + (d - j) \nonumber \\
                    &< e \nonumber \\
      \noalign{\noindent and}
      \val(ae^j\s(c)) &= d - j + j + e - k \nonumber \\
      \label{eq:case4_trailing_term_d_less_j}
                      &= e + (d - k) \\
                      &> e + (d - j). \nonumber
    \end{align}
    But by (\ref{eq:case4_val_beta})
    \begin{equation}
      \label{eq:case4_expand_beta}
      \val(\e^i\s(b) - \e^k b - a\e^j\s(c) + a\e^k c) \ge e.
    \end{equation}
    So we must have
    \begin{align*}
      \val(\e^i\s(b) - \e^k b) &= e + d - j \\
      \val(b) &= e + d - j - k \\
              &= i - f.
    \end{align*}
    In this case,
    \begin{align*}
      \val(b) - \val(a) &= i - f - (d - j) \\
                        &= i + j - (f + d) \\
                        &= e - k \\
                        &< 0 \\
      \noalign{\noindent by (\ref{eq:case4_k_vs_e}) so}
      \val(b) &< \val(a).
    \end{align*}

    If $d \ge j$, then, by Lemma~\ref{lemma:case4_val_a}, $\val(a) \ge 0$.  In
    this case, by (\ref{eq:case4_val_gamma}) and our assumption that $i >d$,
    \begin{align*}
      \val(\e^i \s(a)\gamma) &\ge e + i \\
                    &> e + d.
    \end{align*}
    But by (\ref{eq:case4_val_minor}) and (\ref{eq:minor})
    \begin{equation}
      \label{eq:case4_expand_minor}
      \val(\e^{i+j}\s(b) - \e^{j+k}b - \e^i\s(a)\gamma) = e + d.
    \end{equation}
    So we must have
    \begin{align*}
      \val(\e^{i+j}\s(b) - \e^{j+k}b) &= e + d \\
      \val(b) &= e + d - j - k \\
              &= i - f.
    \end{align*}
    Since in this case $\val(a) \ge 0$ and $\val(b) < 0$ , we see that $\val(b)
    < \val(a)$.
  \end{proof}

  \begin{lemma}
    \label{lemma:case4_val_b_m_ac}
    $\val(b - ac) < \val(c)$ if and only if $d < k$ and when this happens
    $\val(b - ac) - \val(c) = d - k$.
  \end{lemma}
  \begin{proof}
    In all cases,
    \begin{align}
      \val(\e^i\s(b)) &= i + i - f \nonumber \\
      &= i - j - k + d + e \nonumber \\
      &= e + (i - j) + (d - k) \nonumber \\
      \label{eq:case4_trailing_term2}
      &> e + (d - k).
    \end{align}
    
    If $d < k < j$, then (\ref{eq:case4_trailing_term_d_less_j}) holds.  Since
    $d < k$, $\val(ae^j\s(c)) = e + (d - k) < e$.  At the same time,
    (\ref{eq:case4_trailing_term2}) holds.  So for (\ref{eq:case4_expand_beta})
    to hold, we must have
    \begin{align*}
      \val(\e^k b - \e^k ac) &= e + (d-k) \\
      \val(b - ac)  &= e + (d - k) - k  \\
      \val(b - ac) &< e - k \\
      \noalign{\noindent and by (\ref{eq:case4_val_c})}
      \val(b - ac) &< \val(c)
    \end{align*}
    In this case, $\val(b-ac) - \val(c) = e + d - k - k - (e - k) = d - k$.

    If $k \le d < j$, then (\ref{eq:case4_trailing_term_d_less_j}) and
    (\ref{eq:case4_trailing_term2}) still hold, but now $d \ge k$ so that
    $\val(\e^i\s(b) - ae^j\s(c)) \ge e$.  So for
    (\ref{eq:case4_expand_beta}) to hold, we must have
    \begin{align*}
      \val(\e^k b - \e^k ac) &\ge e \\
      \val(b - ac) &\ge e - k  \\
      \noalign{\noindent and by (\ref{eq:case4_val_c})}
      \val(b-ac) &\ge \val(c) .
    \end{align*}

    If $d \ge j > k$, then, by Lemma~\ref{lemma:case4_val_a}, $\val(a) \ge 0$.
    In this case, $\val(a\e^j\s(c)) \ge j + (e - k) > e$ and, by
    (\ref{eq:case4_trailing_term2}), $\val(\e^i\s(b)) > e$.  Now by the same
    argument as the case when $k \le d <j$, $\val(b-ac) \ge \val(c)$.
  \end{proof}

  Now $\val(b) < 0$, $\val(c) < 0$, and $\val(b) < \val(a)$.  So the conditions
  of Section~\ref{sec:hexagons-u} are satisfied.  We have three cases:
  \begin{caselist}
    \item $d < k$.  Then $d - j = \val(a) < 0$ and $\val(b-ac) < \val(c)$, with
      $\val(b-ac) - \val(c) = d - k$.  In this case, the hexagon for $g$ is
      \begin{equation*}
	\heximages
	    {\begin{pmatrix}
		1 & 0 & 0 \\
		0 & 1 & 0 \\
		0 & 0 & 1
	    \end{pmatrix}}
	    {\begin{pmatrix}
		0 & \e^{d-j} & 0 \\
		\e^{j-d} & 0 & 0 \\
		0 & 0 & 1
	    \end{pmatrix}}
	    {\begin{pmatrix}
		1 & 0 & 0 \\
		0 & 0 & \e^{e-k} \\
		0 & \e^{k-e} & 0
	    \end{pmatrix}}
	    {\begin{pmatrix}
		0 & \e^{d-k} & 0 \\
		0 & 0 & \e^{e-k} \\
		\e^{2k+f} & 0 & 0
	    \end{pmatrix}}
	    {\begin{pmatrix}
		0 & 0 & \e^{i-f} \\
		\e^{j-d} & 0 & 0  \\
		0 & \e^{k-e} & 0
	    \end{pmatrix}}
	    {\begin{pmatrix}
		0 & 0 & \e^{i-f} \\
		0 & \e^{j-k} & 0  \\
		\e^{2k+f} & 0 & 0
	    \end{pmatrix}}
      \end{equation*}

    \item $k \le d < j$.  Then $d - j = \val(a) < 0$ and $\val(b-ac) \ge
      \val(c)$.  In this case, the hexagon for $g$ is
      \begin{equation*}
	\heximages
	    {\begin{pmatrix}
		1 & 0 & 0 \\
		0 & 1 & 0 \\
		0 & 0 & 1
	    \end{pmatrix}}
	    {\begin{pmatrix}
		0 & \e^{d-j} & 0 \\
		\e^{j-d} & 0 & 0 \\
		0 & 0 & 1
	    \end{pmatrix}}
	    {\begin{pmatrix}
		1 & 0 & 0 \\
		0 & 0 & \e^{e-k} \\
		0 & \e^{k-e} & 0
	    \end{pmatrix}}
	    {\begin{pmatrix}
		1 & 0 & 0 \\
		0 & 0 & \e^{e-k} \\
		0 & \e^{k-e} & 0
	    \end{pmatrix}}
	    {\begin{pmatrix}
		0 & 0 & \e^{i-f} \\
		\e^{j-d} & 0 & 0  \\
		0 & \e^{k-e} & 0
	    \end{pmatrix}}
	    {\begin{pmatrix}
		0 & 0 & \e^{i-f} \\
		\e^{j-d} & 0 & 0  \\
		0 & \e^{k-e} & 0
	    \end{pmatrix}}	  
      \end{equation*}

    \item $j \le d$.  Then $\val(a) \ge 0$ and $\val(b-ac) \ge \val(c)$.  In
      this case, the hexagon for $g$ is
      \begin{equation*}
	\heximages
	    {\begin{pmatrix}
		1 & 0 & 0 \\
		0 & 1 & 0 \\
		0 & 0 & 1
	    \end{pmatrix}}
	    {\begin{pmatrix}
		1 & 0 & 0 \\
		0 & 1 & 0 \\
		0 & 0 & 1
	    \end{pmatrix}}
	    {\begin{pmatrix}
		1 & 0 & 0 \\
		0 & 0 & \e^{e-k} \\
		0 & \e^{k-e} & 0
	    \end{pmatrix}}
	    {\begin{pmatrix}
		1 & 0 & 0 \\
		0 & 0 & \e^{e-k} \\
		0 & \e^{k-e} & 0
	    \end{pmatrix}}
	    {\begin{pmatrix}
		0 & 0 & \e^{i-f} \\
		0 & 1 & 0  \\
		\e^{f-i} & 0 & 0
	    \end{pmatrix}}
	    {\begin{pmatrix}
		0 & 0 & \e^{i-f} \\
		\e^{j-d} & 0 & 0  \\
		0 & \e^{k-e} & 0
	    \end{pmatrix}}
      \end{equation*}
  \end{caselist}
  In all three cases, the hexagon is completely determined by our choice of
  $i,j,k,d,e,f$.  So all elements of $\X \cap U_1 s_2 I$ have the same
  corresponding hexagon, and Theorem~\ref{thm:Triviality-For-One-Intersection}
  applies.
    
  \subsubsubsection{The case $i \le d$}
  In this case, $\X \cap U_1 w I$ is only nonempty if $w = s_1$, so
  \begin{equation*}
    h = \begin{pmatrix}
      \e^j & 0 & \gamma \\
      \alpha & \e^i & \beta \\
      0 & 0 & \e^k
    \end{pmatrix}.
  \end{equation*}
  By Theorem~\ref{thm:DeterminingIwahoriDoubleCosetForElement}, the necessary
  conditions for $h$ to be in $IxI$ include
  \begin{align}
    \label{eq:case3_val_beta}
    \val(\beta) &= e \\
    \label{eq:case3_val_gamma}
    \val(\e^i\gamma) &= d + e \\
    \label{eq:case3_val_minor}
    \val(\beta\e^j - \alpha\gamma) &> d + e.
  \end{align}
  We will use these to determine the valuations of $a$, $b$, $c$, and $a -
  b/c$.

  From (\ref{eq:gamma}) and (\ref{eq:case3_val_gamma}), and because $j < i \le
  d < f$ we see that
  \begin{equation}
    \label{eq:case3_val_c}
    \val(c) = d + e - i - k = j - f < 0.
  \end{equation}

  Since $j < i \le d$, (\ref{eq:case3_val_beta}) tells us that $\val(\e^j\beta)
  = j + e < d + e$.  But then from (\ref{eq:case3_val_minor}),
  (\ref{eq:case3_val_gamma}), (\ref{eq:alpha}) we see that
  \begin{align}
    \val(\alpha\gamma) &= j + e \nonumber \\
    \val(\alpha) &= j + e - (d + e - i) \nonumber \\
                 &= j + i - d \nonumber \\
    \label{eq:case3_val_a}
    \val(a) &= j + i - d - j \nonumber \\
            &= i - d \le 0.
  \end{align}

  Using (\ref{eq:case3_val_a}) and (\ref{eq:case3_val_c}) we see that
  \begin{align}
    \val(\e^j a\s(c)) &=  j + i - d  + (j -f) \nonumber \\
                      &=  (j+i) + j - (d + f) \nonumber \\
                      &= e + (j - k) > e \nonumber \\
    \noalign{\noindent and hence}
    \label{eq:case3_val_e_to_k_a_c}
    \val(\e^k a c) &= e + (j - k) + (k - j) = e.
  \end{align}
  Since, by (\ref{eq:case3_val_beta}), $\val(\beta) = e$, we can apply
  (\ref{eq:beta}) to see that
  \begin{align}
    \val(\e^i\s(b) - \e^k b - a(\e^j \s(c) - \e^k c)) &= e \nonumber \\
    \label{eq:case3_expand_beta}
    \val(\e^i\s(b) - \e^k b + \e^k ac)) &= e.
  \end{align}
  By (\ref{eq:case3_val_e_to_k_a_c}) and because $i > k$ we see that
  $\val(\e^kb) \ge e$, and hence $\val(\e^i\s(b)) > e$.  But then to satisfy
  (\ref{eq:case3_expand_beta}) we must have
  \begin{align}
    \val(\e^k(b-ac)) &= e \nonumber \\
    \label{eq:case3_val_b_m_ac}
    \val(b-ac) &= e-k.
  \end{align}
  Note that
  \begin{align}
    \val(b-ac) - \val(c) &= e - k - (j -f) \nonumber \\
                         &= i - d \nonumber \\
                         &\le 0, \nonumber \\
    \noalign{\noindent so}
    \val(b-ac) &\le \val(c).
  \end{align}

  By (\ref{eq:case3_val_minor}) and (\ref{eq:minor}),
  \begin{equation}
    \label{eq:case3_expand_minor}
    \val(\e^{i+j}\s(b) - \e^{j+k}b - \e^i\s(a)\gamma) > d + e.
  \end{equation}
  But by (\ref{eq:case3_val_a}) and (\ref{eq:case3_val_gamma}) and because $i
  \le d$ by assumption,
  \begin{align*}
   \val(\e^i\s(a)\gamma) &= i + (i - d) + (d + e - i) \\
                         &= i + e \\
                         &\le d + e.
  \end{align*}
  Since $i + j > j + k$, we must have
  \begin{align}
    \val(\e^{j+k}b) &= i + e \nonumber \\
    \val(b) &= i - (j + k) + e \nonumber \\
    \label{eq:case3_val_b}
    \val(b) &= e + 2i.
  \end{align}
  Note that
  \begin{align}
    \val(b) &= \val(a) + (e + i) + d \nonumber \\
            &= \val(a) + (i - f). \nonumber \\
    \noalign{\noindent Since $i \le d < f$,}
    \label{eq:case3_b_less_a}
    \val(b) &< \val(a).
  \end{align}

  Now $\val(c) < 0$ and $\val(b) < \val(a)$. So the conditions of
  Section~\ref{sec:hexagons-u-s1} are satisfied.  Since $\val(a) \le 0$,
  $\val(b) < 0$, and $\val(b-ac) \le \val(c)$, the hexagon for $g$ has the
  vertices
  \begin{equation*}
    \heximages
	{\begin{pmatrix}
	    0 & 1 & 0 \\
	    1 & 0 & 0 \\
	    0 & 0 & 1
	\end{pmatrix}}
	{\begin{pmatrix}
	    \e^{i-d} & 0 & 0 \\
	    0 & \e^{d-i} & 0 \\
	    0 & 0 & 1
	\end{pmatrix}}
	{\begin{pmatrix}
	    0 & 1 & 0 \\
	    0 & 0 & \e^{j-f} \\
	    \e^{f-j} & 0 & 0
	\end{pmatrix}}
	{\begin{pmatrix}
	    \e^{i-d} & 0 & 0 \\
	    0 & 0 & \e^{j-f} \\
	    0 & \e^{k-e} & 0
	\end{pmatrix}}
	{\begin{pmatrix}
	    0 & 0 & \e^{e+2i} \\
	    0 & \e^{d-i} & 0 \\
	    \e^{f - i} & 0 & 0
	\end{pmatrix}}
	{\begin{pmatrix}
	    0 & 0 & \e^{e+2i} \\
	    \e^{j-i} & 0 & 0 \\
	    0 & \e^{k-e} & 0
	\end{pmatrix}}.
  \end{equation*}
  This hexagon is completely determined by our choice of $i,j,k,d,e,f$.  So all
  elements of $\X \cap U_1 s_1 I$ have the same corresponding hexagon, and
  Theorem~\ref{thm:Triviality-For-One-Intersection} applies.

  \subsubsection{$x = \e^{(d, e, f)}s_1s_2s_1$, with $f \le d < e$}
  Let
  \begin{equation}
    x = \e^{(d, e, f)}s_1 s_2 s_1 = \begin{pmatrix}
        0 & 0 & \e^d \\
        0 & \e^e & 0 \\
        \e^f & 0 & 0
      \end{pmatrix}
  \end{equation}
  where $f \le d < e$, with $d + e + f = 0$.  Let $\nu = (i,j,k)$ with $i > j >
  k$ and $i + j + k = 0$.  We will show that in this case the intersection $\X
  \cap U_1 wI$ is nonempty only when $w = s_2 s_1$, and that
  Theorem~\ref{thm:Triviality-For-One-Intersection} applies.

  Since $f \le d < e$, by
  Theorem~\ref{thm:DeterminingIwahoriDoubleCosetForElement} the entry in the
  third row and first column of $h$ must have valuation $f$ and the determinant
  of the $2 \times 2$ minor which excludes the second column and second row
  must be $f + d$.  That means that $w$ can only be one of $s_1s_2$, $s_2s_1$,
  and $s_1s_2s_1$, since for all other values of $w$ the bottom-left entry is
  0.

  If $w = s_1s_2$, then
  \begin{equation*}
    h = \begin{pmatrix}
      \e^j & \gamma & 0 \\
      0 & \e^k & 0 \\
      \alpha & \beta & \e^i
    \end{pmatrix}.
  \end{equation*}
  In this case, the condition on the minor is that $j + i = f + d$, which means
  $k = e$.  But by assumption, $d < e$ and $f < e$, so
  \begin{align*}
    d + f &< 2e \\
    i + j &< 2k.
  \end{align*}
  This condition cannot be satisfied, since $i > k$ and $j > k$.  Therefore $\X
  \cap U_1 s_1 s_2 I = \emptyset$ in this case.

  If $w = s_1s_2s_1$, then
  \begin{equation*}
    h = \begin{pmatrix}
      \e^k & 0 & 0 \\
      \gamma & \e^j & 0 \\
      \beta & \alpha & \e^i
    \end{pmatrix}.
  \end{equation*}
  In this case, the condition on the minor is that $k + i = f + d$, which means
  $j = e$.  But by assumption, $d < e$ and $i > j$, so
  \begin{align*}
    f + d &= k + i \\
          &> k + j \\
          &= k + e \\
          &> k + d \\
    \noalign{\noindent and therefore}
        f &> k
  \end{align*}
  But for $h$ to be in $IxI$, we must have $\val(\e^k) \ge f$, which means $f
  \le k$.  These two conditions cannot both be satisfied, so $\X
  \cap U_1 s_1 s_2 s_1 I = \emptyset$ in this case.

  If $w = s_2s_1$, then
  \begin{equation}
    h = \begin{pmatrix}
      \e^k & 0 & 0 \\
      \beta & \e^i & \alpha \\
      \gamma & 0 & \e^j
    \end{pmatrix}.
  \end{equation}
  By Theorem~\ref{thm:DeterminingIwahoriDoubleCosetForElement}, the necessary
  conditions for $h$ to be in $IxI$ include
  \begin{align}
    \label{eq:2-e-max-f-min-val-gamma}
    \val(\gamma) &= f \\
    \label{eq:2-e-max-f-min-k-ge-f}
    \val(\e^k) &\ge f \implies k \ge f \\
    \label{eq:2-e-max-f-min-minor-equality}
    \val(\e^{j+k}) &= f + d \implies j + k = f + d \\
    \label{eq:2-e-max-f-min-val-minor}
    \val(\beta\e^j - \alpha\gamma) &> f + d\\
    \label{eq:2-e-max-f-min-val-alpha}
    \val(\e^k\alpha) &\ge f + d.
  \end{align}
  We will use these to determine the valuations of $a$, $b$, and $c$.

  First, we note that (\ref{eq:2-e-max-f-min-minor-equality}) implies that
  \begin{equation}
    \label{eq:2-e-max-f-min-e-equals-i}
    e = i.
  \end{equation}
  
  From (\ref{eq:2-e-max-f-min-val-gamma}), (\ref{eq:gamma}), and
  (\ref{eq:2-e-max-f-min-k-ge-f})  we see
  that \begin{equation}
    \label{eq:2-e-max-f-min-val-c}
    \val(c) = f - k \le 0.
  \end{equation}

  From (\ref{eq:2-e-max-f-min-val-alpha}), (\ref{eq:alpha}), and
  (\ref{eq:2-e-max-f-min-minor-equality}) we see that
  \begin{align}
    \val(a) + j + k &\ge d + f \nonumber \\
    \label{eq:2-e-max-f-min-val-a}
    \val(a) &\ge 0
  \end{align}

  Now we see from (\ref{eq:2-e-max-f-min-val-a}),
  (\ref{eq:2-e-max-f-min-val-gamma}), and
  (\ref{eq:2-e-max-f-min-e-equals-i}) that
  \begin{align}
    \val(\e^i\s(a)\gamma) &\ge i + f \nonumber \\
    &= f + e \nonumber \\
    \label{eq:2-e-max-f-min-trailing-term}
    &> f + d.
  \end{align}
  But from (\ref{eq:2-e-max-f-min-val-minor}) and (\ref{eq:minor}) we know that
  \begin{equation*}
    \val(\e^{i+j}\s(b) - \e^{j+k}b - \e^i\s(a)\gamma) > f + d.
  \end{equation*}
  Since $i > k$, this, combined with
  (\ref{eq:2-e-max-f-min-trailing-term}), implies that
  \begin{align}
    \val(\e^{j+k}b) &> f + d \nonumber \\
    \val(b) &> f + d - (j + k) \\
    \noalign{\noindent and by (\ref{eq:2-e-max-f-min-minor-equality})}
    \label{eq:2-e-max-f-min-val-b}
    \val(b) &> 0.
  \end{align}
  In particular, $\val(b) > \val(c)$.  Since $\val(c) \le 0$ and $\val(a) \ge
  0$, the conditions of Section~\ref{sec:hexagons-u-s2-s1} are satisfied, and
  because $\val(b) > 0$ we see that the hexagon for $g$ has the vertices
  \begin{equation*}
   \heximages
       {\begin{pmatrix}
           0 & 1 & 0 \\
           0 & 0 & 1 \\
           1 & 0 & 0
       \end{pmatrix}}
       {\begin{pmatrix}
           0 & 1 & 0 \\
           0 & 0 & 1 \\
           1 & 0 & 0
       \end{pmatrix}}
       {\begin{pmatrix}
           0 & 1 & 0 \\
           \e^{f-k} & 0 & 0 \\
           0 & 0 & \e^{k-f}
       \end{pmatrix}}
       {\begin{pmatrix}
           0 & 1 & 0 \\
           \e^{f-k} & 0 & 0 \\
           0 & 0 & \e^{k-f}
       \end{pmatrix}}
       {\begin{pmatrix}
           0 & 1 & 0 \\
           0 & 0 & 1 \\
           1 & 0 & 0
       \end{pmatrix}}
       {\begin{pmatrix}
           0 & 1 & 0 \\
           \e^{f-k} & 0 & 0 \\
           0 & 0 & \e^{k-f}
       \end{pmatrix}}.
  \end{equation*}
  This hexagon is completely determined by our choice of $i,j,k,d,e,f$.  So all
  elements of $\X \cap U_1 s_2 s_1 I$ have the same corresponding hexagon, and
  Theorem~\ref{thm:Triviality-For-One-Intersection} applies.

  \subsubsection{$x = \e^{(d, e, f)}s_1 s_2 s_1$, with $d < f \le e$}

  Let
  \begin{equation}
    x = \e^{(d, e, f)}s_1 s_2 s_1 = \begin{pmatrix}
        0 & 0 & \e^d \\
        0 & \e^e & 0 \\
        \e^f & 0 & 0
      \end{pmatrix}
  \end{equation}
  where $d < f \le e$, with $d + e + f = 0$.  Let $\nu = (i,j,k)$ with $i > j >
  k$ and $i + j + k = 0$.  We will show that in this case the intersection $\X
  \cap U_1 wI$ is nonempty only when $w = s_1$, and that
  Theorem~\ref{thm:Triviality-For-One-Intersection} applies.

  Since $d < f \le e$, by
  Theorem~\ref{thm:DeterminingIwahoriDoubleCosetForElement} the entry in the
  first row and third column of $h$ must have valuation $d$ and the determinant
  of the $2 \times 2$ minor which excludes the second column and second row
  must be $d + f$.  That means that $w$ can only be one of $1$, $s_1$, and
  $s_2$, since for all other values of $w$ the top-right entry is 0.

  If $w = 1$, then
  \begin{equation*}
    h = \begin{pmatrix}
      \e^i & \alpha & \beta \\
      0 & \e^j & \gamma \\
      0 & 0 & \e^k
    \end{pmatrix}.
  \end{equation*}
  In this case, the condition on the minor is that $i + k = d + f$, which means
  $j = e$.  But by assumption, $f \le e$ and $i > j$, so
  \begin{align*}
    d + f &= i + k \\
          &> j + k \\
          &= e + k \\
          &\ge f + k \\
    \noalign{\noindent and therefore}
        d &> k
  \end{align*}
  But for $h$ to be in $IxI$, we must have $\val(\e^k) > d$, which means $d
  < k$.  These two conditions cannot both be satisfied, so $\X
  \cap U_1 I = \emptyset$ in this case.

  If $w = s_2$, then
  \begin{equation*}
    h = \begin{pmatrix}
      \e^i & \beta & \alpha \\
      0 & \e^k & 0 \\
      0 & \gamma & \e^j
    \end{pmatrix}.
  \end{equation*}
  In this case, the condition on the minor is that $i + j = d + f$, which means
  $k = e$.  But by assumption, $d < e$ and $f \le e$, so
  \begin{align*}
    d + f &< 2e \\
    i + j &< 2k.
  \end{align*}
  This condition cannot be satisfied, since $i > k$ and $j > k$.  Therefore $\X
  \cap U_1 s_2 I = \emptyset$ in this case.

  If $w = s_1$, then
  \begin{equation*}
    h = \begin{pmatrix}
      \e^j & 0 & \gamma \\
      \alpha & \e^i & \beta \\
      0 & 0 & \e^k
    \end{pmatrix}.
  \end{equation*}
  By Theorem~\ref{thm:DeterminingIwahoriDoubleCosetForElement}, the necessary
  conditions for $h$ to be in $IxI$ include
  \begin{align}
    \label{eq:2-e-max-d-min-val-gamma}
    \val(\gamma) &= d \\
    \label{eq:2-e-max-d-min-k-greater-d}
    \val(\e^k) &> d \implies k > d \\
    \label{eq:2-e-max-d-min-minor-equality}
    \val(\e^{j+k}) &= d + f \implies j + k = d + f \\
    \label{eq:2-e-max-d-min-val-minor}
    \val(\beta\e^j - \alpha\gamma) &\ge d + f\\
    \label{eq:2-e-max-d-min-val-alpha}
    \val(\e^k\alpha) &> d + f.
  \end{align}
  We will use these to determine the valuations of $a$, $b$, $c$, and $b - ac$.

  First, we note that (\ref{eq:2-e-max-d-min-minor-equality}) implies
  that
  \begin{equation}
    \label{eq:2-e-max-d-min-e-equals-i}
    e = i.
  \end{equation}
  
  From (\ref{eq:2-e-max-d-min-val-gamma}), (\ref{eq:gamma}), and
  (\ref{eq:2-e-max-d-min-k-greater-d})  we see
  that \begin{equation}
    \label{eq:2-e-max-d-min-val-c}
    \val(c) = d - k < 0.
  \end{equation}

  From (\ref{eq:2-e-max-d-min-val-alpha}), (\ref{eq:alpha}), and
  (\ref{eq:2-e-max-d-min-minor-equality}) we see that
  \begin{align}
    \val(a) + j + k &> d + f \nonumber \\
    \label{eq:2-e-max-d-min-val-a}
    \val(a) &> 0
  \end{align}

  Now we see from (\ref{eq:2-e-max-d-min-val-a}),
  (\ref{eq:2-e-max-d-min-val-gamma}), and
  (\ref{eq:2-e-max-d-min-e-equals-i}) that
  \begin{align}
    \val(\e^i\s(a)\gamma) &> i + d \nonumber \\
    &= d + e \nonumber \\
    \label{eq:2-e-max-d-min-trailing-term}
    &\ge d + f.
  \end{align}
  But from (\ref{eq:2-e-max-d-min-val-minor}) and (\ref{eq:minor}) we know that
  \begin{equation*}
    \val(\e^{i+j}\s(b) - \e^{j+k}b - \e^i\s(a)\gamma) \ge f + d.
  \end{equation*}
  Since $i > k$, this, combined with
  (\ref{eq:2-e-max-d-min-trailing-term}), implies that
  \begin{align}
    \val(\e^{j+k}b) &\ge f + d \nonumber \\
    \val(b) &\ge f + d - (j + k) \\
    \noalign{\noindent and by (\ref{eq:2-e-max-d-min-minor-equality})}
    \label{eq:2-e-max-d-min-val-b}
    \val(b) &\ge 0.
  \end{align}
  Since $\val(c) < 0$ and $\val(a) > 0$, the conditions of
  Section~\ref{sec:hexagons-u-s1} are satisfied.  Note that $\val(b) > \val(c)$
  and $\val(ac) > \val(c)$, so $\val(b-ac) > \val(c)$.  Therefore we see that
  the hexagon for $g$ has the vertices
  \begin{equation*}
   \heximages
       {\begin{pmatrix}
           0 & 1 & 0 \\
           1 & 0 & 0 \\
           0 & 0 & 1
       \end{pmatrix}}
       {\begin{pmatrix}
           0 & 1 & 0 \\
           1 & 0 & 0 \\
           0 & 0 & 1
       \end{pmatrix}}
       {\begin{pmatrix}
           0 & 1 & 0 \\
           0 & 0 & \e^{d-k} \\
           \e^{k-d} & 0 & 0
       \end{pmatrix}}
       {\begin{pmatrix}
           0 & 1 & 0 \\
           0 & 0 & \e^{d-k} \\
           \e^{k-d} & 0 & 0
       \end{pmatrix}}
       {\begin{pmatrix}
           0 & 1 & 0 \\
           1 & 0 & 0 \\
           0 & 0 & 1
       \end{pmatrix}}
       {\begin{pmatrix}
           0 & 1 & 0 \\
           0 & 0 & \e^{d-k} \\
           \e^{k-d} & 0 & 0
       \end{pmatrix}}.
  \end{equation*}
  This hexagon is completely determined by our choice of $i,j,k,d,e,f$.  So all
  elements of $\X \cap U_1 s_1 I$ have the same corresponding hexagon, and
  Theorem~\ref{thm:Triviality-For-One-Intersection} applies.

  \subsubsection{$x = \e^{(d,e,f)} s_1s_2s_1$, with $d < e < f$ and $e \ge j$}
  \label{sec:first-long-two-cycle}
  Let
  \begin{equation}
    x = \e^{(d, e, f)}s_1s_2s_1 = \begin{pmatrix}
        0 & 0 & \e^d  \\
        0 & \e^e & 0 \\
         \e^f & 0 & 0
      \end{pmatrix}
  \end{equation}
  where $d < e < f$ and $d + e + f = 0$.  Let $\nu = (i,j,k)$ with $i > j > k$
  and $i + j + k = 0$.  Assume that $e \ge j$.  We will show that in this case
  the intersection $\X \cap U_1wI$ is nonempty only when $w = s_1$ or $w = 1$,
  and that if $e \neq i$ only the $w = 1$ intersection is nonempty.  Then we
  will show that if $e \neq i$
  Theorem~\ref{thm:Triviality-For-One-Intersection} applies and otherwise
  Theorem~\ref{thm:Triviality-For-Nice-Stratification} applies.
  
  Since $d < e < f$, by
  Theorem~\ref{thm:DeterminingIwahoriDoubleCosetForElement} the entry in the
  first row and third column of $h$ must have valuation $d$ and the valuation
  of the determinant of the top-right $2 \times 2$ minor must be $d + e$.  That
  means that $w$ can only be one of $1$, $s_1$, and $s_2$, since for all other
  values of $w$ the top-right entry is 0.

  If $w = s_2$, then
  \begin{equation*}
    h = \begin{pmatrix}
      \e^i & \beta & \alpha \\
      0 & \e^k & 0 \\
      0 & \gamma & \e^j
    \end{pmatrix}.
  \end{equation*}
  For $h$ to be in $IxI$, we must must have $\val(\alpha) = d$ and
  $\val(\e^k\alpha) = d + e$.  This means that $e = k$.  But by assumption, $e
  \ge j > k$, so in this case $\X \cap U_1 s_2 I = \emptyset$.

  If $w = s_1$, then
  \begin{equation*}
    h = \begin{pmatrix}
      \e^j & 0 & \gamma \\
      \alpha & \e^i & \beta \\
      0 & 0 & \e^k
    \end{pmatrix}.
  \end{equation*}
  For $h$ to be in $IxI$, we must have $\val(\gamma) = d$ and $\val(\e^i\gamma)
  = d + e$.  This means that $e = i$.  In all other cases, $\X \cap U_1 s_1 I =
  \emptyset$.

  So for $e \ne i$ we only have nonempty intersections with $U_1 wI$
  for $w = 1$.  In this case, \begin{equation*}
    h = \begin{pmatrix}
      \e^i & \alpha & \beta \\
      0 & \e^j & \gamma \\
      0 & 0 & \e^k
    \end{pmatrix}.
  \end{equation*}
  By Theorem~\ref{thm:DeterminingIwahoriDoubleCosetForElement}, the necessary
  conditions for $h$ to be in $IxI$ include
  \begin{align}
    \label{eq:2-f-max-d-min-id-val-beta}
    \val(\beta) &= d \\
    \label{eq:2-f-max-d-min-id-val-gamma}
    \val(\gamma) &> d \\
    \label{eq:2-f-max-d-min-id-k-greater-d}
    \val(\e^k) &> d \implies k > d \\
    \label{eq:2-f-max-d-min-id-i-less-f}
    \val(\e^{j+k}) &> d + e \implies j + k > d + e \implies i < f \\
    \label{eq:2-f-max-d-min-id-val-minor}
    \val(\beta\e^j - \alpha\gamma) &= d + e \\
    \label{eq:2-f-max-d-min-id-val-gamma2}
    \val(\e^i \gamma) &> d + e \\
    \label{eq:2-f-max-d-min-id-val-alpha}
    \val(\e^k\alpha) &> d + e.
  \end{align}
  We will use these to determine the valuations of $a$, $b$, $c$, and $b-ac$.

  First, note that by (\ref{eq:2-f-max-d-min-id-val-alpha}) and (\ref{eq:alpha})
  \begin{equation}
    \label{eq:2-f-max-d-min-id-val-a}
    \val(a) > d + e -  j - k.
  \end{equation}

  By (\ref{eq:2-f-max-d-min-id-val-gamma}), (\ref{eq:2-f-max-d-min-id-val-gamma2}),
  and (\ref{eq:gamma}),
  \begin{equation}
    \label{eq:2-f-max-d-min-id-val-c}
    \val(c) > \max(d-k, (d-k) + (e -i)).
  \end{equation}

  \begin{lemma}
    \label{lemma:2-f-max-d-min-id-val-ac}
    $\val(ac) \ge d - k$.
  \end{lemma}
  \begin{proof}
    Assume $\val(ac) < d - k$.  Then by (\ref{eq:gamma}), $\val(a\gamma) < d$.
    Now by (\ref{eq:2-f-max-d-min-id-val-beta}) and (\ref{eq:beta}) and because
    $i > k$, we must have
    \begin{equation*}
      \val(\e^k b) = \val(a\gamma) < d.
    \end{equation*}
    This would mean, because $i > j$ and $j \le e$, that
    \begin{align*}
      \val(\e^{i+j}\s(b) - \e^{j+k}b) &= j + \val(a\gamma) \\
      &< e + d.
    \end{align*}
    But then, by (\ref{eq:minor}), to satisfy
    (\ref{eq:2-f-max-d-min-id-val-minor}) we must have
    \begin{align*}
      \val(\e^i\s(a)\gamma) &= \val(\e^{i+j}\s(b) - \e^{j+k}b) \\
      & = j + \val(a\gamma).
    \end{align*}
    This requires $i = j$, which is impossible.  So $\val(ac) \ge d-k$.
  \end{proof}

  \begin{lemma}
    \label{lemma:2-f-max-d-min-id-e-g-j-val-ac}
    If $e > j$, then $\val(ac) = d - k$.
  \end{lemma}
  \begin{proof}
    Assume $\val(ac) > d - k$.  Then by (\ref{eq:2-f-max-d-min-id-val-beta}) and
    because $i > k$, we must have $\val(\e^k b) = d$.  In that case
    \begin{align*}
      \val(\e^{i+j}\s(b) - \e^{k+j}b) &= \val(\e^{k+j}b) \\
      &= d + j.
    \end{align*}
    At the same time, because $j > k$,
    \begin{align*}
      \val(\e^i \s(a) \gamma)) &= i + k + \val(ac) \\
      &> i + k + d - k \\
      &= d + i \\
      &> d + j.
    \end{align*}
    Since the difference of these two terms if $\beta\e^j - \alpha\gamma$, we
    have $\val(\beta\e^j - \alpha\gamma) = d + j < d + e$, contradicting
    (\ref{eq:2-f-max-d-min-id-val-minor}).  Therefore we must have $\val(ac)
    \le d - k$.  Since we already know $\val(ac) \ge d-k$ by
    Lemma~\ref{lemma:2-f-max-d-min-id-val-ac}, we conclude that $\val(ac) = d -
    k$.
  \end{proof}

  \begin{lemma}
    \label{lemma:2-f-max-d-min-id-e-g-j-val-a}
    If $e > j$ then $0 > \val(a)$.
  \end{lemma}
  \begin{proof}
    By Lemma~\ref{lemma:2-f-max-d-min-id-e-g-j-val-ac}, $\val(ac) = d - k$.  But
    by (\ref{eq:2-f-max-d-min-id-val-c}), $\val(c) > d - k$.  Therefore $\val(a)
    < 0$.
  \end{proof}
  \begin{lemma}
    \label{lemma:2-f-max-d-min-id-e-g-j-val-c}
    If $e > j$ then $0 > \val(c)$.
  \end{lemma}
  \begin{proof}
    By Lemma~\ref{lemma:2-f-max-d-min-id-e-g-j-val-ac}, $\val(ac) = d - k$.  But
    by (\ref{eq:2-f-max-d-min-id-val-a}), $\val(a) > (d - k) + (e-j)$.  Therefore
    $\val(c) < j - e < 0$.
  \end{proof}

  Now we have three possibilities: $e > i$, $e < i$, and $e = i$.

  \subsubsubsection{The case $e > i$}

  \begin{lemma}
    \label{lemma:2-f-max-d-min-id-e-g-i-val-b}
    If $e > i$, then $\val(b) = d + 2i$.  This means that $\val(b) < 0$,
    $\val(b) < \val(a)$, and $\val(b-ac) = d-k < \val(c) < 0$.
  \end{lemma}
  \begin{proof}
    Since $e > i > j$, we know that $\val(ac) = d - k$ by
    Lemma~\ref{lemma:2-f-max-d-min-id-e-g-j-val-ac}.  So
    \begin{align*}
      \val(\e^i\s(a)\gamma) &= i + d - k + k \\
      &= d + i \\
      &< d + e.
    \end{align*}
    But then, by (\ref{eq:minor}), to satisfy
    (\ref{eq:2-f-max-d-min-id-val-minor}) we must have
    \begin{align*}
      \val(\e^{i+j}\s(b) - \e^{k+j}b) &= d + i \\
      \val(\e^{k+j}b) &= d + i\\
      \val(b) &= d + i - j - k \\
              &= d + 2i.
    \end{align*}
    Now \begin{align*}
      \val(b) &= d + i - j - k \\
              &< d + e - j - k.
    \end{align*}
    By (\ref{eq:2-f-max-d-min-id-val-a}), $d + e - j - k < \val(a)$, so
    $\val(b) < \val(a)$.  Since $\val(a) < 0$ by
    Lemma~\ref{lemma:2-f-max-d-min-id-e-g-j-val-a}, we see that $\val(b) < 0$.

    Since $\val(b) = d - k + (i - j) > d-k$ and by
    Lemma~\ref{lemma:2-f-max-d-min-id-e-g-j-val-ac} $\val(ac) = d - k$, we see
    that $\val(b-ac) = d-k$.  By (\ref{eq:2-f-max-d-min-id-val-c}), $d-k <
    \val(c)$, and by Lemma~\ref{lemma:2-f-max-d-min-id-e-g-j-val-c} $\val(c) <
    0$.
  \end{proof}

  Since $\val(b) < 0$ and $\val(b-ac) < \val(c) < 0$, the conditions of
  Section~\ref{sec:hexagons-u} are satisfied.  By
  Lemma~\ref{lemma:2-f-max-d-min-id-e-g-j-val-a}, $\val(a) < 0$.  This means that
  the hexagon for $g$ has the vertices
  \begin{equation*}
   \heximagessqueezedcarefully{1.8em}{-0.8em}
       {\begin{pmatrix}
           1 & 0 & 0 \\ 
           0 & 1 & 0 \\
           0 & 0 & 1
       \end{pmatrix}}
       {\begin{pmatrix}
           0 & \e^{\val(a)} & 0 \\
           \e^{-\val(a)} & 0 & 0 \\
           0 & 0 & 1
       \end{pmatrix}}
       {\begin{pmatrix}
           1 & 0 & 0 \\
           0 & 0 & \e^{\val(c)} \\
           0 & \e^{-\val(c)}  & 0
       \end{pmatrix}}
       {\begin{pmatrix}
           0 & \e^{d-k - \val(c)} & 0 \\
           0 & 0 & \e^{\val(c)} \\
           \e^{k-d} & 0 & 0
       \end{pmatrix}}
       {\begin{pmatrix}
           0 & 0 & \e^{d+2i} \\
           \e^{-\val(a)} & 0 & 0 \\
           0 & \e^{\val(a) - d - 2i} & 0  
       \end{pmatrix}}
       {\begin{pmatrix}
           0 & 0 & \e^{d+2i} \\
           0 & \e^{-k-2i} & 0 \\
           \e^{k-d} & 0 & 0
       \end{pmatrix}}
  \end{equation*}
  when $e > i$.  Note that all of these hexagons share two opposite vertices:
  the top and bottom one.  Therefore
  Theorem~\ref{thm:Triviality-For-One-Intersection} applies.

  \subsubsubsection{The case $e < i$}

  \begin{lemma}
    \label{lemma:2-f-max-d-min-id-e-l-i-val-b}
    If $e < i$, then $\val(b) = d + e - j - k = i - f$.  This means that
    $\val(b) < 0$ and $\val(b) < \val(a)$.
  \end{lemma}
  \begin{proof}
    By Lemma~\ref{lemma:2-f-max-d-min-id-val-ac}, $\val(ac) \ge d -k$.
    Therefore, by (\ref{eq:gamma}),
    \begin{align*}
      \val(\e^i\s(a)\gamma) &\ge d + i \\
       &> d + e.
    \end{align*}
    Since $i > k$ and (\ref{eq:minor}) holds, to satisfy
    (\ref{eq:2-f-max-d-min-id-val-minor}) we must have
    \begin{align*}
      \val(\e^{j+k} b) &= d + e \\
      \val(b) &= d + e - j - k \\
      &= i - f.
    \end{align*}
    By (\ref{eq:2-f-max-d-min-id-i-less-f}), $\val(b) < 0$.  By
    (\ref{eq:2-f-max-d-min-id-val-a}), $\val(b) < \val(a)$.
  \end{proof}

  \begin{lemma}
    If $e < i$, then $\val(b-ac) = d -k$.  This means that $\val(b-ac) < 0$ and
    $\val(b-ac) < \val(c)$.
  \end{lemma}
  \begin{proof}
    By Lemma~\ref{lemma:2-f-max-d-min-id-e-l-i-val-b}, $\val(b) = d + e - j - k$.
    Hence
    \begin{align*}
      \val(\e^i \s(b)) &= d + e - j - k + i \\
      &> d + e - j \\
      &\ge d.
    \end{align*}
    By Lemma~\ref{lemma:2-f-max-d-min-id-val-ac}, $\val(ac) \ge d-k$.  So
    \begin{align*}
      \val(\e^j a \s(c)) &\ge d - k + j \\
      &> d.
    \end{align*}
    For (\ref{eq:2-f-max-d-min-id-val-beta}) to be satisfied, we must have
    \begin{align*}
      \val(-\e^k b + \e^k ac) &= d\\
      \val(b - ac) &= d-k.
    \end{align*}
    By (\ref{eq:2-f-max-d-min-id-k-greater-d}), $\val(b-ac) < 0$.  By
    (\ref{eq:2-f-max-d-min-id-val-c}), $\val(b-ac) < \val(c)$.
  \end{proof}

  Since $\val(b) < 0$ and $\val(b-ac) < 0$, the conditions of
  Section~\ref{sec:hexagons-u} are satisfied.  We have $\val(b-ac) < \val(c)$
  and $\val(b) < \val(a)$, so the hexagons we get only depend on whether the
  valuations of $a$ and $c$ are negative.

  Now we have four cases:
  \begin{caselist}
    \item
      $\val(c) < 0$ and $\val(a) < 0$.  In this case, the hexagon for $g$ is
      \begin{equation*}
	\heximages
	    {\begin{pmatrix}
		1 & 0 & 0 \\ 
		0 & 1 & 0 \\
		0 & 0 & 1
	    \end{pmatrix}}
	    {
	      \begin{pmatrix}
		0 & \e^{\val(a)} & 0 \\
		\e^{-\val(a)} & 0 & 0 \\
		0 & 0 & 1
              \end{pmatrix}
	    }       
	    {
	      \begin{pmatrix}
		1 & 0 & 0 \\
		0 & 0 & \e^{\val(c)} \\
		0 & \e^{-\val(c)}  & 0
	      \end{pmatrix}
	    }
	    {
              \begin{pmatrix}
		0 & \e^{d-k - \val(c)} & 0 \\
		0 & 0 & \e^{\val(c)} \\
		\e^{k-d} & 0 & 0
              \end{pmatrix}
	    }       
	    {
              \begin{pmatrix}
		0 & 0 & \e^{i-f} \\
		\e^{-\val(a)} & 0 & 0 \\
		0 & \e^{\val(a) - i + f} & 0  
              \end{pmatrix}
	    }
	    {\begin{pmatrix}
		0 & 0 & \e^{i-f} \\
		0 & \e^{j - e} & 0 \\
		\e^{k-d} & 0 & 0
	    \end{pmatrix}}.
      \end{equation*}
    \item $\val(c) < 0$ and $\val(a) \ge 0$.  In this case, the hexagon for $g$
    is
    \begin{equation*}
      \heximages
	  {\begin{pmatrix}
              1 & 0 & 0 \\ 
              0 & 1 & 0 \\
              0 & 0 & 1
	  \end{pmatrix}}
	  {
            \begin{pmatrix}
              1 & 0 & 0 \\ 
              0 & 1 & 0 \\
              0 & 0 & 1
            \end{pmatrix}
	  }       
	  {
            \begin{pmatrix}
	      1 & 0 & 0 \\
	      0 & 0 & \e^{\val(c)} \\
	      0 & \e^{-\val(c)}  & 0
            \end{pmatrix}
	  }
	  {
            \begin{pmatrix}
              0 & \e^{d-k - \val(c)} & 0 \\
              0 & 0 & \e^{\val(c)} \\
              \e^{k-d} & 0 & 0
            \end{pmatrix}
	  }       
	  {
            \begin{pmatrix}
              0 & 0 & \e^{i-f} \\ 
              0 & 1 & 0 \\
              \e^{f-i} & 0 & 0
            \end{pmatrix}
	  }
	  {\begin{pmatrix}
              0 & 0 & \e^{i-f} \\
              0 & \e^{j - e} & 0 \\
              \e^{k-d} & 0 & 0
	  \end{pmatrix}}.
    \end{equation*}
    Note that since $\val(a) \ge 0$, by
    Lemma~\ref{lemma:2-f-max-d-min-id-e-g-j-val-a} we must have $e = j$, which
    means that $d-k = i-f$.  Therefore the bottom vertex and the bottom-right
    vertex coincide in this case.
    \item $\val(c) \ge 0$ and $\val(a) < 0$.  In this case, the hexagon for $g$
      is
      \begin{equation*}
	\heximages
	    {\begin{pmatrix}
		1 & 0 & 0 \\ 
		0 & 1 & 0 \\
		0 & 0 & 1
	    \end{pmatrix}}
	    {
	      \begin{pmatrix}
		0 & \e^{\val(a)} & 0 \\
		\e^{-\val(a)} & 0 & 0 \\
		0 & 0 & 1
              \end{pmatrix}
	    }       
	    {
              \begin{pmatrix}
		1 & 0 & 0 \\ 
		0 & 1 & 0 \\
		0 & 0 & 1
              \end{pmatrix}
	    }
	    {
              \begin{pmatrix}
		0 & 0 & \e^{d - k} \\ 
		0 & 1 & 0 \\
		\e^{k-d} & 0 & 0
              \end{pmatrix}
	    }       
	    {
              \begin{pmatrix}
		0 & 0 & \e^{i-f} \\
		\e^{-\val(a)} & 0 & 0 \\
		0 & \e^{\val(a) - i + f} & 0  
              \end{pmatrix}
	    }
	    {\begin{pmatrix}
		0 & 0 & \e^{i-f} \\
		0 & \e^{j - e} & 0 \\
		\e^{k-d} & 0 & 0
	    \end{pmatrix}}.
      \end{equation*}
      Note that since $\val(c) \ge 0$, by
      Lemma~\ref{lemma:2-f-max-d-min-id-e-g-j-val-c} we must have $e = j$,
      which means that $d-k = i-f$.  Therefore the bottom vertex and the
      bottom-left vertex coincide in this case.
    \item $\val(c) \ge 0$ and $\val(a) \ge 0$.  In this case, the hexagon for
      $g$ is
      \begin{equation*}
	\heximages
	    {\begin{pmatrix}
		1 & 0 & 0 \\ 
		0 & 1 & 0 \\
		0 & 0 & 1
	    \end{pmatrix}}
	    {
              \begin{pmatrix}
		1 & 0 & 0 \\ 
		0 & 1 & 0 \\
		0 & 0 & 1
              \end{pmatrix}
	    }       
	    {
              \begin{pmatrix}
		1 & 0 & 0 \\ 
		0 & 1 & 0 \\
		0 & 0 & 1
              \end{pmatrix}
	    }
	    {
              \begin{pmatrix}
		0 & 0 & \e^{d - k} \\ 
		0 & 1 & 0 \\
		\e^{k-d} & 0 & 0
              \end{pmatrix}
	    }       
	    {
              \begin{pmatrix}
		0 & 0 & \e^{i-f} \\ 
		0 & 1 & 0 \\
		\e^{f-i} & 0 & 0
              \end{pmatrix}
	    }
	    {\begin{pmatrix}
		0 & 0 & \e^{i-f} \\
		0 & \e^{j - e} & 0 \\
		\e^{k-d} & 0 & 0
	    \end{pmatrix}}.
      \end{equation*}
      Note that since $\val(c) \ge 0$, by
      Lemma~\ref{lemma:2-f-max-d-min-id-e-g-j-val-c} we must have $e = j$,
      which means that $d-k = i-f$.  Therefore the bottom vertex, the
      bottom-left vertex, and the bottom-right vertex all coincide in this
      case.
  \end{caselist}
  Note that all of these hexagons share two opposite vertices: the top and
  bottom one.  Therefore Theorem~\ref{thm:Triviality-For-One-Intersection}
  applies.

  \subsubsubsection{The case $e = i$}

  In this case, there are intersections with both $U_1 I$ and $U_1 s_1 I$.  If
  $w = s_1$, then
  \begin{equation*}
    h = \begin{pmatrix}
      \e^j & 0 & \gamma \\
      \alpha & \e^i & \beta \\
      0 & 0 & \e^k
    \end{pmatrix}.
  \end{equation*}
  By Theorem~\ref{thm:DeterminingIwahoriDoubleCosetForElement}, the necessary
  conditions for $h$ to be in $IxI$ include
  \begin{align}
    \label{eq:2-f-max-d-min-s1-val-gamma}
    \val(\gamma) &= d \\
    \label{eq:2-f-max-d-min-s1-val-beta}
    \val(\beta) &> d \\
    \label{eq:2-f-max-d-min-s1-k-greater-d}
    \val(\e^k) &> d \implies k > d \\
    \label{eq:2-f-max-d-min-s1-minor-equality}
    \val(\e^i\gamma) &= d + e \\
    \label{eq:2-f-max-d-min-s1-val-minor}
    \val(\beta\e^j - \alpha\gamma) &> d + e
  \end{align}
  We will use these to determine the valuations of $a$, $b$, $c$, and $b-ac$.

  From (\ref{eq:2-f-max-d-min-s1-val-gamma}), (\ref{eq:gamma}), and
  (\ref{eq:2-f-max-d-min-s1-k-greater-d}) we see that
  \begin{equation}
    \label{eq:2-f-max-d-min-s1-val-c}
    \val(c) = d - k < 0.
  \end{equation}

  From (\ref{eq:2-f-max-d-min-s1-val-minor}) and (\ref{eq:minor}) we see that
  \begin{align}
    \val(\e^{i+j}\s(b) - \e^{j+k}b - \e^i\s(a)\gamma) &> d + e \nonumber\\
    \noalign{\noindent and since $e = i$}
    \label{eq:2-f-max-d-min-s1-trailing-term}
    \val(\e^j\s(b) - \e^{k+j-i}b - \s(a)\gamma) &> d. \\
    \noalign{\noindent At the same time, by (\ref{eq:2-f-max-d-min-s1-val-beta}) and
    (\ref{eq:altbeta}),}
    \val(\e^i\s(b) - \e^k b - a\gamma) &> d. \nonumber
  \end{align}
  Now $\val(\s(a)\gamma) = \val(a\gamma)$.  If this valuation were less than or
  equal to $d$, then, because $j > k$ and $i > k$ we would have to have
  $\val(\e^k b) = \val(\e^{k + j - i}b) = \val(a\gamma)$ to cancel the terms of
  valuation $d$ or lower.  But this would require $j = i$, which is
  impossible.  Therefore,
  \begin{align}
    \val(a\gamma) &> d \nonumber \\
    \noalign{\noindent and by (\ref{eq:2-f-max-d-min-s1-val-gamma})}
    \label{eq:2-f-max-d-min-s1-val-a}
    \val(a) &> 0.
  \end{align}

  This means that $\val(\s(a)\gamma) > d$, so to satisfy
  (\ref{eq:2-f-max-d-min-s1-trailing-term}) we must have
  \begin{align}
    \val(\e^{k + j - i} b) &> d \nonumber \\
    \val(b) &> d + i - j - k \nonumber \\
            &> d - k \nonumber  \\
            &= \val(c).
   \end{align}

  Since $\val(c) < 0$ and $\val(a) > 0$, the conditions of
  Section~\ref{sec:hexagons-u-s1} are satisfied.  Note that $\val(b) > \val(c)$
  and $\val(ac) > \val(c)$, so $\val(b-ac) > \val(c)$.  But there are no
  restrictions on how $\val(b)$ compares with $0$.  Therefore we see that
  if $\val(b) \ge 0$ the hexagon for $g$ has the vertices
  \begin{equation}
    \label{eq:2-f-max-d-min-s1-hex1}
    \heximages
       {\begin{pmatrix}
           0 & 1 & 0 \\
           1 & 0 & 0 \\
           0 & 0 & 1
       \end{pmatrix}}
       {\begin{pmatrix}
           0 & 1 & 0 \\
           1 & 0 & 0 \\
           0 & 0 & 1
       \end{pmatrix}}
       {\begin{pmatrix}
           0 & 1 & 0 \\
           0 & 0 & \e^{d-k} \\
           \e^{k-d} & 0 & 0
       \end{pmatrix}}
       {\begin{pmatrix}
           0 & 1 & 0 \\
           0 & 0 & \e^{d-k} \\
           \e^{k-d} & 0 & 0
       \end{pmatrix}}
       {\begin{pmatrix}
           0 & 1 & 0 \\
           1 & 0 & 0 \\
           0 & 0 & 1
       \end{pmatrix}}
       {\begin{pmatrix}
           0 & 1 & 0 \\
           0 & 0 & \e^{d-k} \\
           \e^{k-d} & 0 & 0
       \end{pmatrix}}
  \end{equation}
  and if $\val(b) < 0$ it has the vertices
  \begin{equation}
    \label{eq:2-f-max-d-min-s1-hex2}
    \heximages
       {\begin{pmatrix}
           0 & 1 & 0 \\
           1 & 0 & 0 \\
           0 & 0 & 1
       \end{pmatrix}}
       {\begin{pmatrix}
	   & \hphantom{\e^{-\val(b)}} & \hphantom{\e^{\val(b)}} \\[-1em]
	   0 & 1 & 0 \\ 
           1 & 0 & 0 \\
           0 & 0 & 1 
       \end{pmatrix}}
       {\begin{pmatrix}
           0 & 1 & 0 \\
           0 & 0 & \e^{d-k} \\
           \e^{k-d} & 0 & 0
       \end{pmatrix}}
       {\begin{pmatrix}
           0 & 1 & 0 \\
           0 & 0 & \e^{d-k} \\
           \e^{k-d} & 0 & 0
       \end{pmatrix}}
       {\begin{pmatrix}
           0 & 0 & \e^{\val(b)} \\
           1 & 0 & 0 \\
           0 & \e^{-\val(b)}  & 0
       \end{pmatrix}}
       {\begin{pmatrix}
           0 & 0 & \e^{\val(b)} \\
           0 & \e^{d - k - \val(b)} & 0 \\
           \e^{k-d} & 0  & 0
       \end{pmatrix}}.
  \end{equation}

  Now we look at $w = 1$.  Since $e = i > j$, by
  Lemma~\ref{lemma:2-f-max-d-min-id-e-g-j-val-ac} we know that $\val(ac) = d -
  k$.  This means that $\val(\e^i\s(a)\gamma) = d + i = d + e$.  So for
  (\ref{eq:2-f-max-d-min-id-val-minor}) to be satisfied, we must have,
  by~(\ref{eq:minor}),
  \begin{align*}
    \val(\e^{j+k}b) &\ge d + e \\
    \val(b) &\ge d + e - j - k \\
            &= d  -k + (i - j) \\
            &> d-k.
  \end{align*}
  Since $\val(ac) = d -k$, we conclude that
  \begin{equation}
    \label{eq:2-f-max-d-min-id-e-eq-i-val-b-minus-ac}
    \val(b-ac) = d - k.
  \end{equation}
  Then by (\ref{eq:2-f-max-d-min-id-val-c}) and
  Lemma~\ref{lemma:2-f-max-d-min-id-e-g-j-val-c},
  \begin{equation}
    \label{eq:2-f-max-d-min-id-e-eq-i-val-b-minus-ac2}
    \val(b-ac) < \val(c) < 0.
  \end{equation}



  By Lemma~\ref{lemma:2-f-max-d-min-id-e-g-j-val-a}, $\val(a) < 0$.  So the
  conditions of Section~\ref{sec:hexagons-u} are satisfied.  $d - k =
  \val(b-ac) < \val(c)$ and $\val(c) = d - k - \val(a)$.  The hexagon we get
  depends on how $\val(b)$ compares to $\val(a)$.  If $\val(b) < \val(a)$, the
  hexagon for $g$ has the vertices
  \begin{equation*}
    \heximagessqueezed{4em}
        {
          \begin{pmatrix}
            1 & 0 & 0 \\ 
            0 & 1 & 0 \\
            0 & 0 & 1
          \end{pmatrix}
        }
        {
          \begin{pmatrix}
            & \hphantom{\e^{\val(a) - \val(b)}} & \hphantom{\e^{\val(b)}} \\[-1em]
            0 & \e^{\val(a)} & 0 \\
            \e^{-\val(a)} & 0 & 0 \\
            0 & 0 & 1
          \end{pmatrix}
        }
        {
          \begin{pmatrix}
            1 & 0 & 0 \\
            0 & 0 & \e^{d - k - \val(a)} \\
            0 & \e^{k - d + \val(a)}  & 0
          \end{pmatrix} 
        }
        {
          \begin{pmatrix}
            0 & \e^{\val(a)} & 0 \\
            0 & 0 & \e^{d - k - \val(a)} \\
            \e^{k-d} & 0 & 0
           \end{pmatrix}
        }
        {
	  \begin{pmatrix}
            0 & 0 & \e^{\val(b)} \\
            \e^{-\val(a)} & 0 & 0 \\
            0 & \e^{\val(a) - \val(b)} & 0  
          \end{pmatrix}
        }
        {
          \begin{pmatrix}
            0 & 0 & \e^{\val(b)} \\
            0 & \e^{d-k-\val(b)} & 0 \\
            \e^{k-d} & 0 & 0
          \end{pmatrix}
        }
  \end{equation*}
  and if $\val(b) \ge \val(a)$ it has the vertices
  \begin{equation*}
    \heximages
        {
          \begin{pmatrix}
            1 & 0 & 0 \\ 
            0 & 1 & 0 \\
            0 & 0 & 1
          \end{pmatrix}
        }
        {
          \begin{pmatrix}
            0 & \e^{\val(a)} & 0 \\
            \e^{-\val(a)} & 0 & 0 \\
            0 & 0 & 1
          \end{pmatrix}
        }
        {
          \begin{pmatrix}
            1 & 0 & 0 \\
            0 & 0 & \e^{d - k - \val(a)} \\
            0 & \e^{k - d + \val(a)}  & 0
          \end{pmatrix} 
        }
        {
          \begin{pmatrix}
            0 & \e^{\val(a)} & 0 \\
            0 & 0 & \e^{d - k - \val(a)} \\
            \e^{k-d} & 0 & 0
           \end{pmatrix}
        }
        {
          \begin{pmatrix}
            0 & \e^{\val(a)} & 0 \\
            \e^{-\val(a)} & 0 & 0 \\
            0 & 0 & 1
          \end{pmatrix}
        }
        {
          \begin{pmatrix}
            0 & \e^{\val(a)} & 0 \\
            0 & 0 & \e^{d - k - \val(a)} \\
            \e^{k-d} & 0 & 0
          \end{pmatrix}
        }.
  \end{equation*}

  We want to apply Theorem~\ref{thm:Triviality-For-Nice-Stratification} to this
  case.  In the notation of that theorem, $w_0 = s_1$ and $w_1 = 1$.  The
  subsets $Y_\delta$ correspond to subsets defined by $\val(a) = \delta$.  We
  let $\mu_\delta$ be $(-\val(a), \val(a), 0) = (-\delta, \delta, 0)$, so that
  \begin{equation*}
    \e^{\mu_\delta} = \begin{pmatrix}
      \e^{-\val(a)} & 0 & 0 \\
      0 & \e^{\val(a)} & 0 \\
      0 & 0 & 1
    \end{pmatrix}
  \end{equation*}
  and let $Y'_\delta = \e^{\mu_\delta} Y_\delta$.  Then when $\val(b) <
  \val(a)$ the hexagon corresponding to elements of $Y'_\delta$ is
  \begin{equation*}
    \heximagessqueezed{6.5em}
        {
          \begin{pmatrix}
            \e^{-\val(a)} & 0 & 0 \\ 
            0 & \e^{\val(a)} & 0 \\
            0 & 0 & 1
          \end{pmatrix}
        }
        {
          \begin{pmatrix}
             & \hphantom{\e^{\val(a) - \val(b)}} &
	       \hphantom{\e^{\val(b)-\val(a)}} \\[-1em]
            0 & 1 & 0 \\
            1 & 0 & 0 \\
            0 & 0 & 1
          \end{pmatrix}
        }
        {
          \begin{pmatrix}
            \e^{-\val(a)} & 0 & 0 \\
            0 & 0 & \e^{d - k} \\
            0 & \e^{k - d + \val(a)}  & 0
          \end{pmatrix} 
        }
        {
          \begin{pmatrix}
            \hphantom{\e^{-\val(a)}} & \hphantom{\e^{k - d + \val(a)}} & \\[-1em]
            0 & 1 & 0 \\
            0 & 0 & \e^{d - k} \\
            \e^{k-d} & 0 & 0
           \end{pmatrix}
        }
        {
          \begin{pmatrix}
            0 & 0 & \e^{\val(b)-\val(a)} \\
            1 & 0 & 0 \\
            0 & \e^{\val(a) - \val(b)} & 0  
          \end{pmatrix}
        }
        {
          \begin{pmatrix}
            0 & 0 & \e^{\val(b) - \val(a)} \\
            0 & \e^{d-k-\val(b) + \val(a)} & 0 \\
            \e^{k-d} & 0 & 0
          \end{pmatrix}
        }
  \end{equation*}
  and when $\val(b) \ge \val(a)$ it is
  \begin{equation*}
    \heximages
        {
          \begin{pmatrix}
            \e^{-\val(a)} & 0 & 0 \\ 
            0 & \e^{\val(a)} & 0 \\
            0 & 0 & 1
          \end{pmatrix}
        }
        {
          \begin{pmatrix}
            0 & 1 & 0 \\
            1 & 0 & 0 \\
            0 & 0 & 1
          \end{pmatrix}
        }
        {
          \begin{pmatrix}
            \e^{-\val(a)} & 0 & 0 \\
            0 & 0 & \e^{d - k} \\
            0 & \e^{k - d + \val(a)}  & 0
          \end{pmatrix} 
        }
        {
          \begin{pmatrix}
            0 & 1 & 0 \\
            0 & 0 & \e^{d - k} \\
            \e^{k-d} & 0 & 0
           \end{pmatrix}
        }
        {
          \begin{pmatrix}
            0 & 1 & 0 \\
            1 & 0 & 0 \\
            0 & 0 & 1
          \end{pmatrix}
        }
        {
          \begin{pmatrix}
            0 & 1 & 0 \\
            0 & 0 & \e^{d - k} \\
            \e^{k-d} & 0 & 0
          \end{pmatrix}
        }.
  \end{equation*}
  Comparing these to the hexagons in (\ref{eq:2-f-max-d-min-s1-hex1}) and
  (\ref{eq:2-f-max-d-min-s1-hex2}) we see that all four sets of hexagons share
  two opposite vertices: the top right and bottom left one.  Define the set $Z$
  as in the statement of Theorem~\ref{thm:Triviality-For-Nice-Stratification}.
  Then $\A$ acts on $Z$ by left-multiplication, and $\X$ is a disjoint union of
  translates of $Z$ by elements of $\A(F)/\A$.  So by
  Proposition~\ref{prop:closedness-of-hexagon-piece}, $Z$ is closed.

  Now we show that $\X \cap U_1 s_1 I$ is closed in $Z$.  Observe that by
  (\ref{eq:2-f-max-d-min-s1-hex1}) and (\ref{eq:2-f-max-d-min-s1-hex2}) the
  hexagon corresponding to any element of $(\X \cap U_1 s_1 I)$ has a top
  vertex that coincides with the top-right vertex and a top-left vertex that
  coincides with the bottom-left vertex.  That is, this hexagon is
  degenerate, and actually is a trapezoid that lies to one side of the line
  connecting the top-right and bottom-left vertices.  A hexagon corresponding
  to an element of $Y'_\delta$ has those same top-right and bottom-left
  vertices, but has top and top-left vertices that are distinct.  In
  particular, those two vertices are on the opposite side of the
  top-right-to-bottom-left line from the hexagons corresponding to elements of
  $(\X \cap U_1 s_1 I)$.  Thus the closure of $(\X \cap U_1 s_1 I)$ in $X$
  cannot contain any elements of any of the $Y'_\delta$, and in particular $(\X
  \cap U_1 s_1 I)$ is closed in $Z$.

  Finally, we show that
  \begin{equation*}
    (\X \cap U_1 s_1 I) \bigcup \left(\bigcup_{\delta > m} Y'_\delta\right)
  \end{equation*}
  is closed in $Z$.  To show this, it will be enough to show that if $\delta_1
  > \delta_2$ then no hexagon corresponding to an element of $Y'_{\delta_1}$
  can contain a hexagon corresponding to an element of $Y'_{\delta_2}$.  Now
  $Y_{\delta_1}$ and $Y_{\delta_2}$ share the same top vertex.  The sides
  connecting the top and top-right vertex are different lengths.  In fact, the
  length depends $\delta_1$ and $\delta_2$.  Since $0 > \delta_1 > \delta_2$,
  so that $|\delta_1| < |\delta_2$, the side length of the hexagon
  corresponding to an element of $Y_{\delta_1}$ is smaller.  When we translate
  to get $Y'_{\delta_1}$ and $Y'_{\delta_2}$, we are translating both hexagons
  parallel to the line connecting the top and top-right vertex, and the two
  top-right vertices end up in the same place.  But this means that we have to
  translate the hexagon corresponding to an element of $Y_{\delta_2}$ further,
  so that its top vertex is no longer inside the hexagon corresponding to an
  element of $Y_{\delta_1}$, as shown in Figure~\ref{fig:different-lengths}.
  \begin{figure}[htbp]
    \begin{center}
      \includegraphics{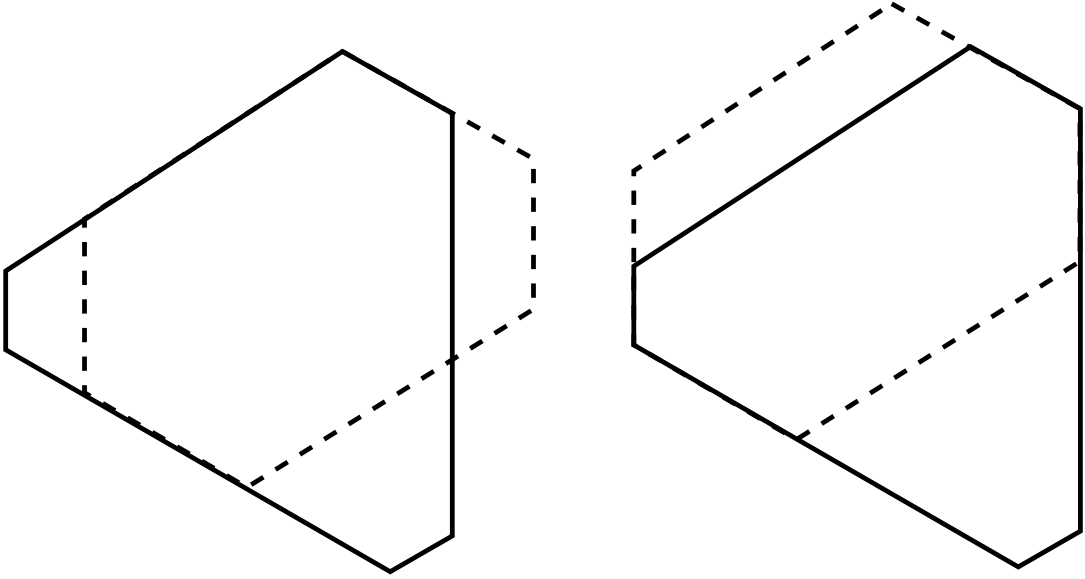}
    \end{center}
    \caption{Hexagons corresponding to elements of $Y_{\delta_1}$,
    $Y_{\delta_2}$, $Y'_{\delta_1}$, $Y'_{\delta_2}$.  Solid lines are
    $Y_{\delta_1}$ and $Y'_{\delta_1}$; dashed lines are $Y_{\delta_2}$ and
    $Y'_{\delta_2}$. $Y_{\delta_1}$ and $Y_{\delta_2}$ are shown on the left;
    $Y'_{\delta_1}$ and  $Y'_{\delta_2}$ are shown on the right.}
    \label{fig:different-lengths}
  \end{figure}

  This means that no hexagon corresponding to an element of $Y'_{\delta_1}$
  contains a hexagon corresponding to an element of $Y'_{\delta_2}$.  That is,
  the closure of $Y'_{\delta_1}$ in $X$ does not intersect $Y'_{\delta_2}$.
  Since the closure of a finite union is the union of the closures, the closure
  of
  \begin{equation*}
    (\X \cap U_1 s_1 I) \bigcup \left(\bigcup_{\delta > m} Y'_\delta\right)
  \end{equation*}
  in $X$ does not intersect $Y'_\delta$ for $\delta \le m$.  Which means that
  this set is closed in $Z$, and we can apply
  Theorem~\ref{thm:Triviality-For-Nice-Stratification} to this case.
  
  \subsubsection{$x = \e^{(d,e,f)} s_1s_2s_1$, with $f \le e \le d$ and $e \ge j$}
  Let
  \begin{equation}
    x = \e^{(d, e, f)}s_1s_2s_1 = \begin{pmatrix}
        0 & 0 & \e^d  \\
        0 & \e^e & 0 \\
         \e^f & 0 & 0
      \end{pmatrix}
  \end{equation}
  where $f \le e \le d$ and $d + e + f = 0$.  Let $\nu = (i,j,k)$ with $i > j >
  k$ and $i + j + k = 0$.  Assume that $e \ge j$.  We will show that in this
  case the intersection $\X \cap U_1wI$ is nonempty only when $w = s_2s_1$ or $w
  =s_1s_2s_1$, and that if $e \neq i$ only the $w = s_1 s_2 s_1$ intersection
  is nonempty.  Then we will show that if $e \neq i$
  Theorem~\ref{thm:Triviality-For-One-Intersection} applies and otherwise
  Theorem~\ref{thm:Triviality-For-Nice-Stratification} applies.

  Since $f \le e \le d$, by
  Theorem~\ref{thm:DeterminingIwahoriDoubleCosetForElement} the entry in the
  third row and first column of $h$ must have valuation $f$ and the valuation
  of the determinant of the bottom-left $2 \times 2$ minor must be $f + e$.
  That means that $w$ can only be one of $s_1s_2$, $s_2s_1$, and $s_1s_2s_1$,
  since for all other values of $w$ the bottom-left entry is 0.

  If $w = s_1s_2$, then
  \begin{equation*}
    h = \begin{pmatrix}
      \e^j & \gamma & 0 \\
      0 & \e^k & 0 \\
      \alpha & \beta & \e^i
    \end{pmatrix}.
  \end{equation*}
  For $h$ to be in $IxI$, we must must have $\val(\alpha) = f$ and
  $\val(\e^k\alpha) = f + e$.  This means that $e = k$.  But by assumption, $e
  \ge j > k$, so in this case $\X \cap U_1 s_1 s_2 I = \emptyset$.

  If $w = s_2s_1$, then
  \begin{equation*}
    h = \begin{pmatrix}
      \e^k & 0 & 0 \\
      \beta & \e^i & \alpha \\
      \gamma & 0 & \e^j
    \end{pmatrix}.
  \end{equation*}
  For $h$ to be in $IxI$, we must have $\val(\gamma) = f$ and $\val(\e^i\gamma)
  = f + e$.  This means that $e = i$.  In all other cases, $\X \cap U_1 s_2 s_1
  I = \emptyset$.

  So for $e \ne i$ we only have nonempty intersections with $U_1 wI$
  for $w = s_1s_2s_1$.  In this case, \begin{equation*}
    h = \begin{pmatrix}
      \e^k & 0 & 0 \\
      \gamma & \e^j & 0 \\
      \beta & \alpha & \e^i
    \end{pmatrix}.
  \end{equation*}
  By Theorem~\ref{thm:DeterminingIwahoriDoubleCosetForElement}, the necessary
  conditions for $h$ to be in $IxI$ include
  \begin{align}
    \label{eq:2-d-max-f-min-id-val-beta}
    \val(\beta) &= f \\
    \label{eq:2-d-max-f-min-id-val-gamma}
    \val(\gamma) &\ge f \\
    \label{eq:2-d-max-f-min-id-k-greater-f}
    \val(\e^k) &\ge f \implies k \ge f \\
    \label{eq:2-d-max-f-min-id-i-less-d}
    \val(\e^{j+k}) &\ge f + e \implies j + k \ge f + e \implies i \le d \\
    \label{eq:2-d-max-f-min-id-val-minor}
    \val(\beta\e^j - \alpha\gamma) &= f + e \\
    \label{eq:2-d-max-f-min-id-val-gamma2}
    \val(\e^i \gamma) &\ge f + e \\
    \label{eq:2-d-max-f-min-id-val-alpha}
    \val(\e^k\alpha) &\ge f + e.
  \end{align}
  We will use these to determine the valuations of $a$, $b$, $c$, and $b-ac$.

  First, note that by (\ref{eq:2-d-max-f-min-id-val-alpha}) and
  (\ref{eq:alpha})
  \begin{equation}
    \label{eq:2-d-max-f-min-id-val-a}
    \val(a) \ge f + e -  j - k.
  \end{equation}

  By (\ref{eq:2-d-max-f-min-id-val-gamma}),
  (\ref{eq:2-d-max-f-min-id-val-gamma2}), and (\ref{eq:gamma}),
  \begin{equation}
    \label{eq:2-d-max-f-min-id-val-c}
    \val(c) \ge \max(f-k, (f-k) + (e -i)).
  \end{equation}

  \begin{lemma}
    \label{lemma:2-d-max-f-min-id-val-ac}
    $\val(ac) \ge f - k$.
  \end{lemma}
  \begin{proof}
    Assume $\val(ac) < f - k$.  Then by (\ref{eq:gamma}), $\val(a\gamma) < f$.
    Now by (\ref{eq:2-d-max-f-min-id-val-beta}) and (\ref{eq:beta}) and because
    $i > k$, we must have
    \begin{equation*}
      \val(\e^k b) = \val(a\gamma) < f.
    \end{equation*}
    This would mean, because $i > j$ and $j \le e$, that
    \begin{align*}
      \val(\e^{i+j}\s(b) - \e^{j+k}b) &= j + \val(a\gamma) \\
      &< e + f.
    \end{align*}
    But then, by (\ref{eq:minor}), to satisfy
    (\ref{eq:2-d-max-f-min-id-val-minor}) we must have
    \begin{align*}
      \val(\e^i\s(a)\gamma) &= \val(\e^{i+j}\s(b) - \e^{j+k}b) \\
      & = j + \val(a\gamma).
    \end{align*}
    This requires $i = j$, which is impossible.  So $\val(ac) \ge f-k$.
  \end{proof}

  \begin{lemma}
    \label{lemma:2-d-max-f-min-id-e-g-j-val-ac}
    If $e > j$, then $\val(ac) = f - k$.
  \end{lemma}
  \begin{proof}
    Assume $\val(ac) > f - k$.  Then by (\ref{eq:2-d-max-f-min-id-val-beta})
    and because $i > k$, we must have $\val(\e^k b) = f$.  In that case
    \begin{align*}
      \val(\e^{i+j}\s(b) - \e^{k+j}b) &= \val(\e^{k+j}b) \\
      &= f + j.
    \end{align*}
    At the same time, because $j > k$,
    \begin{align*}
      \val(\e^i \s(a) \gamma)) &= i + k + \val(ac) \\
      &> i + k + f - k \\
      &= f + i \\
      &> f + j.
    \end{align*}
    Since the difference of these two terms if $\beta\e^j - \alpha\gamma$, we
    have $\val(\beta\e^j - \alpha\gamma) = f + j < f + e$, contradicting
    (\ref{eq:2-d-max-f-min-id-val-minor}).  Therefore we must have $\val(ac)
    \le f - k$.  Since we already know $\val(ac) \ge f-k$ by
    Lemma~\ref{lemma:2-d-max-f-min-id-val-ac}, we conclude that $\val(ac) = f -
    k$.
  \end{proof}

  \begin{lemma}
    \label{lemma:2-d-max-f-min-id-e-g-j-val-a}
    If $e > j$ then $0 \ge \val(a)$.
  \end{lemma}
  \begin{proof}
    By Lemma~\ref{lemma:2-d-max-f-min-id-e-g-j-val-ac}, $\val(ac) = f - k$.
    But by (\ref{eq:2-d-max-f-min-id-val-c}), $\val(c) \ge f - k$.  Therefore
    $\val(a) \le 0$.
  \end{proof}

  \begin{lemma}
    \label{lemma:2-d-max-f-min-id-e-g-j-val-c}
    If $e > j$ then $0 > \val(c)$.
  \end{lemma}
  \begin{proof}
    By Lemma~\ref{lemma:2-d-max-f-min-id-e-g-j-val-ac}, $\val(ac) = f - k$.
    But by (\ref{eq:2-d-max-f-min-id-val-a}), $\val(a) \ge (f - k) + (e-j)$.
    Therefore $\val(c) \le j - e < 0$.
  \end{proof}

  Now we have three possibilities: $e > i$, $e < i$, and $e = i$.

  \subsubsubsection{The case $e > i$}

  \begin{lemma}
    \label{lemma:2-d-max-f-min-id-e-g-i-val-b}
    If $e > i$, then $\val(b) = f + 2i$.  This means that $\val(b) < 0$,
    $\val(b) < \val(a)$, and $\val(b-ac) = f-k < \val(c) < 0$.
  \end{lemma}
  \begin{proof}
    Since $e > i > j$, we know that $\val(ac) = f - k$ by
    Lemma~\ref{lemma:2-d-max-f-min-id-e-g-j-val-ac}.  So
    \begin{align*}
      \val(\e^i\s(a)\gamma) &= i + f - k + k \\
      &= f + i \\
      &< f + e.
    \end{align*}
    But then, by (\ref{eq:minor}), to satisfy
    (\ref{eq:2-d-max-f-min-id-val-minor}) we must have
    \begin{align*}
      \val(\e^{i+j}\s(b) - \e^{k+j}b) &= f + i \\
      \val(\e^{k+j}b) &= f + i\\
      \val(b) &= f + i - j - k \\
              &= f + 2i.
    \end{align*}
    Now \begin{align*}
      \val(b) &= f + i - j - k \\
              &< f + e - j - k.
    \end{align*}
    By (\ref{eq:2-d-max-f-min-id-val-a}), $f + e - j - k \le \val(a)$, so
    $\val(b) < \val(a)$.  Since $\val(a) \le 0$ by
    Lemma~\ref{lemma:2-d-max-f-min-id-e-g-j-val-a}, we see that $\val(b) < 0$.
    
    Since $\val(b) = f - k + (i - j) > f-k$ and by
    Lemma~\ref{lemma:2-d-max-f-min-id-e-g-j-val-ac} $\val(ac) = f - k$, we see
    that $\val(b-ac) = f-k$.  By (\ref{eq:2-d-max-f-min-id-val-c}), $f-k < (f -
    k) + (e-i) \le \val(c)$, and by
    Lemma~\ref{lemma:2-d-max-f-min-id-e-g-j-val-c} $\val(c) < 0$.
  \end{proof}

  Since $\val(b) < 0$ and $\val(b-ac) < \val(c) < 0$, the conditions of
  Section~\ref{sec:hexagons-u-s1-s2-s1} are satisfied.  By
  Lemma~\ref{lemma:2-d-max-f-min-id-e-g-j-val-a}, $\val(a) \le 0$.  This means
  that the hexagon for $g$ has the vertices
  \begin{equation*}
    \heximagessqueezedcarefully{2.0em}{-0.4em}
       {\begin{pmatrix}
           0 & 0 & 1 \\
           0 & 1 & 0 \\
           1 & 0 & 0 
       \end{pmatrix}}
       {\begin{pmatrix}
           0 & \e^{\val(a)} & 0 \\
           0 & 0 & \e^{-\val(a)}  \\
           1 & 0 & 0
       \end{pmatrix}}
       {\begin{pmatrix}
           0 & 0 & 1 \\
           \e^{\val(c)}  & 0 & 0 \\
           0 & \e^{-\val(c)}  & 0
       \end{pmatrix}}
       {\begin{pmatrix}
           0 & \e^{f-k - \val(c)} & 0 \\
           \e^{\val(c)} & 0 & 0 \\
           0 & 0 & \e^{k-f} 
       \end{pmatrix}}
       {\begin{pmatrix}
           \e^{f+2i}  & 0 & 0 \\
           0 & 0 & \e^{-\val(a)} \\
           0 & \e^{\val(a) - f - 2i} & 0  
       \end{pmatrix}}
       {\begin{pmatrix}
           \e^{f+2i} & 0 & 0 \\
           0 & \e^{-k-2i} & 0 \\
           0 & 0 & \e^{k-f} 
       \end{pmatrix}}
  \end{equation*}
  when $e > i$.  Note that all of these hexagons share two opposite vertices:
  the top and bottom one.  Therefore
  Theorem~\ref{thm:Triviality-For-One-Intersection} applies.

  \subsubsubsection{The case $e < i$}
  \begin{lemma}
    \label{lemma:2-d-max-f-min-id-e-l-i-val-b}
    If $e < i$, then $\val(b) = f + e - j - k = i - d$.  This means that
    $\val(b) \le 0$ and $\val(b) \le \val(a)$.
  \end{lemma}
  \begin{proof}
    By Lemma~\ref{lemma:2-d-max-f-min-id-val-ac}, $\val(ac) \ge f -k$.
    Therefore, by (\ref{eq:gamma}),
    \begin{align*}
      \val(\e^i\s(a)\gamma) &\ge f + i \\
       &> f + e.
    \end{align*}
    Since $i > k$ and (\ref{eq:minor}) holds, to satisfy
    (\ref{eq:2-d-max-f-min-id-val-minor}) we must have
    \begin{align*}
      \val(\e^{j+k} b) &= f + e \\
      \val(b) &= f + e - j - k \\
      &= i - d.
    \end{align*}
    By (\ref{eq:2-d-max-f-min-id-i-less-d}), $\val(b) \le 0$.  By
    (\ref{eq:2-d-max-f-min-id-val-a}), $\val(b) \le \val(a)$.
  \end{proof}

  \begin{lemma}
    If $e < i$, then $\val(b-ac) = f -k$.  This means that $\val(b-ac) \le 0$
    and $\val(b-ac) \le \val(c)$.
  \end{lemma}
  \begin{proof}
    By Lemma~\ref{lemma:2-d-max-f-min-id-e-l-i-val-b}, $\val(b) = f + e - j -
    k$.  Hence
    \begin{align*}
      \val(\e^i \s(b)) &= f + e - j - k + i \\
      &> f + e - j \\
      &\ge f.
    \end{align*}
    By Lemma~\ref{lemma:2-d-max-f-min-id-val-ac}, $\val(ac) \ge f-k$.  So
    \begin{align*}
      \val(\e^j a \s(c)) &\ge f - k + j \\
      &> f.
    \end{align*}
    For (\ref{eq:2-d-max-f-min-id-val-beta}) to be satisfied, we must have
    \begin{align*}
      \val(-\e^k b + \e^k ac) &= f\\
      \val(b - ac) &= f-k.
    \end{align*}
    By (\ref{eq:2-d-max-f-min-id-k-greater-f}), $\val(b-ac) \le 0$.  By
    (\ref{eq:2-d-max-f-min-id-val-c}), $\val(b-ac) \le \val(c)$.
  \end{proof}

  Since $\val(b) \le 0$, $\val(b-ac) \le \val(c)$, and $\val(b-ac) \le 0$, the
  conditions of Section~\ref{sec:hexagons-u-s1-s2-s1} are satisfied.  We have
  $\val(b) \le \val(a)$, so the hexagons we get only depend on whether the
  valuations of $a$ and $c$ are positive.

  Now we have four cases:
  \begin{caselist}
    \item $\val(c) \le 0$ and $\val(a) \le 0$.  In this case, the hexagon for
      $g$ is
      \begin{equation*}
	\heximages
	    {\begin{pmatrix}
		0 & 0 & 1 \\
		0 & 1 & 0 \\
		1 & 0 & 0 
	    \end{pmatrix}}
	    {\begin{pmatrix}
		0 & \e^{\val(a)} & 0 \\
		0 & 0 & \e^{-\val(a)}  \\
		1 & 0 & 0
	    \end{pmatrix}}
	    {\begin{pmatrix}
		0 & 0 & 1 \\
		\e^{\val(c)}  & 0 & 0 \\
		0 & \e^{-\val(c)}  & 0
	    \end{pmatrix}}
	    {\begin{pmatrix}
		0 & \e^{f-k - \val(c)} & 0 \\
		\e^{\val(c)} & 0 & 0 \\
		0 & 0 & \e^{k-f} 
	    \end{pmatrix}}
	    {\begin{pmatrix}
		\e^{i-d}  & 0 & 0 \\
		0 & 0 & \e^{-\val(a)} \\
		0 & \e^{\val(a) - i + d} & 0  
	    \end{pmatrix}}
	    {\begin{pmatrix}
		\e^{i-d} & 0 & 0 \\
		0 & \e^{j-e} & 0 \\
		0 & 0 & \e^{k-f} 
	    \end{pmatrix}}.
      \end{equation*}
    \item $\val(c) \le 0$ and $\val(a) > 0$.  In this case, the hexagon for $g$
      is
      \begin{equation*}
	\heximages
	    {\begin{pmatrix}
		0 & 0 & 1 \\
		0 & 1 & 0 \\
		1 & 0 & 0 
	    \end{pmatrix}}
	    {\begin{pmatrix}
		0 & 0 & 1 \\
		0 & 1 & 0 \\
		1 & 0 & 0 
	    \end{pmatrix}}
	    {\begin{pmatrix}
		0 & 0 & 1 \\
		\e^{\val(c)}  & 0 & 0 \\
		0 & \e^{-\val(c)}  & 0
	    \end{pmatrix}}
	    {\begin{pmatrix}
		0 & \e^{f-k - \val(c)} & 0 \\
		\e^{\val(c)} & 0 & 0 \\
		0 & 0 & \e^{k-f} 
	    \end{pmatrix}}
	    {\begin{pmatrix}
		\e^{i-d}  & 0 & 0 \\
		0 & 1 & 0 \\
		0 & 0 & \e^{d - i}  
	    \end{pmatrix}}
	    {\begin{pmatrix}
		\e^{i-d} & 0 & 0 \\
		0 & \e^{j-e} & 0 \\
		0 & 0 & \e^{k-f} 
	    \end{pmatrix}}.
      \end{equation*}
    Note that since $\val(a) > 0$, by
    Lemma~\ref{lemma:2-d-max-f-min-id-e-g-j-val-a} we must have $e = j$, which
    means that $d-k = i-f$.  Therefore the bottom vertex and the bottom-right
    vertex coincide in this case.
    \item $\val(c) > 0$ and $\val(a) \le 0$.  In this case, the hexagon for $g$
      is
      \begin{equation*}
	\heximages
	    {\begin{pmatrix}
		0 & 0 & 1 \\
		0 & 1 & 0 \\
		1 & 0 & 0 
	    \end{pmatrix}}
	    {\begin{pmatrix}
		0 & \e^{\val(a)} & 0 \\
		0 & 0 & \e^{-\val(a)}  \\
		1 & 0 & 0
	    \end{pmatrix}}
	    {\begin{pmatrix}
		0 & 0 & 1 \\
		0 & 1 & 0 \\
		1 & 0 & 0 
	    \end{pmatrix}}
	    {\begin{pmatrix}
		\e^{f-k} & 0 & 0 \\
		0 & 1 & 0 \\
		0 & 0 & \e^{k-f} 
	    \end{pmatrix}}
	    {\begin{pmatrix}
		\e^{i-d}  & 0 & 0 \\
		0 & 0 & \e^{-\val(a)} \\
		0 & \e^{\val(a) - i + d} & 0  
	    \end{pmatrix}}
	    {\begin{pmatrix}
		\e^{i-d} & 0 & 0 \\
		0 & \e^{j-e} & 0 \\
		0 & 0 & \e^{k-f} 
	    \end{pmatrix}}.
      \end{equation*}
      Note that since $\val(c) > 0$, by
      Lemma~\ref{lemma:2-d-max-f-min-id-e-g-j-val-c} we must have $e = j$,
      which means that $d-k = i-f$.  Therefore the bottom vertex and the
      bottom-left vertex coincide in this case.
    \item $\val(c) > 0$ and $\val(a) > 0$.  In this case, the hexagon for $g$
      is
      \begin{equation*}
	\heximages
	    {\begin{pmatrix}
		0 & 0 & 1 \\
		0 & 1 & 0 \\
		1 & 0 & 0 
	    \end{pmatrix}}
	    {\begin{pmatrix}
		0 & 0 & 1 \\
		0 & 1 & 0 \\
		1 & 0 & 0 
	    \end{pmatrix}}
	    {\begin{pmatrix}
		0 & 0 & 1 \\
		0 & 1 & 0 \\
		1 & 0 & 0 
	    \end{pmatrix}}
	    {\begin{pmatrix}
		\e^{f-k} & 0 & 0 \\
		0 & 1 & 0 \\
		0 & 0 & \e^{k-f} 
	    \end{pmatrix}}
	    {\begin{pmatrix}
		\e^{i-d} & 0 & 0 \\
		0 & 1 & 0 \\
		0 & 0 & \e^{d-i} 
	    \end{pmatrix}}
	    {\begin{pmatrix}
		\e^{i-d} & 0 & 0 \\
		0 & \e^{j-e} & 0 \\
		0 & 0 & \e^{k-f} 
	    \end{pmatrix}}.
      \end{equation*}
      Note that since $\val(c) > 0$, by
      Lemma~\ref{lemma:2-d-max-f-min-id-e-g-j-val-c} we must have $e = j$,
      which means that $d-k = i-f$.  Therefore the bottom vertex, the
      bottom-left vertex, and the bottom-right vertex all coincide in this
      case.
  \end{caselist}
  Note that all of these hexagons share two opposite vertices: the top and
  bottom one.  Therefore Theorem~\ref{thm:Triviality-For-One-Intersection}
  applies.

  \subsubsubsection{The case $e = i$}

  In this case, there are intersections with both $U_1 s_1s_2s_1 I$ and $U_1
  s_2s_1 I$.  If $w = s_2s_1$, then
  \begin{equation*}
    h = \begin{pmatrix}
      \e^k & 0 & 0 \\
      \beta & \e^i & \alpha \\
      \gamma & 0 & \e^j
    \end{pmatrix}.
  \end{equation*}
  By Theorem~\ref{thm:DeterminingIwahoriDoubleCosetForElement}, the necessary
  conditions for $h$ to be in $IxI$ include
  \begin{align}
    \label{eq:2-d-max-f-min-s1-val-gamma}
    \val(\gamma) &= f \\
    \label{eq:2-d-max-f-min-s1-val-beta}
    \val(\beta) &\ge f \\
    \label{eq:2-d-max-f-min-s1-minor-equality}
    \val(\e^i\gamma) &= f + e \\
    \label{eq:2-d-max-f-min-s1-k-greater-f}
    \val(\e^{j+k}) &\ge f + e \implies j + k \ge f + i \implies k > f \\
    \label{eq:2-d-max-f-min-s1-val-minor}
    \val(\beta\e^j - \alpha\gamma) &\ge f + e
  \end{align}
  We will use these to determine the valuations of $a$, $b$, $c$, and $b-ac$.

  From (\ref{eq:2-d-max-f-min-s1-val-gamma}), (\ref{eq:gamma}), and
  (\ref{eq:2-d-max-f-min-s1-k-greater-f}) we see that
  \begin{equation}
    \label{eq:2-d-max-f-min-s1-val-c}
    \val(c) = f - k < 0.
  \end{equation}

  From (\ref{eq:2-d-max-f-min-s1-val-minor}) and (\ref{eq:minor}) we see that
  \begin{align}
    \val(\e^{i+j}\s(b) - \e^{j+k}b - \e^i\s(a)\gamma) &\ge f + e \nonumber\\
    \noalign{\noindent and since $e = i$}
    \label{eq:2-d-max-f-min-s1-trailing-term}
    \val(\e^j\s(b) - \e^{k+j-i}b - \s(a)\gamma) &\ge f. \\
    \noalign{\noindent At the same time, by (\ref{eq:2-d-max-f-min-s1-val-beta}) and
    (\ref{eq:altbeta}),}
    \val(\e^i\s(b) - \e^k b - a\gamma) &\ge f. \nonumber
  \end{align}
  Now $\val(\s(a)\gamma) = \val(a\gamma)$.  If this valuation were less than
  $f$, then, because $j > k$ and $i > k$  we would have to have $\val(\e^k b) =
  \val(\e^{k + j - i}b) = \val(a\gamma)$ to cancel the terms of valuation lower
  than $f$.  But this would require $j = i$, which is impossible.  Therefore,
  \begin{align}
    \val(a\gamma) &\ge f \nonumber \\
    \noalign{\noindent and by (\ref{eq:2-d-max-f-min-s1-val-gamma})}
    \label{eq:2-d-max-f-min-s1-val-a}
    \val(a) &\ge 0.
  \end{align}

  This means that $\val(\s(a)\gamma) \ge f$, so to satisfy
  (\ref{eq:2-d-max-f-min-s1-trailing-term}) we must have
  \begin{align}
    \val(\e^{k + j - i} b) &\ge f \nonumber \\
    \val(b) &\ge f + i - j - k \nonumber \\
            &> f - k \nonumber  \\
            &= \val(c).
  \end{align}

  Since $\val(c) < 0$, $\val(a) \ge 0$, and $\val(b) > \val(c)$, the conditions
  of Section~\ref{sec:hexagons-u-s2-s1} are satisfied.  Note that there are no
  restrictions on how $\val(b)$ compares with $0$.  Therefore we see that if
  $\val(b) > 0$ the hexagon for $g$ has the vertices
  \begin{equation}
    \label{eq:2-d-max-f-min-s1-hex1}
    \heximages
	{\begin{pmatrix}
            0 & 1 & 0 \\
            0 & 0 & 1 \\
            1 & 0 & 0
	\end{pmatrix}}
	{\begin{pmatrix}
            0 & 1 & 0 \\
            0 & 0 & 1 \\
            1 & 0 & 0
	\end{pmatrix}}
	{\begin{pmatrix}
            0 & 1 & 0 \\
            \e^{f-k} & 0 & 0  \\
            0 & 0 & \e^{k-f}
	\end{pmatrix}}
	{\begin{pmatrix}
            0 & 1 & 0 \\
            \e^{f-k} & 0 & 0  \\
            0 & 0 & \e^{k-f}
	\end{pmatrix}}
	{\begin{pmatrix}
            0 & 1 & 0 \\
            0 & 0 & 1 \\
            1 & 0 & 0
	\end{pmatrix}}
	{\begin{pmatrix}
            0 & 1 & 0 \\
            \e^{f-k} & 0 & 0  \\
            0 & 0 & \e^{k-f}
	\end{pmatrix}}
  \end{equation}
  and if $\val(b) \le 0$ it has the vertices
  \begin{equation}
    \label{eq:2-d-max-f-min-s1-hex2}
    \heximages
	{\begin{pmatrix}
            0 & 1 & 0 \\
            0 & 0 & 1 \\
            1 & 0 & 0
	\end{pmatrix}}
	{\begin{pmatrix}
            0 & 1 & 0 \\
            0 & 0 & 1 \\
            1 & 0 & 0
	\end{pmatrix}}
	{\begin{pmatrix}
            0 & 1 & 0 \\
            \e^{f-k} & 0 & 0  \\
            0 & 0 & \e^{k-f}
	\end{pmatrix}}
	{\begin{pmatrix}
            0 & 1 & 0 \\
            \e^{f-k} & 0 & 0  \\
            0 & 0 & \e^{k-f}
	\end{pmatrix}}
	{\begin{pmatrix}
            \e^{\val(b)} & 0 & 0 \\
            0 & 0 & 1 \\
            0 & \e^{-\val(b)} & 0
	\end{pmatrix}}
	{\begin{pmatrix}
            \e^{\val(b)} & 0 & 0 \\
            0 & \e^{f-k-\val(b)} & 0  \\
            0 & 0 & \e^{k-f}
	\end{pmatrix}}.
  \end{equation}

  Now we look at $w = s_1s_2s_1$.  Since $e = i > j$, by
  Lemma~\ref{lemma:2-d-max-f-min-id-e-g-j-val-ac} we know that $\val(ac) = f -
  k$.  This means that $\val(\e^i\s(a)\gamma) = f + i = f + e$.  So for
  (\ref{eq:2-d-max-f-min-id-val-minor}) to be satisfied, we must have,
  by~(\ref{eq:minor}),
  \begin{align*}
    \val(\e^{j+k}b) &\ge f + e \\
    \val(b) &\ge f + e - j - k \\
            &= f  - k + (i - j) \\
            &> f-k.
  \end{align*}
  Since $\val(ac) = f -k$, we conclude that
  \begin{equation}
    \label{eq:2-d-max-f-min-id-e-eq-i-val-b-minus-ac}
    \val(b-ac) = f - k.
  \end{equation}
  Then by (\ref{eq:2-d-max-f-min-id-val-c}) and
  Lemma~\ref{lemma:2-d-max-f-min-id-e-g-j-val-c},
  \begin{equation}
    \label{eq:2-d-max-f-min-id-e-eq-i-val-b-minus-ac2}
    \val(b-ac) \le \val(c) < 0.
  \end{equation}



  By Lemma~\ref{lemma:2-d-max-f-min-id-e-g-j-val-a}, $\val(a) \le 0$.  By
  (\ref{eq:2-d-max-f-min-id-e-eq-i-val-b-minus-ac2}), $\val(b-ac) \le \val(c) <
  0$.  So the conditions of Section~\ref{sec:hexagons-u-s1-s2-s1} are
  satisfied.  $f - k = \val(b-ac)$ and $\val(c) = f - k - \val(a)$.
  The hexagon we get depends on how $\val(b)$ compares to $\val(a)$.  If
  $\val(b) \le \val(a)$, the hexagon for $g$ has the vertices
  \begin{equation*}
    \heximagessqueezed{4em}
    	{\begin{pmatrix}
	    0 & 0 & 1 \\
	    0 & 1 & 0 \\
	    1 & 0 & 0 
	\end{pmatrix}}
	{\begin{pmatrix}
	    0 & \e^{\val(a)} & 0 \\
	    0 & 0 & \e^{-\val(a)}  \\
	    1 & 0 & 0
	\end{pmatrix}}
	{\begin{pmatrix}
	    0 & 0 & 1 \\
	    \e^{f - k - \val(a)}  & 0 & 0 \\
	    0 & \e^{\val(a) - f + k}  & 0
	\end{pmatrix}}
	{\begin{pmatrix}
	    0 & \e^{\val(a)} & 0 \\
	    \e^{f-k-\val(a)} & 0 & 0 \\
	    0 & 0 & \e^{k-f} 
	\end{pmatrix}}
	{\begin{pmatrix}
	    \e^{\val(b)}  & 0 & 0 \\
	    0 & 0 & \e^{-\val(a)} \\
	    0 & \e^{\val(a) - \val(b)} & 0  
	\end{pmatrix}}
	{\begin{pmatrix}
	    \e^{\val(b)} & 0 & 0 \\
	    0 & \e^{f-k-\val(b)} & 0 \\
	    0 & 0 & \e^{k-f} 
	\end{pmatrix}}.
  \end{equation*}
  and if $\val(b) > \val(a)$ it has the vertices
  \begin{equation*}
    \heximages
    	{\begin{pmatrix}
	    0 & 0 & 1 \\
	    0 & 1 & 0 \\
	    1 & 0 & 0 
	\end{pmatrix}}
	{\begin{pmatrix}
	    0 & \e^{\val(a)} & 0 \\
	    0 & 0 & \e^{-\val(a)}  \\
	    1 & 0 & 0
	\end{pmatrix}}
	{\begin{pmatrix}
	    0 & 0 & 1 \\
	    \e^{f - k - \val(a)}  & 0 & 0 \\
	    0 & \e^{\val(a) - f + k}  & 0
	\end{pmatrix}}
	{\begin{pmatrix}
	    0 & \e^{\val(a)} & 0 \\
	    \e^{f-k-\val(a)} & 0 & 0 \\
	    0 & 0 & \e^{k-f} 
	\end{pmatrix}}
	{\begin{pmatrix}
	    0 & \e^{\val(a)} & 0 \\
	    0 & 0 & \e^{-\val(a)}  \\
	    1 & 0 & 0
	\end{pmatrix}}
	{\begin{pmatrix}
	    0 & \e^{\val(a)} & 0 \\
	    \e^{f-k-\val(a)} & 0 & 0 \\
	    0 & 0 & \e^{k-f} 
	\end{pmatrix}}.
  \end{equation*}

  We want to apply Theorem~\ref{thm:Triviality-For-Nice-Stratification} to this
  case.  In the notation of that theorem, $w_0 = s_2s_1$ and $w_1 = s_1s_2s_1$.
  The subsets $Y_\delta$ correspond to subsets defined by $\val(a) = \delta$.
  We let $\mu_\delta$ be $(-\val(a), \val(a), 0) = (-\delta, \delta, 0)$, so
  that
  \begin{equation*}
    \e^{\mu_\delta} = \begin{pmatrix}
      \e^{-\val(a)} & 0 & 0 \\
      0 & \e^{\val(a)} & 0 \\
      0 & 0 & 1
    \end{pmatrix}.
  \end{equation*}
  and let $Y'_\delta = \e^{\mu_\delta} Y_\delta$.  Then when $\val(b) \le
  \val(a)$ the hexagon corresponding to elements of $Y'_\delta$ is
  \begin{equation*}
    \heximagessqueezed{6.5em}
    	{\begin{pmatrix}
	    0 & 0 & \e^{-\val(a)} \\
	    0 & \e^{\val(a)} & 0 \\
	    1 & 0 & 0 
	\end{pmatrix}}
	{\begin{pmatrix}
            \hphantom{\e^{\val(b)-\val(a)}} &
	    \hphantom{\e^{\val(a) - \val(b)}} & \\[-1em]
	    0 & 1 & 0 \\
	    0 & 0 & 1  \\
	    1 & 0 & 0
	\end{pmatrix}}
	{\begin{pmatrix}
	    0 & 0 & \e^{-\val(a)} \\
	    \e^{f - k}  & 0 & 0 \\
	    0 & \e^{\val(a) - f + k}  & 0
	\end{pmatrix}}
	{\begin{pmatrix}
            & \hphantom{\e^{k - d + \val(a)}} &
	    \hphantom{\e^{-\val(a)}} \\[-1em]
	    0 & 1 & 0 \\
	    \e^{f-k} & 0 & 0 \\
	    0 & 0 & \e^{k-f} 
	\end{pmatrix}}
	{\begin{pmatrix}
	    \e^{\val(b)-\val(a)}  & 0 & 0 \\
	    0 & 0 & 1 \\
	    0 & \e^{\val(a) - \val(b)} & 0  
	\end{pmatrix}}
	{\begin{pmatrix}
	    \e^{\val(b)-\val(a)} & 0 & 0 \\
	    0 & \e^{f-k-\val(b) + \val(a)} & 0 \\
	    0 & 0 & \e^{k-f} 
	\end{pmatrix}}
  \end{equation*}
  and when $\val(b) > \val(a)$ it is
  \begin{equation*}
    \heximages
    	{\begin{pmatrix}
	    0 & 0 & \e^{-\val(a)} \\
	    0 & \e^{\val(a)} & 0 \\
	    1 & 0 & 0 
	\end{pmatrix}}
	{\begin{pmatrix}
	    0 & 1 & 0 \\
	    0 & 0 & 1  \\
	    1 & 0 & 0
	\end{pmatrix}}
	{\begin{pmatrix}
	    0 & 0 & \e^{-\val(a)} \\
	    \e^{f - k}  & 0 & 0 \\
	    0 & \e^{\val(a) - f + k}  & 0
	\end{pmatrix}}
	{\begin{pmatrix}
	    0 & 1 & 0 \\
	    \e^{f-k} & 0 & 0 \\
	    0 & 0 & \e^{k-f} 
	\end{pmatrix}}
	{\begin{pmatrix}
	    0 & 1 & 0 \\
	    0 & 0 & 1  \\
	    1 & 0 & 0
	\end{pmatrix}}
	{\begin{pmatrix}
	    0 & 1 & 0 \\
	    \e^{f-k} & 0 & 0 \\
	    0 & 0 & \e^{k-f} 
	\end{pmatrix}}
  \end{equation*}
  Comparing these to the hexagons in (\ref{eq:2-d-max-f-min-s1-hex1}) and
  (\ref{eq:2-d-max-f-min-s1-hex2}) we see that all four sets of hexagons share
  two opposite vertices: the top right and bottom left one.  Further, for $(\X
  \cap U_1 s_2s_1 I)$ the top vertex coincides with the top-right vertex and
  the top-left vertex coincides with the bottom-left vertex.  By an argument
  similar to the one we gave for the $e = i$ case in
  Section~\ref{sec:first-long-two-cycle}, we can apply
  Theorem~\ref{thm:Triviality-For-Nice-Stratification} to this case.  

  \subsubsection{$x = \e^{(d, e, f)}$, with $f \le e \le d$}
  Let
  \begin{equation}
    x = \e^{(d, e, f)} = \begin{pmatrix}
        \e^d  & 0 & 0\\
        0 & \e^e & 0 \\
        0 & 0 & \e^f
      \end{pmatrix}
  \end{equation}
  where $f \le e \le d$.  Let $\nu = (i,j,k)$ with $i > j > k$ and $i + j + k =
  0$.  We will show that in this case the intersection $\X \cap U_1 wI$ is
  nonempty only when $w = 1$, and that
  Theorem~\ref{thm:Triviality-For-One-Intersection} applies.

  By (\ref{eq:image-matrix}),
  \begin{equation*}
    h = w^{-1}\begin{pmatrix}
      \e^i & \alpha & \beta \\
      0 & \e^j & \gamma \\
      0 & 0 & \e^k    
    \end{pmatrix}w.
  \end{equation*}
  In all cases, the bottom-right entry of $h$ is one of $\e^i$, $\e^j$, and
  $\e^k$.  So in order to have $h \in IxI$, we must have $f = i$, $f=j$, or
  $f=k$.  But because $i + j + k = 0$ and $i > j > k$, we must have $i > 0$.
  At the same time, because $d + e + f = 0$ and $f \le e \le d$, we must have
  $f \le 0$.  So $f \neq i$.

  If $f = j$, then we must have $w = s_2$ or $w = s_1 s_2$.  If $w = s_2$, then
  \begin{equation*}
    h = \begin{pmatrix}
      \e^i & \beta & \alpha \\
      0 & \e^k & 0 \\
      0 & \gamma & \e^j    
    \end{pmatrix}.
  \end{equation*}
  To have $h \in IxI$, we must have $j + k = e + f$.  Since $f = j$, this means
  $e = k$.  But $f \le e$ and $k < j$, so this is impossible.

  If $w = s_1 s_2$, then
  \begin{equation*}
    h = \begin{pmatrix}
      \e^k & 0 & 0 \\
      \beta& \e^i & \alpha \\
      \gamma & 0 & \e^j    
    \end{pmatrix}.
  \end{equation*}
  To have $h \in IxI$, we must have $j + i = e + f$, so that $k = d$.  But $d
  \ge 0$ and $k < 0$, so this is impossible.
  
  If $f = k$, then we must have $w = 1$ or $w = s_1$. If $w = s_1$, then
  \begin{equation*}
    h = \begin{pmatrix}
      \e^j & 0 & \gamma \\
      \alpha & \e^i & \beta \\
      0 & 0 & \e^k    
    \end{pmatrix}.
  \end{equation*}
  To have $h \in IxI$ we must have $k + i = f + e$, so that $i = e$.  Then
  $d = j$, and we must have $d < e$, since $j < i$.  But, by assumption, $d \ge
  e$, so this is impossible.

  Thus, we must have $w = 1$.  In this case, $f = k$ and 
  \begin{equation*}
    h = \begin{pmatrix}
      \e^i & \alpha & \beta \\
      0 & \e^j & \gamma \\
      0 & 0 & \e^k    
    \end{pmatrix}.
  \end{equation*}
  By Theorem~\ref{thm:DeterminingIwahoriDoubleCosetForElement}, the necessary
  conditions for $h$ to be in $IxI$ include
  \begin{align}
    \label{eq:1-equalities}
    j + k &= f + e \implies e = j,\quad i = d \\
    \label{eq:1-val-alpha}
    \val(\alpha\e^k) &> f + e \\
    \label{eq:1-val-gamma}
    \val(\gamma) &> f \\
    \label{eq:1-val-beta}
    \val(\beta) &> f.
  \end{align}

  From (\ref{eq:1-val-alpha}), (\ref{eq:alpha}), and (\ref{eq:1-equalities}) we
  see that
  \begin{align}
    \val(a) + j + k &> j + k \nonumber \\
    \label{eq:1-val-a}
    \val(a) &> 0.
  \end{align}

  From (\ref{eq:1-val-gamma}), (\ref{eq:gamma}), and (\ref{eq:1-equalities}) we
  see that
  \begin{align}
    \val(c) + k &> k \nonumber \\
    \label{eq:1-val-c}
    \val(c) &> 0.
  \end{align}

  From (\ref{eq:1-val-beta}), (\ref{eq:beta}), and (\ref{eq:1-equalities}) we
  see that
  \begin{align}
    \val(\e^i\s(b) - \e^k(b) - a\gamma) &> k. \nonumber \\
    \noalign{\noindent But by (\ref{eq:1-val-a}), (\ref{eq:1-val-gamma}), and
    (\ref{eq:1-equalities}), $\val(a\gamma) > k$.  So}
    \val(\e^i\s(b) - \e^k(b)) &> k \nonumber \\
    \noalign{\noindent and since $i > k$ we must have}
    \val(b) &> 0.
  \end{align}
  In this case, $a$, $b$, and $c$ can all be eliminated from $g$ using the
  right-action of $I$, leaving
  \begin{equation*}
    g = \begin{pmatrix}
      1 & 0 & 0 \\
      0 & 1 & 0 \\
      0 & 0 & 1    
    \end{pmatrix}.
  \end{equation*}
  The corresponding hexagon has all the vertices at the same point.  So
  Theorem~\ref{thm:Triviality-For-One-Intersection} applies.
  
  \bibliographystyle{hplainyr}
  \bibliography{reductive-groups}
    
\end{document}